\documentclass[a4paper,12pt,reqno]{article}
\usepackage[english]{babel}
\usepackage{amsmath,amsfonts,amssymb,amsthm,bbm}
\usepackage{graphics,epsfig,psfrag} 
\usepackage{paralist,subfigure,multirow,float}
\usepackage{hyperref} 
\usepackage[active]{srcltx} 
\usepackage[dvipsnames]{xcolor}
\usepackage{soul} 
\usepackage[latin1]{inputenc}
\usepackage[OT1]{fontenc}

\newtheorem{theorem}{Theorem}[section]
\newtheorem{corollary}[theorem]{Corollary}
\newtheorem{lemma}[theorem]{Lemma}
\newtheorem{proposition}[theorem]{Proposition}
\newtheorem{remark}[theorem]{Remark}

\newtheorem{hypothesis}[theorem]{Hypothesis}
\newtheorem{definition}[theorem]{Definition}

\theoremstyle{definition}

 \usepackage{autonum} 

\DeclareMathOperator{\dist}{dist}
\DeclareMathOperator{\dv}{div}

\DeclareMathOperator\Tr{Tr}

\DeclareMathOperator\inter{int}

\def\N{\mathbb{N}}
\def\Z{\mathbb{Z}}

\def\R{\mathbb{R}}

\let\O=\Omega
\let\e=\varepsilon
\def\epsilon{\varepsilon}
\let\vp=\varphi

\let\t=\tilde
\let\ol=\overline
\let\ul=\underline
\let\.=\cdot
\let\0=\emptyset

\let\mc=\mathcal

\def\1{\mathbbm{1}}
\def\Sph{{\mathbb{S}}^{N-1}}

\def\hat{\widehat}
\def\tilde{\widetilde}

\def\seq#1{(#1_n)_{n\in\N}}

\def\as{\quad\text{as }\;}

\newenvironment{formula}[1]{\begin{equation}\label{#1}}{\end{equation}\noindent}

\numberwithin{equation}{section}
\newcommand{\be}{\begin{equation}}
\newcommand{\ee}{\end{equation}}
\newcommand{\baa}{\begin{array}}
\newcommand{\eaa}{\end{array}}
\newcommand{\ba}{\begin{eqnarray}}
\newcommand{\ea}{\end{eqnarray}}
\def\Fi#1{\begin{formula}{#1}}
\def\Ff{\end{formula}\noindent}
\def\BS{\color{Bittersweet}}

\def\PTF{pulsating traveling front}

\def\ais{asymptotic invasion shape}
\def\W{\mc{W}}
\def\V{\mc{V}}

\begin{document}
\date{}
\title{\bf{Reaction-diffusion equations in periodic media: convergence to pulsating fronts\ \thanks{This work has received funding from Excellence Initiative of Aix-Marseille Universit\'e~-~A*MIDEX, a French ``Investissements d'Avenir'' programme, from the French ANR ReaCh (ANR-23-CE40-0023-02) project, and from the European Union -- Next Generation EU, PRIN project 2022W58BJ5 ``PDEs and optimal control methods in mean field games, population dynamics and multi-agent models''. The second author is grateful to the hospitality of Universit\`a degli Studi di Roma La Sapienza, where part of this work was done. The first author is supported by the fundamental research funds for the central universities and the National Natural Science Foundation of China (No. 12471201).}}}
\author{}
\author{Hongjun Guo$^{\hbox{\small{ a}}}$, Fran\c cois Hamel$^{\hbox{\small{ b}}}$ and Luca Rossi$^{\hbox{\small{ c}}}$\\
\\
\footnotesize{$^{\hbox{a }}$School of Mathematical Sciences, Key Laboratory of Intelligent Computing and Applications }\\
\footnotesize{(Ministry of Education), Institute for Advanced Study, Tongji University, Shanghai, China}\\
\footnotesize{$^{\hbox{b }}$Aix Marseille Univ, CNRS, I2M, Marseille, France}\\
\footnotesize{$^{\hbox{c }}$SAPIENZA Univ Roma, Istituto ``G.~Castelnuovo'', Roma, Italy}}
\maketitle
	
\begin{abstract}
\noindent{This paper is concerned with reaction-diffusion-advection equations in spatially periodic media. Under an assumption of weak stability of the constant states~$0$ and $1$, and of existence of pulsating traveling fronts connecting them, we show that fronts' profiles appear, along sequences of times and points, in the large-time dynamics of the solutions of the Cauchy problem, whether their initial supports are bounded or unbounded. The types of equations that fit into our assumptions are the combustion and the bistable ones. We also show a generalized Freidlin-G\"artner formula and other geometrical properties of the asymptotic invasion shapes, or spreading sets, of invading solutions, and we relate these sets to the upper level sets of the solutions.}
\end{abstract}
	
\tableofcontents


\section{Introduction and main results}\label{intro}

We consider the equation
\Fi{general}
\partial_t u=\dv(A(x)\nabla u)+q(x)\.\nabla u + f(x,u),\quad t>0,\ x\in\R^N,
\Ff
and we are interested in the large-time dynamics of bounded solutions $u$, and especially in the convergence to some front profiles along level sets. Throughout the paper, the operator $\nabla$ stands for the gradient operator with respect to the variables $x\in\R^N$. Regarding the terms 
in~\eqref{general}, we assume throughout the paper that
$$f(x,0)=f(x,1)=0\ \hbox{ for all }x\in\R^N,$$
and that $A(x)$, $q(x)$, $f(x,s)$ are $(1,\cdots,1)$-periodic with respect to $x\in\R^N$ (for short, we say periodic in the sequel), i.e.
$$\forall\,h\in\Z^N,\qquad A(\.+h)\equiv A,\qquad q(\.+h)\equiv q,\qquad f(\.+h,\.)\equiv f.$$
The matrix field $x\mapsto A(x)=(a_{ij}(x))_{1\le i,j\le N}$ ranges in the set of symmetric positive definite matrices, and it is of class $C^{1,\alpha}(\R^N)$ for some $\alpha>0$. The vector field $x\mapsto q(x)=(q_i(x))_{1\le i\le N}\in\R^N$ is of class $C^1(\R^N)$, together with
$$\mathrm{div}\,q=0\hbox{ in }\R^N\ \hbox{ and }\ \int_{(0,1)^N}\!\!q(x)dx=0.$$
Lastly, $f:(x,s)\mapsto f(x,s)$ is of class $C^{0,\alpha}(\R^N\times[0,1])$, $\partial_sf$ exists and is of class $C^{0,\alpha}(\R^N\times[0,1])$.\footnote{These assumptions on the regularity of $f$ are a bit stronger than some other assumptions in the literature, but they will be used in the proofs of some intermediate results (Proposition~\ref{pro:U} and Lemmata~\ref{lem:interior}-\ref{lem:exterior}) which constitute some steps of the main results.} We consider solutions $u$ of~\eqref{general} with measurable initial conditions $u_0:\R^N\to[0,1]$, in the sense that $u(t,\cdot)\to u_0$ as $t\to0^+$ in $L^1_{loc}(\R^N)$. Each solution $u$ is understood as the unique bounded solution with initial condition $u_0$, and it is such~that
$$0\le u(t,x)\le1\ \hbox{ for all $t>0$ and $x\in\R^N$},$$
with strict inequalities if $u_0$ is not almost everywhere equal to $0$ or to $1$, by the strong maximum principle. Furthermore, from parabolic regularity theory, $u$ is then a classical solution in~$(0,+\infty)\times\R^N$, namely it is of class $C^1$ in $t$ and $C^2$ in~$x$, and $\displaystyle\|\partial_tu\|_{C^{0,\alpha/2}([1,+\infty)\times\R^N)}+\max_{1\le i\le N}\|\partial_{x_i}u\|_{C^{0,1}([1,+\infty)\times\R^N)}+\max_{1\le i,j\le N}\|\partial_{x_ix_j}u\|_{C^{0,\alpha/2}([1,+\infty)\times\R^N)}$ is bounded by a universal constant depending only on $A$, $q$ and $f$, with $\|v\|_{C^{0,\beta}(E)}:=\|v\|_{L^\infty(E)}+\sup_{\zeta\neq\zeta'\in E}|v(\zeta)-v(\zeta')|/|\zeta-\zeta'|^\beta$ for $E\subset\R\times\R^N$ and $\beta>0$.

Propagation phenomena are an important aspect of reaction-diffusion equations. A prominent question is to know whether and how one of the steady states $0$ or $1$ invades the other one. To describe such invasions, fronts propagating with average speed and connecting these two steady states play an essential role. We show in this paper that the fronts are attractive, in the sense that they appear as spatially local limits, along sequences of points, in the large-time behavior of solutions of~\eqref{general}, even if the initial conditions $u_0$ are far from any front. This question of convergence to front profiles is a  widely open problem in the theory of reaction-diffusion equations in arbitrary dimension~$N$ or for general initial conditions.

We first highlight in Section~\ref{sec11} one main result (Theorem~\ref{th1}) on the specific important case of compactly supported initial conditions. That result is actually part of more general results which are detailed in Section~\ref{sec12}. The main hypotheses used in these main results are commented in Sections~\ref{sec21}-\ref{sec23}. The conclusions and examples of applications of the main results are discussed in Section~\ref{sec24}. The proofs of the main results are carried out in Section~\ref{sec3}, while Section~\ref{sec4} is devoted to the proof of auxiliary results related to the main hypotheses and properties of pulsating traveling fronts and asymptotic invasion shapes.

Throughout the paper, ``$|\ |$'' and ``$\ \cdot\ $'' denote the Euclidean norm and inner product in $\R^N$. For $x\in\R^N$ and $r>0$, $B_r(x)$ is the open Euclidean ball with center~$x$ and radius~$r$, $B_r:=B_r(0)$, and $\Sph:=\{e\in\R^N:|e|=1\}$ is the unit Euclidean sphere of~$\R^N$. We set $\R^+:=(0,+\infty)$ and, for $x\in\R^N\setminus\{0\}$,
$$\hat{x}:=\frac{x}{|x|}.$$


\subsection{A main result with compactly supported initial conditions}\label{sec11}

For the main results, we assume that the steady states $0$ and $1$ are weakly stable for $f$, in the following sense:

\begin{hypothesis}\label{hyp:comb-bi}
There exists $\delta\in(0,1/2)$ such that, for every $x\in\R^N$, the function $s\mapsto f(x,s)$ is non-increasing in $[0, \delta]$ and decreasing in $[1-\delta, 1]$.
\end{hypothesis}

Next, we will assume an implicit hypothesis which is not expressed in terms of the coefficients of the equation. Namely, we assume that~\eqref{general} has pulsating traveling fronts describing the invasion of the steady state~$0$ by the steady state~$1$, in every direction:

\begin{hypothesis}\label{hyp:c*>0}
For any $e\in\Sph$, there exists a \PTF\ connecting~$1$ to~$0$ in the direction $e$ with a positive speed, in the sense that there are $c^*(e)>0$ and an entire $($defined for all $t\in\R$$)$ classical solution $\phi_e:\R\times\R^N\to(0,1)$ of~\eqref{general} of the type
\Fi{ptf1}
\phi_e(t,x)=U_e(x,x\cdot e-c^*(e)t),
\Ff
where the function $U_e(x,z)$ is periodic in the $x$-variable and satisfies
\Fi{ptf2}
\lim_{z\to-\infty}U_e(x,z)=1,\ \ \lim_{z\to+\infty}U_e(x,z)=0,\ \hbox{uniformly with respect to $x\in\R^N$}.
\Ff
\end{hypothesis}

Hypotheses~\ref{hyp:comb-bi}-\ref{hyp:c*>0} are related to the {\em bistability} of the equation~\eqref{general},  in a broad sense. 
They are commented in Section~\ref{sec21}. Under Hypotheses~\ref{hyp:comb-bi}-\ref{hyp:c*>0}, the speed~$c^*(e)$ in any direction~$e$ is unique, $\phi_e$ is unique up to shift in $t$, the map $e\mapsto c^*(e)$ is continuous, and the map $e\mapsto\phi_e$ is continuous in some sense after normalization, see Proposition~\ref{pro:U}. We point out here that, since $c^*(e)>0$ under Hypothesis~\ref{hyp:c*>0}, the change of variable $(t,x)\mapsto(x,x\cdot e-c^*(e)t)$ is a $C^\infty$ diffeomorphism from $\R^{N+1}$ onto itself, hence $U_e$ is continuous, as is $\phi_e$. Furthermore, since $0<\phi_e<1$ in $\R\times\R^N$ from the strong parabolic maximum principle, there holds $0<U_e(x,z)<1$ for all $(x,z)\in\R^N\times\R$, and
\be\label{Ue}
\forall\,a\le b\in\R,\ \ \ 0<\min_{(x,z)\in\R^N\times[a,b]}U_e(x,z)\le\max_{(x,z)\in\R^N\times[a,b]}U_e(x,z)<1,
\ee
by continuity in $(x,z)$ and periodicity with respect to $x$. From parabolic estimates and the $C^{0,\alpha}(\R^N\times[0,1])$ regularity of $f$ and $\partial_sf$, the functions $\phi_e$ and $\partial_t\phi_e$ together with their first-order, respectively second-order, derivatives with respect to~$t$, respectively with respect to~$x$, are continuous and bounded in $\R\times\R^N$. Hence, the functions $U_e$, $\nabla U_e$, $\partial_zU_e$, $\nabla\partial_zU_e$ and $\partial^2_{zz}U_e$ are at least continuous and bounded in~${\R^N\times\R}$, where $\nabla:=\nabla_x$.

A fundamental open problem in the theory of reaction-diffusion equations~\eqref{general} in dimensions $N\ge2$ is the local-in-space convergence at large time to profiles of pulsating fronts along some sequences of points. To address this question, we consider the {\em $\O$-limit set} of a solution $u:\R^+\times\R^N\to[0,1]$ of~\eqref{general}, which is defined by
\be\label{defOmegau}
\Omega(u):=\left\{\begin{array}{rcl}\!\!\psi\in L^\infty(\R^N) & \!\!\!\!:\!\!\!\! & u(t_n,x_n+\.)\displaystyle\mathop{\longrightarrow}_{n\to+\infty}\psi\text{ in $L^\infty_{loc}(\R^N)$,}\\
& \!\!\!\! & \text{for some sequences $(t_n)_{n\in\N}$ in $\R^+$ diverging to $+\infty$}\\
& \!\!\!\! & \text{and $(x_n)_{n\in\N}$ in $\R^N$}\end{array}\!\!\right\}.
\ee
This set is composed by all possible asymptotic profiles of $u$. It can also be written as
\be\label{eqOmegau}
\Omega(u)=\bigcup_{e\in\Sph}\Omega_e(u),
\ee
where
$$\Omega_e(u):=\left\{\begin{array}{rcl}\!\!\psi\in L^\infty(\R^N) & \!\!\!\!:\!\!\!\! & u(t_n,x_n+\.)\displaystyle\mathop{\longrightarrow}_{n\to+\infty}\psi\text{ in $L^\infty_{loc}(\R^N)$ as $n\to+\infty$,}\\
& \!\!\!\! & \text{for some sequences $(t_n)_{n\in\N}$ in $\R^+$ diverging to $+\infty$}\\
& \!\!\!\! & \text{and $(x_n)_{n\in\N}$ in $\R^N\setminus\{0\}$ such that $\displaystyle\hat{x_n}\mathop{\longrightarrow}_{n\to+\infty}e$}\end{array}\!\!\right\}$$
is called the $\O$-limit set of $u$ in the direction $e$.\footnote{The inclusion~``$\supset$'' in~\eqref{eqOmegau} holds by definition. The reverse one follows from the fact that, for any sequence $\seq{x}$ in $\R^N$, one can extract a subsequence such that $(\hat{x_{n_k}})_{k\in\N}$ converges to some $e\in\Sph$, unless all but finitely many $x_n$'s are equal to the $0$ vector in $\R^N$. However, the latter type of sequences can be omitted in the definition~\eqref{defOmegau} of~$\O(u)$ and replaced by $(e/n)_{n\in\N^*}$, for given $e\in\Sph$, since, for any sequence $(t_n)_{n\in\N}$ with $\liminf_{n\to+\infty}t_n>0$, one has, by parabolic estimates, $u(t_n,\.)-u(t_n,e/n+\.)\to0$ as $n\to+\infty$ in $L^\infty(\R^N)$.} Notice that the sets $\Omega(u)$ and~$\Omega_e(u)$ are all non-empty and included in $C^2(\R^N)$ from standard parabolic estimates.

In the sequel, we say that a function $u$ is an {\it invading solution} of~\eqref{general} if
\be\label{definvading}
u(t,\cdot)\to1\ \hbox{as $t\to+\infty$ locally uniformly in $\R^N$}.
\ee
If $u$ is an invading solution, then $u(t_n,x_n+\cdot)\to1$ as $n\to+\infty$ locally uniformly in~$\R^N$, for any sequence $(t_n)_{n\in\N}$ diverging to $+\infty$ and for any {\it bounded} sequence $(x_n)_{n\in\N}$ in $\R^N$. However, the sequences $(x_n)_{n\in\N}$ in~\eqref{defOmegau} can also be {\it unbounded}, and~\eqref{definvading} is therefore not enough at all to describe the way the state $1$ invades the whole space~$\R^N$. We want to know more about the local profiles appearing at large time in the regions where $u$ is away from $0$ and $1$.

An important case of our more general results concerns compactly supported initial data, in which case we show that any directional limit set~$\Omega_e(u)$ contains some profiles of pulsating traveling fronts connecting $1$ to $0$:

\begin{theorem}\label{th1}
Assume that Hypotheses~$\ref{hyp:comb-bi}$-$\ref{hyp:c*>0}$ hold. Let $u$ be an invading solution to~\eqref{general} with a compactly supported initial condition $u_0$. Then, for any $e\in\Sph$, there is $\nu\in\Sph$ such that
\be\label{enu}
\Omega_e(u)\supset\big\{\phi_\nu(t,\cdot+y):(t,y)\in\R\times\R^N\big\}\ \cup\ \{0,1\}.
\ee
\end{theorem}

More details about the directions $\nu$ in~\eqref{enu} will be given in the other main results in Section~\ref{sec12}, together with the extension to non-compactly supported initial data, while some comments on the conclusion will be provided in Section~\ref{sec24}.


\subsection{More general main results}\label{sec12}

In order to describe more precisely the shape of the possible transition between the states $1$ and $0$ at large time for the solutions $u$ of~\eqref{general} with initial conditions having arbitrary support, we use the following notion of {\em \ais}:

\begin{definition}\label{def:W}
We say that a solution $u$ to~\eqref{general} admits an {\em \ais} $\W\subset\R^N$ if $\W$ coincides  with the interior of its closure and satisfies 
\Fi{ass-cpt}
\begin{cases}
\displaystyle\lim_{t\to+\infty}\,\Big(\min_{x\in C}u(t, tx)\Big)=1 & \text{for any non-empty compact set }C\subset\mc{W},\vspace{3pt}\\
\displaystyle\lim_{t\to+\infty}\,\Big(\max_{x\in C}u(t,tx)\Big)=0 & \text{for any non-empty compact set }C\subset\R^N\setminus\ol{\mc{W}}\,.
\end{cases}
\Ff
\end{definition}

Examples and counter-examples of existence of the asymptotic invasion shape $\W$, and general properties of $\W$, are discussed in Section~\ref{sec23}.

Under Hypotheses~\ref{hyp:comb-bi}-\ref{hyp:c*>0}, it follows from~\cite{R1}\footnote{This is a consequence of~\cite[Theorems 1.4-1.5]{R1} which hold for general space-time dependent equations, not necessarily periodic, and provide an \ais\ that reduces to the set $\W_0$ in \eqref{formuleFG} thanks to Hypothesis~\ref{hyp:c*>0}.}
that all the invading solutions of~\eqref{general} with compactly supported initial conditions $u_0$ admit an identical \ais\ $\W$, which we call $\W_0$ (as a reminder of the fact that the initial condition is compactly supported) and which is given by the so-called Freidlin-G\"artner formula:
\Fi{formuleFG}\baa{rcl}
\W_0 & = & \displaystyle\bigcap_{\xi\in\Sph}\big\{x\in\R^N:x\cdot\xi<c^*(\xi)\big\}\\
& = & \displaystyle\big\{re:e\in\Sph,\ r\in[0,w(e))\big\}\ \hbox{ with }\ w(e):=\min_{\xi\in\Sph,\,\xi\cdot e>0}\ \frac{c^*(\xi)}{\xi\cdot e}>0.\eaa
\Ff
From this formula and the continuity of the map $\xi\mapsto c^*(\xi)$ from $\Sph$ to $\R^+=(0,+\infty)$ (by Proposition~\ref{pro:U} below), the set $\W_0$ is open, bounded, convex, it contains a neighborhood of the origin, the minimum in~\eqref{formuleFG} is truly reached, the map $e\mapsto w(e)$ is continuous from $\Sph$ to $\R^+$ as well, and $\W_0$ is the {\em Wulff shape} of the envelop set of the speeds $c^*(\xi)$ of pulsating fronts. 
The set $\W_0$ can also be viewed as the large-time limit of the rescaled upper level sets of the invading solutions $u$ with compactly supported initial supports (by Proposition~\ref{pro:level} below). 
Furthermore, at any point~${z}=w(\hat{z})\hat{z}$ of~$\partial\W_0$ where~$\W_0$ has a tangent plane, the expression~\eqref{formuleFG} for $w(\hat{z})$ is uniquely minimized by $\xi=\nu$, where~$\nu$ is the exterior unit normal of~$\W_0$ at~${z}$,\footnote{\label{f2}Indeed, consider any minimizer $\xi$ in the expression~\eqref{formuleFG} for $w(\hat{z})$, that is, $\xi\in\Sph$ with $\xi\cdot\hat{z}>0$ and $w(\hat{z})=c^*(\xi)/(\xi\cdot\hat{z})$. The set $\W_0$ then lies by~\eqref{formuleFG} in the open half-space $E\!:=\!\{x\!\in\!\R^N\!:\!x\cdot\xi\!<\!c^*(\xi)\}$, while $z=w(\hat{z})\hat{z}$ belongs to both $\partial\W_0$ and $\partial E$. Since $\W_0$ has a tangent plane at $z$ with exterior unit normal $\nu$ and since $\W_0$ must lie on one side of this tangent plane by convexity, it follows that this tangent plane must be $\partial E$ and that $\xi=\nu$.} that is,
$${z}\.\nu=c^*(\nu).$$
Therefore, we refer to the above identity as the ``regular Freidlin-G\"artner formula''.

The next theorem shows, first, that such formula holds for any solution $u$ at every regular point $z$ of~$\partial\W$ without requiring the boundedness of the support of $u_0$, and second, that the $\O$-limit set of $u$ in the direction $\hat z$ contains profiles of pulsating traveling fronts. Regular here means that $\W$ satisfies both the interior and the exterior ball conditions at the boundary point ${z}$, which assert respectively that there exists an open Euclidean ball~$B$ (resp.~$B'$) such that $B\subset\W$ (resp.~$B'\subset\R^N\setminus\W$) and ${z}\in\partial B$ (resp.~${z}\in\partial B'$). The simultaneous validity of the two conditions imply in particular that~$\W$ has a tangent plane at~$z$, 
with the (unique) exterior unit normal to~$\W$ at~${z}$ being given by $\hat{z-y}=\hat{y'-z}$, 
where $y$ and $y'$ are the centers of $B$ and $B'$ respectively.

\begin{theorem}\label{th2}
Assume that Hypotheses~$\ref{hyp:comb-bi}$-$\ref{hyp:c*>0}$ hold. Let $u$ be a solution to~\eqref{general} admitting an \ais\ $\W$. Assume that there exists a point ${z}\in\partial\W$ at which $\W$ satisfies the interior and exterior ball conditions. Let  $\nu$ be the exterior unit normal to~$\W$ at~${z}$. Then the following properties hold:
\begin{enumerate}[$(i)$\;]
\item ${z}\.\nu=c^*(\nu)$\,;
\item $\displaystyle\O_{\hat{z}}(u)\supset\big\{\phi_\nu(t,\cdot+y):(t,y)\in\R\times\R^N\big\}\cup\{0,1\}$.\end{enumerate}
\end{theorem}

It turns out that $\W$ always contains a neighborhood of the origin, because it is open and contains the origin (by Proposition~\ref{pro:Wchar}). Hence, the vector $\hat{z}$ is well defined for any~${z}\in\partial\W$. The validity of the interior and exterior ball conditions at $z$ in Theorem~\ref{th2} holds in particular at any point ${z}$ where $\partial\W$ is twice-differentiable.
Theorem $\ref{th2}$ could be summarized by saying that: $(i)$~the regular Freidlin-G\"artner formula holds at every regular point of $\partial\W$; $(ii)$~in the direction of any regular point, the solution approaches the profiles of pulsating traveling fronts along sequences of times.

We now extend Theorem $\ref{th2}$ beyond regular points by introducing the notion of generalized normal set, defined as the set of limits of exterior unit normals of regular approximated~points:
$$\mc{V}({z})\!:=\!\left\{\!\!\baa{rcl}
\displaystyle\nu\!=\!\lim_{n\to+\infty}\nu_n & \!\!\!:\!\!\! & \displaystyle{z}_n\in\partial\W,\ \lim_{n\to+\infty}{z}_n={z},\\
& \!\!\!\!\!\! & \hbox{$\W$ satisfies the interior and exterior ball conditions at ${z}_n$,}\\
& \!\!\!\!\!\! & \hbox{$\nu_n$ is the exterior unit normal at $z_n$}\eaa\!\!\!\right\}\!.$$
Theorem~\ref{th2} then leads to the following result:

\begin{corollary}\label{cor1}
Assume that Hypotheses~$\ref{hyp:comb-bi}$-$\ref{hyp:c*>0}$ hold. Let $u$ be a solution to~\eqref{general} admitting an \ais\ $\W$. Let ${z}$ be any point of $\partial\W$ such that ${\mc{V}({z})\neq\emptyset}$. Then the following properties hold:
\begin{enumerate}[$(i)$\;]
\item ${z}\.\nu=c^*(\nu)$ \ for all $\nu\in\mc{V}({z})$;
\item $\displaystyle\O_{\hat{z}}(u)\supset\big\{\phi_\nu(t,\cdot+y):(t,y)\in\R\times\R^N,\,\nu\in\mc{V}({z})\big\}
\cup\{0,1\}$.
\end{enumerate}
\end{corollary}

Property $(i)$ above is a {\em generalized} regular Freidlin-G\"artner formula, in which the exterior unit normal $\nu$ is replaced by the whole set of generalized exterior unit normals $\mc{V}$. We show in Proposition~\ref{pro:Wchar} below that the set $\partial\W$ is Lipschitz-regular, but unfortunately this is not enough to guarantee the existence of the generalized exterior unit normals at the boundary. However, since a sufficient condition for the validity of the interior and exterior ball conditions at a point is the second order differentiability of the boundary at that point, it follows from the Alexandrov theorem that Corollary~\ref{cor1} can be applied in particular at the points ${z}\in\partial\W$ around which $\W$ is convex, that is, for which there is $r>0$ such that $\W\cap B_r({z})$ is~convex.

The arguments employed in the proof Theorem \ref{th2} lead to the following important result in the case where $u_0$ is compactly supported. We recall that, in such case, it follows from~\cite{R1} that the solution $u$~admits an \ais\ $\W_0$, given by the Freidlin-G\"artner formula~\eqref{formuleFG}.

\begin{corollary}\label{cor2}
Assume that Hypotheses~$\ref{hyp:comb-bi}$-$\ref{hyp:c*>0}$ hold. Let $u$ be an invading solution to~\eqref{general} with a compactly supported initial condition $u_0$. Its \ais~$\W_0$ therefore exists and is given by~\eqref{formuleFG}. Then, for any $e\in\Sph$, one has 
\be\label{minimizer}\Omega_{e}(u)
\supset
\left\{
\begin{array}{rcl}
\phi_\xi(t,\cdot+y) & : & (t,y)\in\R\times\R^N,\ \xi\text{ is a minimizer}\\
& & \text{in the expression of $w(e)$ in \eqref{formuleFG}}
\end{array}\right\}
\ \cup\ \{0,1\}.
\ee
Lastly, if $\partial\W_0$ is everywhere differentiable, then
\be\label{differentiable}
\Omega(u)\supset\big\{\phi_e(t,\cdot+y):(t,y)\in\R\times\R^N,\ e\in\Sph\big\}\ \cup\ \{0,1\}.
\ee
\end{corollary}

Corollary~\ref{cor2} contains Theorem \ref{th1}, and additionally provides the directions $\nu$ in~\eqref{enu}. It asserts that the convergence toward a \PTF, along some sequence of times, occurs in any direction $e\in\Sph$, even in those where $\W_0$ exhibits a corner, i.e.~where the Freidlin-G\"artner formula~\eqref{formuleFG} admits multiple minimizers. Corners may actually exist,  as shown in~\cite[Corollary~6.2]{GR2}, hence the $\O$-limit set in the corresponding directions contains the profiles of several distinct pulsating~fronts. 

It is natural to wonder whether the equality actually holds in Corollary~\ref{cor2} when~$\partial\W_0$ is everywhere differentiable, i.e.
$$\Omega(u)=\big\{\phi_e(t,\cdot+y):(t,y)\in\R\times\R^N,\ e\in\Sph\big\}\ \cup\ \{0,1\}.$$
Such conclusion is known to hold for homogeneous equations (where $\W_0$ is a ball), as follows from~\cite{AW,BH2,U2}, but in the periodic setting it has only been derived in dimension $N=1$, for $f$ either of the bistable type, as follows from~\cite{DHZ2,DGM,X2,X3}, or of the Fisher-KPP type~\eqref{fkppg} below, see~\cite{HNRR2}. To our knowledge, ours are the first convergence-to-front results for periodic equations in dimensions higher than one.

More comments on the conclusions and applications of the main results, and counter-examples when the hypotheses are not fulfilled, are given in Section~\ref{sec24}.

\paragraph{Some additional notations.} For a subset $E\subset\R^N$, we call
$$\1_E:\R^N\to\{0,1\}$$
the indicator function of the set $E$, that is, $\1_E(x)=1$ if $x\in E$ and $\1_E(x)=0$ if $x\in\R^N\setminus E$. For $x\in\R^N$ and $A\subset\R^N$, we define $\dist(x,A):=\inf\big\{|x-y|:y\in A\}$ (with $\dist(x,\emptyset)=+\infty$) and, for any two subsets $A,B\subset\R^N$, we define the Hausdorff distance between $A$ and $B$ as
$$d_{\mc{H}}(A,B):=\max\Big(\sup_{x\in A}\dist(x,B),\,\sup_{y\in B}\dist(y,A)\Big)$$
with the conventions $d_{\mc{H}}(A,\emptyset)=d_{\mc{H}}(\emptyset,A)=+\infty$ if $A\neq\emptyset$ and $d_{\mc{H}}(\emptyset,\emptyset)=0$.


\section{Comments on the hypotheses and the conclusions, and complementary results}\label{sec2}

We first discuss in Section~\ref{sec21} some conditions under which Hypotheses~\ref{hyp:comb-bi} and~\ref{hyp:c*>0} are fulfilled, and the invasion property entailed by them. Section~\ref{sec23} lists examples and counter-examples of existence of asymptotic invasion shapes $\W$, and some general properties of $\W$ in an even more general framework, without Hypotheses~\ref{hyp:comb-bi}-\ref{hyp:c*>0}. Lastly, we give in Section~\ref{sec24} some applications of the main results, and we explain the necessity of the main hypotheses.


\subsection{Weak bistability, pulsating fronts, and invasion property}\label{sec21}

\paragraph{Sufficient conditions in the homogeneous case.}

In the homogeneous case with $A\equiv I_N$ (the identity matrix of size $N\times N$), $q\equiv0$ and $f(x,s)=f(s)$,~\eqref{general} reduces to
\Fi{homo}
\partial_t u=\Delta u+f(u),\quad t>0,\ x\in\R^N.
\Ff
If~$f$ is of ignition type, that~is,
\Fi{ignition}
\exists\,\alpha\in(0,1),\ \ f=0\hbox{ in $[0,\alpha]$ and $f>0$ in $(\alpha,1)$},
\Ff
or if $f$ is of bistable type, that is,
\Fi{bistable}
\exists\,\alpha\in(0,1),\ \ f<0\hbox{ in }(0,\alpha)\hbox{ and }f>0\hbox{ in }(\alpha,1),
\Ff
with $\int_0^1f\!>\!0$, then~\eqref{homo} admits planar fronts connecting $1$ to $0$ with positive and unique speed~$c$, namely solutions $u(t,x)=\varphi(x\cdot e-ct)$, with $e\in\Sph$ and $\varphi:\R\to(0,1)$ solving
\Fi{eqvarphi}
\varphi''+c\varphi'+f(\varphi)=0\hbox{ in }\R,\ \ \varphi(-\infty)=1,\ \ \varphi(+\infty)=0,\ \ c>0,
\Ff
see~\cite{AW,FM}. Such fronts are pulsating traveling fronts too, with the same speed $c$, hence Hypothesis~\ref{hyp:c*>0} holds. Furthermore, if $f$ is non-increasing in a right neighborhood of~$0$ and decreasing in a left neighborhood of $1$ (for instance if $f'(1)<0$ for~\eqref{ignition}, or $f'(0)<0$, $f'(1)<0$ for~\eqref{bistable}), then Hypothesis~\ref{hyp:comb-bi} holds as well, and any pulsating traveling front in the direction $e$ is a planar front by~\cite[Theorem~3.1]{BH2} (or by the uniqueness property in Proposition~\ref{pro:U} below). Planar traveling fronts might also exist for multistable nonlinearities $f$, for instance if the speed of a front connecting any pair of consecutive stable states is smaller than the speeds of the fronts connecting pairs of consecutive stable states above them, see~\cite{FM}.

We point out that Hypothesis~\ref{hyp:c*>0} can of course be fulfilled without Hypothesis~\ref{hyp:comb-bi}: for instance, for~\eqref{homo} and $f>0$ in $(0,1)$, the set of admissible speeds in each direction~$e$ is a closed interval $[c^*(e),+\infty)\subset\R^+$~\cite{AW,F,KPP} (and $c^*(e)>0$ does not depend on $e\in\Sph$ for~\eqref{homo}). Conversely, Hypothesis~\ref{hyp:comb-bi} can be fulfilled without Hypo\-thesis~\ref{hyp:c*>0}, for instance for~\eqref{homo} and $f$ of bistable type~\eqref{bistable} with $f'(0)<0$, $f'(1)<0$ and $\int_0^1f\le0$.

\paragraph{The general periodic equation~\eqref{general}, properties of pulsating traveling fronts.}

For~\eqref{general}, under our standing assumptions, Hypothesis~\ref{hyp:c*>0} is known to hold when $f$ is of the generalized ignition case
\Fi{fignition}
\exists\,\alpha\in(0,1),\ \ \left\{\baa{l}
f=0\hbox{ in $\R^N\times[0,\alpha]$},\vspace{3pt}\\
\displaystyle\forall\,s\in(\alpha,1),\ \ 0\le\min_{\R^N}f(\cdot,s)<\max_{\R^N}f(\cdot,s),\eaa\right.
\Ff
see~\cite{BH1,X1,X22,X4}. In this case, Hypothesis~\ref{hyp:comb-bi} holds if one further assumes that $f(x,\cdot)$ is decreasing in $[1-\delta,1]$ for some $\delta>0$. Hypothesis~\ref{hyp:comb-bi} is also fulfilled if $f$ of the strong bistable type, namely if
\Fi{bistable2}
\partial_sf(x,0)<0\hbox{ and $\partial_sf(x,1)<0$ for all $x\in\R^N$}
\Ff
(remember that $\partial_sf$ is continuous in $\R^N\times[0,1]$ and periodic in $x$). We call this case ``strongly bistable" because both steady states $0$ and $1$ are linearly stable (from above and below respectively), but $f(x,\cdot)$ may have more than one zero in the interval $[0,1]$. In the case~\eqref{bistable2}, Hypothesis~\ref{hyp:c*>0} is known to hold in dimension~$1$~\cite{DHZ2,FZ,Z2} or in higher dimensions under various additional assumptions on $A$, $q$ and $f$, for instance in highly or slowly oscillating media~\cite{DLL,DS,X4} or when there does not exist any periodic steady state between $0$ and $1$~\cite{DG,D2,GR} (see also~\cite{HPS,PX,VV,X2,X3,XZ} for further references in the case of almost-homogeneous coefficients).

The main properties of the pulsating traveling fronts and speeds given in Hypothesis~\ref{hyp:c*>0} are summarized in Proposition~\ref{pro:U} below. That result is actually proved under Hypothesis~\ref{hyp:c*>0} and under an hypothesis slightly weaker than Hypothesis~\ref{hyp:comb-bi}, in which the monotonicity of $f(x,\cdot)$ in a left-neighborhood of $1$ is not anymore assumed to be strict:

\begin{hypothesis}\label{hypbis}
There exists $\delta\in(0,1/2)$ such that, for every $x\in\R^N$, the function $s\mapsto f(x,s)$ is non-increasing in $[0, \delta]$ and in $[1-\delta, 1]$.
\end{hypothesis}

\begin{proposition}\label{pro:U}
Under Hypotheses~$\ref{hyp:c*>0}$ and~$\ref{hypbis}$, let $\phi_e(t,x)=U_e(x,x\cdot e-c^*(e)t)$ be a \PTF\ connecting $1$ to $0$ in a direction $e\in\Sph$. Then its speed~$c^*(e)$ is unique, $\partial_t\phi_e(t,x)>0$ in $\R\times\R^N$, $\phi_e$ is unique up to shift in $t$, $\partial_zU_e(x,z)<0$ in~$\R^N\times\R$, $U_e$ is unique up to shift in $z$, and it satisfies, for some $\lambda_0>0$ and $C>0$,
\be\label{Ueexp}
\forall\,(x,z)\in\R^N\times\R,\quad U_e(x,z)\leq C e^{-\lambda_0z}.
\ee
Furthermore, the map $e\mapsto c^*(e)$ is continuous from $\Sph$ to $\R^+$. Lastly, for any fixed $\mu\in(0,1)$, if $\phi_e$ is normalized so that $\phi_e(0,0)=\mu$, then the maps $e\mapsto\phi_e$ and $ e\mapsto U_e$ are continuous from $\Sph$ to $L^\infty_{loc}(\R\times\R^N)$  and $L^\infty(\R^N\times\R)$, respectively.
\end{proposition}

As a matter of fact, the continuity property in Proposition~\ref{pro:U} generalizes some earlier results established for reactions of ignition type~\eqref{fignition}~\cite{AG} or strongly bistable reactions of type~\eqref{bistable2}~\cite{DLL}. The $C^2$ smoothness of the speeds $c^*(e)$ with respect to $e$ is also derived in~\cite{G} in the bistable case, under further assumptions on the coefficients.

\paragraph{The invasion property.} 

An important property entailed by Hypotheses~\ref{hyp:comb-bi}-\ref{hyp:c*>0} is the local-in-space convergence of $u(t,\cdot)$ to $1$ as $t\to+\infty$ when $u_0$ is ``large" enough on a large enough ball, namely~\eqref{invading} below. That property follows from~\cite{DR}, where the condition~${\max_{\R^N}\!f(\cdot,s)>0}$ for all $s\in[1-\delta,1)$ was used (this condition follows from Hypothesis~\ref{hyp:comb-bi}, and even $f(x,s)>0$ for all $(x,s)\in\R^N\times[1-\delta,1)$). But we show in Section~\ref{sec41} that the property~$\max_{\R^N}f(\cdot,s)>0$ for all $s\in[1-\delta,1)$ is satisfied as well when Hypothesis~\ref{hyp:comb-bi} is replaced by the weaker Hypothesis~\ref{hypbis}, still together with Hypothesis~\ref{hyp:c*>0}, and we then get the following result.

\begin{proposition}\label{pro:c>0invasion}
Under Hypotheses~$\ref{hyp:c*>0}$ and~$\ref{hypbis}$, there exist $\theta\in(0,1)$ and $\rho>0$ such that
\Fi{invading}
\Big(u_0\geq \theta\ \hbox{ in~$B_\rho$}\Big)\ \implies\ \Big(u(t,\cdot)\to1\ \hbox{ as $t\to+\infty$ locally uniformly in $\R^N$}\Big),
\Ff
where $u$ is the solution of~\eqref{general} with initial condition $u_0:\R^N\to[0,1]$.
\end{proposition}

\begin{definition}\label{definvasion}
Throughout the paper, the existence of $\theta\in(0,1)$ and $\rho>0$ for which~\eqref{invading} holds is called the invasion property for~\eqref{general}.
\end{definition}

We point out that if instead of Hypothesis~$\ref{hyp:c*>0}$, the equation~\eqref{general} admits a \PTF\ connecting~$1$ to~$0$ in a direction $e\in\Sph$ with a speed $c^*(e)\leq0$, then the invasion property necessarily fails, as a consequence of the comparison principle. By~\cite[Remark 2]{DR}, the invasion property and the periodicity of~\eqref{general} imply that invading solutions $u$ of~\eqref{general} spread at least with a positive linear speed in all directions:
\Fi{spreading}
\exists\,\gamma>0,\quad\min_{x\in\overline{B_{\gamma t}}}u(t,x)\to1\as\!t\to+\infty.
\Ff

For the homogeneous equation~\eqref{homo}, the invasion property is equivalent to 
\Fi{theta}
\hbox{there is }\theta\in(0,1)\hbox{ such that }f>0\hbox{ in }[\theta,1),\ \hbox{ and }\int_s^1\!\!f>0\hbox{ for all }s\in[0,1),
\Ff
and $\theta$ in~\eqref{theta} and in~\eqref{invading} can be defined as the same real number~\cite{DP,P4}. The invasion property is therefore satisfied in particular if $f\ge0$ in $[0,1]$ and if $f$ is positive in a left neighborhood of $1$. It also holds if~$f$ is of ignition type~\eqref{ignition} or of bistable type~\eqref{bistable} with $\int_0^1f>0$ and $\theta$ in~\eqref{invading} can be any real number in the interval $(\alpha,1)$, with $\rho>0$ large enough, depending on $\theta$, see also~\cite{AW,FM}. For~\eqref{homo}, a sufficient condition for the invasion property is that~\eqref{eqvarphi} admits a solution $(c,\varphi)$, without requiring here Hypothesis~\ref{hyp:comb-bi}, see~\cite[Proposition~1.3]{HR1}. The invasion property actually does not imply Hypotheses~\ref{hyp:comb-bi} or~\ref{hyp:c*>0}. For instance, if $f>0$ in~$(0,1)$, then the invasion property and Hypothesis~\ref{hyp:c*>0} hold for~\eqref{homo}, but Hypothesis~\ref{hyp:comb-bi} does not. Again  for~\eqref{homo}, if $f$ of multistable type, one could have that the invasion property and Hypothesis~\ref{hyp:comb-bi} hold, but Hypothesis~\ref{hyp:c*>0} does not (when instead of a single traveling front there is a stack of fronts) by~\cite{FM}.

For the general periodic equation~\eqref{general}, the invasion property holds when ${\partial_sf(\cdot,0)>0}$ in $\R^N$ and $f$ is of the generalized positive case
\Fi{fpositive}
0\le\min_{\R^N}f(\cdot,s)<\max_{\R^N}f(\cdot,s)\hbox{ for every $s\in(0,1)$},
\Ff
see~\cite{BHN,BHN1,FG,W} (here $\theta$ can be arbitrary in $(0,1)$ and $\rho>0$ can be arbitrary). Notice that Hypothesis~\ref{hyp:comb-bi} does not hold in this case. From~\cite{DR}, the invasion property also holds when $f$ is of the generalized ignition type~\eqref{fignition} with $f>0$ in $\R^N\times(\alpha,1)$. Under Hypothesis~\ref{hyp:comb-bi}, the invasion property holds with $f$ of type~\eqref{fignition}, from~\cite{BH1} and Proposition~\ref{pro:c>0invasion}. Lastly, the invasion property is fulfilled as well when $f\ge,\not\equiv0$  in ${\R^N\times[0,1]}$ and the principal part of the operator is in the non-divergence form $\Tr(A(x)D^2u)$ and~$q$ is not too large (not necessarily divergence free nor with $0$ average)~\cite{DR}. For reactions~$f$ which may take negative values, the proposition below provides a sufficient condition for the invasion property to hold.

\begin{proposition}\label{proinvasion}
Let $g:[0,1]\to\R$ be a $C^1$ function such that $g(0)=g(1)=0$ and satisfying~\eqref{theta} and $g'(1)<0$. Let $A_0=(A_{ij})_{1\le i,j\le N}$ be a constant symmetric positive definite matrix. There is $\epsilon>0$ such that, if
$$\max_{1\le i,j\le N}\|a_{ij}-A_{ij}\|_{C^1(\R^N)}+\max_{1\le i\le N}\|q_i\|_{L^\infty(\R^N)}\le\epsilon$$
and $f(x,s)\ge g(s)$ for all $(x,s)\in\R^N\times[0,1]$, then the invasion property holds for~\eqref{general}.
\end{proposition}

Proposition~\ref{proinvasion} shows that the invasion property holds in the case of almost-constant diffusion and almost-zero drift coefficients, provided that the reaction is larger than a homogeneous unbalanced bistable term. However, as for~\eqref{homo}, the invasion property does not hold in general without the condition $f(x,s)\ge g(s)$ and $g$ satisfying~\eqref{theta}. For the periodic equation~\eqref{general}, typically with a bistable structure as in~\eqref{bistable2}, heterogeneities with large amplitude or large period can also rule out the invasion property~\cite{DHZ1,DLBL,E1,E2,HFR,N,X3,XZ}.


\subsection{The \ais: properties and examples}\label{sec23}

\paragraph{General properties.}

In Definition~\ref{def:W}, the requirement that $\W$ coincides with the interior of its closure guarantees the uniqueness of the \ais, when it exists. The next result contains a characterization of $\W$. From~\cite[Remark~2]{DR}, under the invasion property, any invading solution $u$ satisfies~\eqref{spreading}, whence its \ais\ $\W$, if any, contains a ball $B_\gamma$ centered at the origin. Despite the fact that $\W$ may not be convex in general, it turns out that it contains the convex hull of the union of $B_\gamma$ with any other point $\xi\in\W$. This property leads to the Lipschitz regularity of $\W$. All these results are summarized in the following statement.

\begin{proposition}\label{pro:Wchar}
If a solution $u$ to~\eqref{general} admits an \ais\ $\W$, then
\Fi{Wint}
\W=\inter\,\big\{x\in\R^N:u(t, tx)\to1\hbox{ as }t\to+\infty\big\}.
\Ff
Moreover, under the invasion property, $\W$ is star-shaped with respect to the origin and~$\partial \W$ $($if non-empty$)$ is Lipschitz-regular. More precisely, letting $\gamma>0$ be such that~\eqref{spreading} holds for the solution of~\eqref{general} with initial datum $\theta\,\1_{B_\rho}$, where $\theta$ and $\rho$ are as in Definition~$\ref{definvasion}$, then, for any ${z}\in\partial\W$, the following estimates hold:
\Fi{inner-cone}
\forall\,\lambda\in[0,1),\quad B_{(1-\lambda)\gamma}(\lambda{z})\subset\W,
\Ff
\Fi{outer-cone}
\forall\,\lambda\in(1,+\infty),\quad B_{(\lambda-1)\gamma}(\lambda{z})\subset \R^N\setminus\W.
\Ff
\end{proposition}

The two inclusions~\eqref{inner-cone}-\eqref{outer-cone} correspond to an interior and an exterior cone condition, respectively, at any boundary point ${z}\in\partial\W$, with the opening of the cones only depending on $\gamma$ and $|{z}|$. Hence they provide an explicit estimate of the Lipschitz constant of $\partial\W$, showing that $\partial\W$ is locally uniformly Lipschitz-continuous.

It turns out that, when a solution $u$ to~\eqref{general} is initially compactly supported, its \ais\ $\W$, if any, can also be viewed as the large-time limit of the rescaled upper level sets, defined by
$$E_\lambda(t):=\big\{x\in\R^N:u(t,x)>\lambda\big\},\quad\lambda\in(0,1),\ \ t>0.$$

\begin{proposition}\label{pro:level}
If a solution $u$ to~\eqref{general} admits an asymptotic invasion shape $\W$ and if its initial datum is compactly supported, then, for every $\lambda\in(0,1)$,
\be\label{ElambdaW}
\frac{1}{t}E_\lambda(t)\mathop{\longrightarrow}_{t\to+\infty}\W
\ee
in the sense of the Hausdorff distance.\footnote{When $\W=\emptyset$, then~\eqref{ElambdaW} means that, for every $\lambda\in(0,1)$, the sets $E_\lambda(t)$ are empty for $t$ large.}
\end{proposition}

We point out that~\eqref{ElambdaW} does not hold in general if the initial datum is not compactly supported, see~\cite[Proposition~6.5]{HR1}.\footnote{\cite[Proposition~6.5]{HR1} was concerned with the homogeneous equation~\eqref{homo} for the initial condition $u_0=\1_U$ and $U=\{(x',x_N)\in\R^{N-1}\times\R:x_N\le\sqrt{|x'|}\}$ and the nonlinearity $f(s)=s(1-s)$, but from its proof, the same conclusion holds for instance for a function $f$ of type~\eqref{bistable} with $\int_0^1f>0$.} However, under some additional conditions on $u_0$ and $f$ for~\eqref{homo}, the rescaled upper level sets $t^{-1}E_\lambda(t)$ converge locally to $\W$ in the sense of the Hausdorff distance as $t\to+\infty$, see~\cite[Theorem~2.4]{HR1}.

\begin{remark}
Proposition~$\ref{pro:level}$ and property~\eqref{Wint} of Proposition~$\ref{pro:Wchar}$ hold under the sole standing assumptions of the paper, and even without Hypotheses~$\ref{hyp:comb-bi}$,~$\ref{hyp:c*>0}$ or~$\ref{hypbis}$. Furthermore, the results~\eqref{inner-cone}-\eqref{outer-cone} of Proposition~$\ref{pro:Wchar}$ hold under the sole invasion property. As a consequence, from Proposition~$\ref{pro:c>0invasion}$, the results~\eqref{inner-cone}-\eqref{outer-cone} hold true under Hypotheses~$\ref{hyp:c*>0}$ and~$\ref{hypbis}$. But they also hold under the conditions of Proposition~$\ref{proinvasion}$. 
\end{remark}

\paragraph{Examples of \ais\ for homogeneous equations with $u_0$ compactly supported.}

Firstly, if~\eqref{eqvarphi} has a solution $(c,\varphi)$, then the \ais\ $\W_0$ of any invading solution $u$ to~\eqref{homo} exists and it is equal to
\Fi{WBc*}
\W_0=B_{c^*},
\Ff
where $c^*>0$ is the minimal speed for which~\eqref{eqvarphi} has a solution, see~\cite{AW} under further assumptions on $f$ and~\cite{HR1} in the general case. This result holds even without Hypothesis~\ref{hyp:comb-bi}, and the speed $c$ in~\eqref{eqvarphi} is not unique in general, but the minimal speed $c^*>0$ does exist~\cite{HR1}. With Hypothesis~\ref{hyp:comb-bi},~\eqref{WBc*} can also be viewed as a consequence of~\eqref{formuleFG}.

In the homogeneous case with $q\equiv0$, $f(x,s)=f(s)$ and $A=(A_{ij})_{1\le i,j\le N}$ is a constant (independent of $x$) symmetric positive definite matrix,~\eqref{general} reduces to
\Fi{homobis}
\partial_t u=\sum_{1\le i,j\le N}A_{ij}\partial_{x_ix_j}u+f(u),\quad t>0,\ x\in\R^N.
\Ff
A change of variable shows that the \ais\ $\W_0$ of any invading solution $u$ with compactly supported $u_0$ is an open ellipsoid of the type
$$\W_0=\mathcal{R}(\mathcal{E}),$$
where $\mathcal{R}$ is an orthogonal mapping, $\mathcal{E}:=\big\{x\in\R^N:x_1^2/\lambda_1+\cdots+x_N^2/\lambda_N<(c^*)^2\big\}$ and~$(\lambda_i)_{1\le i\le N}$ are the eigenvalues of the constant matrix $A$.

The \ais\ can also exist and be empty: for instance, for~\eqref{homo} or~\eqref{homobis} in the bistable case~\eqref{bistable}, if $\|u_0\|_{L^\infty(\R^N)}\le\alpha$, or if $\int_0^1f\le0$ and $f'(0)<0$, then $\|u(t,\cdot)\|_{L^\infty(\R^n)}\to0$ as $t\to+\infty$, whence the \ais\ is empty. 

\paragraph{Examples for~\eqref{homo} with unbounded initial support.}

Even for~\eqref{homo} and even with the existence of a pair $(c,\varphi)$ solving~\eqref{eqvarphi} (hence, Hypothesis~\ref{hyp:c*>0} holds), the question of the existence of an \ais\ is much more intricate when the initial datum $u_0$ is no longer compactly supported. Various conditions for the existence of an \ais\ have been given in~\cite{HR1} when $u_0=\1_U$, involving notions of ``bounded" or ``unbounded" directions of~$U$. For instance, if
$$U=\big\{x=(x',x_N)\in \R^{N-1}\times\R:x_N\le\Gamma(x')\big\}$$
with $\Gamma\in L^\infty_{loc}(\R^{N-1})$ satisfying $\Gamma(x')/|x'|\to\alpha\in[-\infty,+\infty]$ as $|x'|\to+\infty$, then the \ais\ $\W$ exists. If $\alpha=+\infty$, then $\mc{W}=\R^N$. If $0<\alpha<+\infty$, then~$\W$ is a shifted cone
\be\label{Walpha}
\mc{W}=\big\{x\in\R^N:x_N< \alpha\,|x'|+c^*\sqrt{1+\alpha^2}\big\},
\ee
that is, $\W=B_{c^*}+\mathcal{C}$, where $\mathcal{C}:=\{x=(x',x_N)\in \R^{N-1}\times\R:x_N\le\alpha|x'|\}$ and $c^*$ is the minimal speed for which~\eqref{eqvarphi} has a solution. Here, $\W$ is neither convex nor $C^1$ (and it does not satisfy the exterior ball condition at its boundary point $(0,\cdots,0,c^*\sqrt{1+\alpha^2})$), see Figure~\ref{fig:cones}~(a) with $\Gamma(x')=\alpha\,|x'|$. If~$\alpha=0$ (this is the case for instance when $\gamma$ is bounded), then~$\W$ is the half-space $\W=\{x\in\R^N:{x_N<c^*}\}$. If~$-\infty<\alpha<0$, then~$\mc{W}=B_{c^*}+\mathcal{C}$ is still the $c^*$-neighborhood of the cone $\mathcal{C}$, but $\mc{W}$ is now convex and $C^1$ (it is not $C^2$ but it satisfies the interior and exterior ball conditions at every boundary point), see Figure~\ref{fig:cones}~(b) with $\Gamma(x')=\alpha\,|x'|$. If $\alpha=-\infty$, then $\mc{W}=\big\{x\in\R^N:{|x'|<c^*},\ {x_N<0}\big\}\cup B_{c^*}$, which is convex and $C^1$, but not $C^2$. 

\begin{figure}[ht]
 \centering
 \subfigure[$\alpha<0$]
   {\includegraphics[width=.45\linewidth]{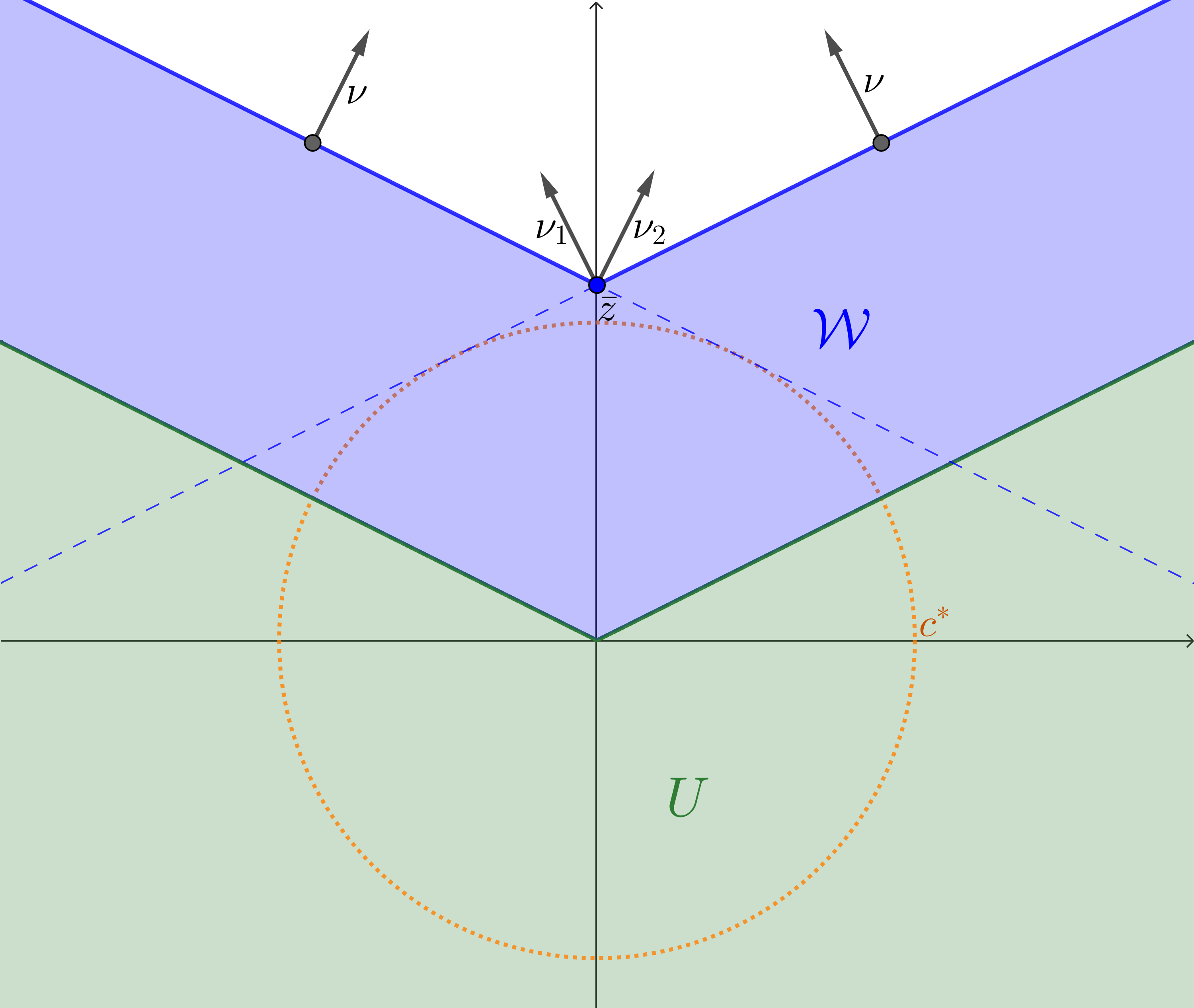}}
 \hspace{8mm}
 \subfigure[$\alpha>0$]
   {\includegraphics[width=.45\linewidth]{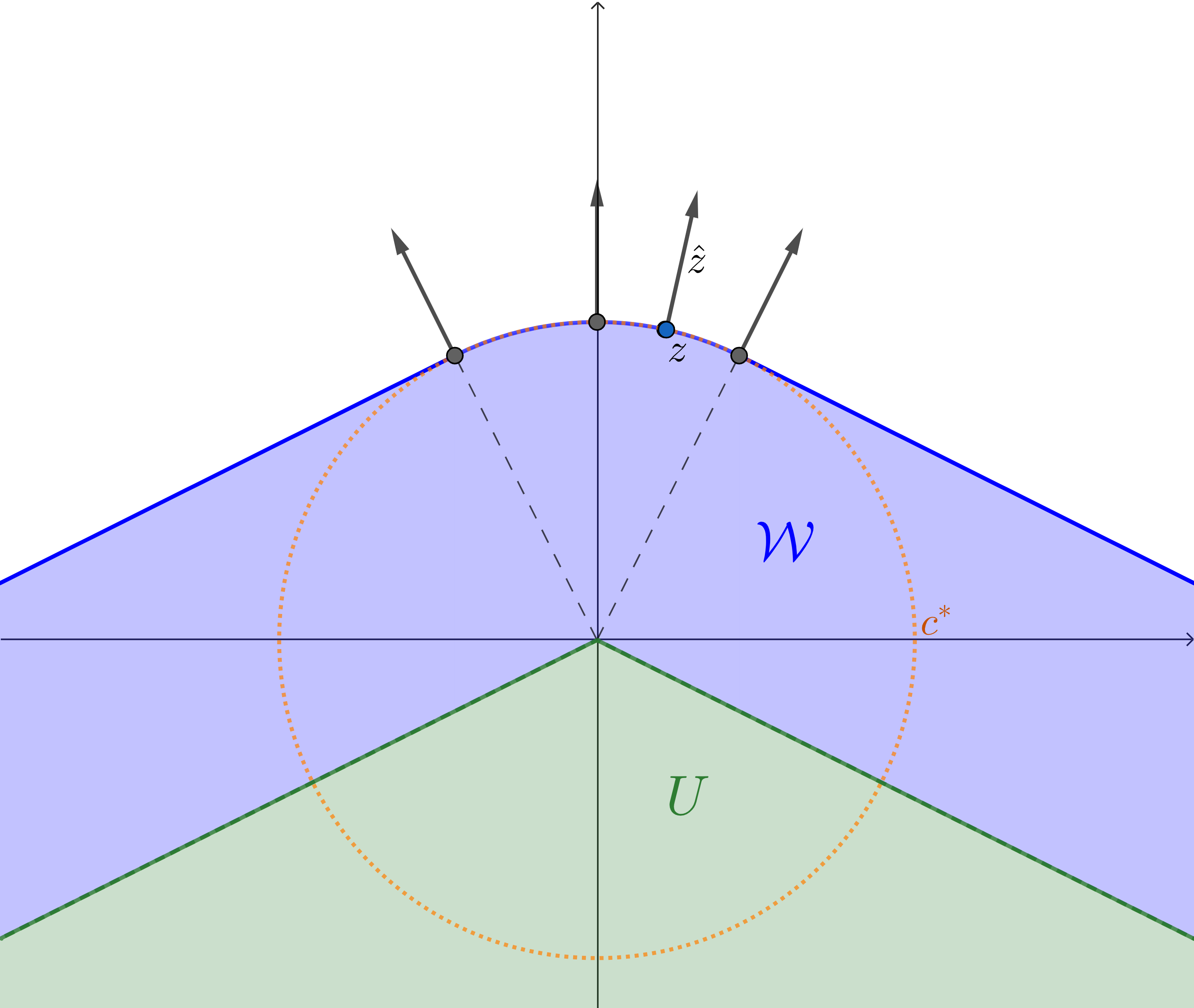}}
 \caption{The \ais\ $\W$ for $u_0=\1_U$ with $U=\{x_N\leq\alpha|x'|\}$. In the case~(a) $\alpha>0$, $\W$ is not regular at $\bar z$, where $\V(\bar z)=\{\nu_1,\nu_2\}$ (in dimension $N=2$). In~the~case~(b) $\alpha<0$, $\W$ satisfies the exterior and interior ball conditions at every boundary point. }
 \label{fig:cones}
 \end{figure}
 
\paragraph{Counter-examples for~\eqref{homo}.}

Three counter-examples on the existence of an {\ais} $\W$ for~\eqref{homo} are given. Firstly, for~$f$ of bistable type~\eqref{bistable} with $\int_0^1f>0$ and $f'(0)<0$, there are compactly supported $u_0$ such that ${\|u(t,\cdot)-\Phi\|_{L^\infty(\R^N)}}\to0$ as $t\to+\infty$, where $\Phi:\R^N\to(0,1)$ is a steady state of~\eqref{homo} such that $\Phi(x)\to0$ as $|x|\to+\infty$, whence $\W$ does not exist, see~{\rm{\cite{DM1,MZ1,MZ2,P1,Z1}}}. 

Secondly, since by~\cite{HR1} the existence of a solution $(c,\varphi)$ of~\eqref{eqvarphi} implies the invasion property, it is natural to wonder whether, under the sole invasion property, the invading solutions which are initially compactly supported admit an \ais. This is actually false in general, as follows from~\cite{FM} in dimension $1$ and from~\cite{DM2} in any dimension, for some multistable functions~$f$: there can exist invading solutions developing into a terrace of fronts expanding with different speeds (see~\cite{DGM,GR,P3,P4} for more results on terraces of fronts for~\eqref{homo} or~\eqref{general}).

Thirdly, even for invading solutions $u$ and even with the existence of a solution $(c,\varphi)$ of~\eqref{eqvarphi}, the \ais\ does not exist in general if $u_0$ is not compactly supported. For instance, if $f(s)=s(1-s)$ and $u_0=\1_U$ with $U:=\bigcup_{n\in\N}\overline{B_{2^n+1}}\setminus B_{2^n-1}$, then~$\W$ does not exist~\cite{HR1}. 
There are also initial conditions $u_0:\R\to(0,1]$ oscillating in a highly non-periodic manner between two exponential functions $e^{-\lambda|x|}$ and $e^{-\mu|x|}$ as~$|x|\to+\infty$, with distinct $\lambda>0$ and $\mu>0$, for which~$\W$ not exist~\cite{HN,Y}.

\paragraph{Examples for the periodic equation~\eqref{general}.}

If the initial condition $u_0\not\equiv0$ is compactly supported, and if $f$ satisfies the generalized Fisher-KPP assumption
\be\label{fkppg}
0<f(x,s)\le\partial_sf(x,0)s\ \hbox{ for all $(x,s)\in\R^N\times(0,1)$},
\ee
then $u$ admits an \ais\ $\W_0$, independent of $u_0$, given by~\eqref{formuleFG}, where $c^*(\xi)>0$ is here the minimal speed of pulsating traveling fronts connecting~$1$ to~$0$ in the direction $\xi$ in the sense of~\eqref{ptf1}-\eqref{ptf2} (all speeds $c\in[c^*(\xi),+\infty)$ are admissible under~\eqref{fkppg}, in contrast with Hypothesis~\ref{hyp:comb-bi}, see Proposition~\ref{pro:U}). These results follow from~\cite{BH1,BHN1,W} and the original formula~\cite{FG}. Actually, formula~\eqref{formuleFG} for compactly supported~$u_0$ holds for a more general class of reaction terms $f$, basically as soon as pulsating traveling fronts exist in every direction~\cite{R1}. This includes, beyond the Fisher-KPP type~\eqref{fkppg}: the generalized positive type~\eqref{fpositive} (see also~\cite{W}), the generalized ignition type~\eqref{fignition}, as well as the bistable type~\eqref{bistable2} with Hypothesis~\ref{hyp:c*>0}. 

Consider now step-like initial conditions $u_0:\R^N\to[0,1]$, for which $u_0(x)=0$ when $x\cdot e$ is large, and $u_0(x)\to1$ as $x\cdot e\to-\infty$, for some $e\in\Sph$. If $f$ is of the generalized ignition or positive cases~\eqref{fignition} or~\eqref{fpositive}, then $u$ has the \ais
\Fi{Wstep}
\W=\big\{x\in\R^N:x\cdot e<c^*(e)\big\},
\Ff
where $c^*(e)$ is the unique or minimal speed of pulsating traveling fronts connecting~$1$ to~$0$ in the direction $e$~\cite{LZ1,LZ2,W,Zh}. For bistable-type equations, results of the type~\eqref{Wstep} are also known, provided pulsating fronts exist in the direction $e$~\cite{X3,X4}.


\subsection{Comments on the conclusions of the main results}\label{sec24}
 
\paragraph{The necessity of the assumptions.} 

Before giving some examples of applications of the main results of Section~\ref{intro}, we first point out that the conclusions do not hold or even do not make sense in general without any of the Hypotheses~\ref{hyp:comb-bi} or~\ref{hyp:c*>0}, or if $u$ is not invading (in Theorem~\ref{th1} or Corollary~\ref{cor2}) or if $u$ does not admit an \ais\ (in Theorem~\ref{th2} or Corollary~\ref{cor1}).

First of all, the conclusions do not hold true in general without Hypothesis~\ref{hyp:comb-bi}, which is required for the uniqueness of the speed of pulsating fronts in each direction. Indeed, if the speeds are not unique, as it is the case for the homogeneous equation~\eqref{homo} with $f>0$ in $(0,1)$, then considering any other front solution $u(t,x):=\varphi(x\cdot e-ct)$ in a direction $e\in\Sph$, with a speed $c>c^*$, one clearly has that $\W=\{x\in\R^N:x\cdot e<c\}$, thus ${z}\cdot\nu=c>c^*$ at any ${z}\in\partial\W$, and moreover $\Omega(u)$ is just given by $\{0,1\}$ and any translation of the profile $\varphi(x\.e)$, hence it does not contains the profiles of the minimal front.

Consider next~\eqref{homo} with $f$ of bistable type~\eqref{bistable} with $f'(0)<0$, $f'(1)<0$, and $\int_0^1f>0$. Hence, Hypotheses~\ref{hyp:comb-bi}-\ref{hyp:c*>0} hold. On the one hand, if $\|u_0\|_{L^\infty(\R^N)}<\alpha$, the solution $u$ with initial condition $u_0$ satisfies $\|u(t,\cdot)\|_{L^\infty(\R^N)}\to0$ as $t\to+\infty$, whence it admits the \ais\ $\W=\emptyset$, but it is not invading and the conclusions of Theorem~\ref{th1} or Corollary~\ref{cor2} do not hold, since $\Omega(u)=\{0\}$.\footnote{The same observations remain valid if $\|u_0\|_{L^1(\R^N)}$ is small enough, since then $\|u(1,\cdot)\|_{L^\infty(\R^N)}$ would be less than $\alpha$, by comparing $u$ with a linear heat equation and using the Lipschitz-continuity of $f$).} On the other hand, by~\cite{DM1,MZ1,MZ2,P1,Z1}, there is a unique $R^*>0$ such that $\|u(t,\cdot)-\Phi\|_{L^\infty(\R^N)}\to0$ as $t\to+\infty$ when $u_0=\1_{B_{R^*}}$, with $\Phi:\R^N\to(0,1)$ being a steady state of~\eqref{homo} such that $\Phi(x)\to0$ as $|x|\to+\infty$. Here, $u$ is not invading and does not have any \ais, so Theorem~\ref{th2} and Corollary~\ref{cor1} do not make sense. The conclusions of Theorem~\ref{th1} and Corollary~\ref{cor2} do not hold either, since in this case $\Omega_e(u)=\{\Phi(\cdot+a):a\in\R^N,\,a\cdot e\ge0\}\cup\{0\}$ for any $e\in\Sph$ and the planar front profiles do not coincide with $\Phi$ up to shift.

In the examples of the previous paragraph, the solutions $u$ were not invading solutions. Without Hypothesis~\ref{hyp:c*>0}, the conclusions of the main results do not hold either even for invading solutions. 
For instance, from~\cite{DM2,FM}, there are homogeneous equations~\eqref{homo} with multistable functions $f$ for which invading solutions with initial compact support exist but fronts connecting $1$ to $0$ do not (so the conclusions of Theorems~\ref{th1} and~\ref{th2} and Corollaries~\ref{cor1}-\ref{cor2} do not make sense in such a case).

\paragraph{Earlier results for~\eqref{homo} and compactly supported $u_0$.}

Consider~\eqref{homo} with~$f$ of ignition type~\eqref{ignition} and $f'(1)<0$, or of bistable type~\eqref{bistable} with $f'(0)<0$, $f'(1)<0$, and $\int_0^1f>0$. Hypotheses~\ref{hyp:comb-bi}-\ref{hyp:c*>0} are then fulfilled. Pick any invading solution $u$ with~$u_0$ compactly supported. By~\eqref{WBc*}, $u$ admits the \ais\ $\W=B_{c^*}$, with~$(c^*,\varphi^*)$ denoting the unique solution of~\eqref{eqvarphi} (up to shift for $\varphi^*$). It follows from~\cite{BH2,U2} that, for any $e\in\Sph$, $\phi_e(t,x)\equiv\varphi^*(x\cdot e-c^*t)$ in $\R\times\R^N$ up to shift in time, and that
\Fi{Omegaeu}
\Omega_e(u)=\big\{x\mapsto\varphi^*(x\cdot e+\sigma):\sigma\in\R\big\}\cup\{0,1\}.
\Ff
This result is coherent --~and even in this case stronger~-- 
than the inclusions regarding~$\Omega_e(u)$ or~$\Omega_{\hat{z}}(u)$ in the main results since the inclusions there are equalities here. The same conclusion~\eqref{Omegaeu} regarding the invading solutions to~\eqref{homo} with compactly supported~$u_0$ holds with more general bistable-type functions $f$ for which~\eqref{eqvarphi} admits a solution~$(c^*,\varphi^*)$ (see~\cite{DGM,FM,P4} in dimension $N=1$ and~\cite{DM2} for general $N\ge1$). It holds as well with Fisher-KPP type functions $f$ such that $f>0$ and $s\mapsto f(s)/s$ is decreasing in~$(0,1)$, where $\varphi^*(x\cdot e-c^*t)$ is a planar front connecting $1$ and $0$ with minimal speed~$c^*$ (see~\cite{B,HNRR1,L,NRR,U1} if $N=1$ and~\cite{D1,G1,RRR} for general $N\ge1$).

\paragraph{Applications and counter-examples for the homogeneous equation~\eqref{homo} with non-compactly supported~$u_0$.}

Consider the solutions $u$ of~\eqref{homo} with $u_0=\1_U$. From~\cite[Theorem~3.1]{BH2}, Hypotheses~\ref{hyp:comb-bi}-\ref{hyp:c*>0} mean that~\eqref{eqvarphi} has a solution $(c^*,\varphi^*)$ and that $f$ is non-increasing in a right neighborhood of $0$ and decreasing in a left neighborhood of $1$. The case of non-compactly supported~$u_0$ is much richer than the compact case and leads to new and different phenomena according to the shape of~$U$. To illustrate these phenomena, consider for instance $U$ of conical type
\Fi{Ualpha}
U=\big\{x=(x',x_N)\in\R^{N-1}\times\R:x_N\le\alpha\,|x'|\big\},\ \hbox{
with $\alpha\in\R^*$.}\footnote{In the case $\alpha=0$ or more generally if $U$ is trapped between two shifts of the half-plane $\{x_N\le0\}$, it follows from~\cite{BH2,P3}, or~\cite{BH2,FM,MN,MNT,RR1} in the pure bistable case~\eqref{bistable}, that $\W=\{x_N<c^*\}$, $\nu=\mathrm{e}_N$ for any ${z}\in\partial\W$, and that $\Omega_{\hat{z}}(u)$ is equal to the union of $\{0,1\}$ and the planar profiles $x\mapsto\varphi^*(x_N+\sigma)$.}
\Ff

If $\alpha<0$, then, by Section~\ref{sec23}, $u$ has the \ais~${\mc{W}=B_{c^*}+U}$ (the same asymptotic invasion shape $\W$ would also arise for~$u_0=\1_{U'}$ with $U'$ trapped between two shifts of~$U$), see Figure \ref{fig:cones}~(b). In particular, $\mc{W}$ satisfies the interior and exterior ball conditions at any point ${z}\in\partial\W$. Then Theorem~\ref{th2} implies that ${z}\cdot\nu=c^*$ (as one can also directly check here) and that $\Omega_{\hat{z}}(u)$ contains all shifts of the planar profile $x\mapsto\varphi^*(x\cdot\nu)$, together with the constants $0$ and $1$. The set $\partial\W$ is composed by a spherical part, whose points ${z}=({z}',{z}_N)$ satisfy $|{z}|=c^*$ and $\nu=\hat{z}$, and a conical part where $|{z}|> c^*$ and
\Fi{formulenu}
\nu=\frac{\displaystyle-\alpha\,\hat{z'}+\mathrm{e}_N}{\sqrt{\alpha^2+1}},\ \hbox{ where $\mathrm{e}_N:=(0,\cdots,0,1)$}.
\Ff
As a matter of fact, in this specific case,~\cite[Theorem~3.3]{HR3} implies that, for every sequence $(t_n)_{n\in\N}\to+\infty$ and every sequence $(x_n)_{n\in\N}$ in $\R^N$ such that $\sup_{n\in\N}|x'_n|<+\infty$, the sequence $(u(t_n,x_n+\cdot))_{n\in\N}$ converges, up to a subsequence, to a function depending on the variable $x_N$ only. Other local asymptotic one-dimensional properties for~\eqref{homo} with $u_0=\1_U$ are given in~\cite{HR2,HR3}. However,~\cite{HR2,HR3} do not give the convergence to fronts' profiles at large time and, up to our knowledge, Theorem~\ref{th2} and Corollary~\ref{cor1} are new even for~\eqref{homo} in the simple case of the cone~\eqref{Ualpha} (or if $U$ were trapped between two shifts of that cone).

If $\alpha>0$, then $u$ still has the \ais\ expressed by~$\mc{W}=B_{c^*}+U$, or equivalently~\eqref{Walpha} (this would hold true for $u_0=\1_{U'}$ with $U'$ lying between two shifts of $U$), but now $\partial\W$ is not differentiable at the vertex $\bar z:=(0,\cdots,0,c^*\sqrt{\alpha^2+1})$, see Figure \ref{fig:cones}~(a). Outside $\bar z$, the set $\partial\W$ is of class $C^2$ and satisfies ${z}\cdot\nu=c^*$ for every $z\in\partial\W\setminus\{\bar z\}$, with outward unit normal still given by~\eqref{formulenu}, hence Theorem~\ref{th2} implies that $\Omega_{\hat{z}}(u)$ contains all shifts of $x\mapsto\varphi^*(x\cdot\nu)$, in addition to $\{0,1\}$. But Theorem~\ref{th2} can not be applied at the vertex $\bar z$, where $\W$ does not fulfill the exterior ball condition. However, Corollary~\ref{cor1} implies that $\bar z\cdot\nu=c^*$ for all $\nu\in\mc{V}(\bar z)$, as one can also directly check
noticing that
$$\mc{V}({\bar z})=\Big\{\frac{\alpha e'+\mathrm{e}_N}{\sqrt{\alpha^2+1}}:e'\in\R^{N-1},\,|e'|=1\Big\}=\big\{\nu\in\Sph:{\bar z}\cdot\nu=c^*\big\},$$
and, in addition, that $\Omega_{\hat{\bar z}}(u)=\Omega_{\mathrm{e}_N}(u)$ contains all shifts of all the profiles $x\mapsto\varphi^*(x\cdot\nu)$ for $\nu\in\mc{V}({\bar z})$, together with $\{0,1\}$. This result is new in this geometry in any dimension $N\ge2$ under the sole Hypotheses~\ref{hyp:comb-bi}-\ref{hyp:c*>0}.

On the other hand, still for $\alpha>0$, more precise results are known for~\eqref{homo} in dimension $N=2$ and $f$ of bistable type~\eqref{bistable} with $\int_0^1f>0$. Namely, it follows from~\cite{HMR1,HMR2,NT} that: 1)~$\Omega_e(u)=\{1\}$ for any $e=(e_1,e_2)\in\mathbb{S}^1$ with $e_2<\alpha|e_1|$; 2)~$\Omega_e(u)$ is the union of $\{0,1\}$ with all shifts of the planar profile $x\mapsto\varphi^*((x_2-\alpha x_1)/\sqrt{\alpha^2+1})$, resp. $x\mapsto\varphi^*((x_2+\alpha x_1)/\sqrt{\alpha^2+1})$, if $e_2\ge\alpha|e_1|$ and $e_1>0$, resp. $e_1<0$; 3)~$\Omega_{(0,1)}(u)$ is the union of $\{0,1\}$ with all shifts of the two planar profiles $x\mapsto\varphi^*((x_2\pm\alpha x_1)/\sqrt{\alpha^2+1})$ and all shifts of a {\em $V$-shaped} profile $\phi:\R^2\to(0,1)$ (which is even in $x_1$, decreasing in $x_2$, and converges to $1$ and $0$ as $x_2-\alpha |x_1|\to-\infty$ and $+\infty$ respectively). In particular, $\Omega_{(0,1)}(u)$ does not contain the planar profile $x\mapsto\varphi^*(x_2)$, but it contains a non-planar profile, showing that the inclusion in part~(ii) of Corollary~\ref{cor1} is strict in general. Furthermore, the set of directions $\nu$ such that $\phi_\nu(t,\cdot+y)\in\O(u)$ is not connected in general. However, in this example, the inclusion in part~(ii) of Theorem~\ref{th2} is an equality at regular points of $\partial\W$. We actually conjecture it is truly an equality at points ${z}\in\partial\W$ around which~$\W$ is sufficiently smooth, but this question is left open.

\paragraph{The general periodic equation~\eqref{general}.}

Up to our knowledge, the convergence to profiles of pulsating traveling fronts at large time is completely new for general periodic equations~\eqref{general}. The only known related results were obtained under further assumptions on~\eqref{general} for front-like initial conditions in dimension $1$~\cite{DHZ2,DGM,X2} or in higher dimensions as follows from~\cite{BH3,X3}, or for equations with Fisher-KPP reactions~\eqref{fkppg} in dimension~$1$~\cite{HNRR2}.

In any dimension $N\ge1$, under Hypotheses~\ref{hyp:comb-bi}-\ref{hyp:c*>0}, even for invading solutions which are initially compactly supported, the convergence to profiles of pulsating traveling fronts was not known. Theorem~\ref{th1} and Corollary~\ref{cor2} especially answer this question. Furthermore, if the \ais\ $\W_0$ of the invading solutions with compactly supported $u_0$ is of class $C^1$, then necessarily $\mc{V}({z})=\{\nu({z})\}$ at any point ${z}\in\partial\W_0$, where~$\nu(z)$ is the outward normal, and since the set of all such normals is then equal to the whole sphere $\Sph$, it follows that $\Omega(u)$ contains all shifts of the profiles of the pulsating traveling fronts connecting $1$ to $0$ in all directions, as stated in Corollary~\ref{cor2}. A sufficient condition for $\W_0$ to be $C^1$ is that, for each $e\in\Sph$, there is a unique mini\-mizer $\xi_e\in\Sph$ realizing the minimum in~\eqref{formuleFG}~\cite{GHR2} (see also~\cite{S} for related results on the map $e\mapsto\xi_e$ for $\partial_tu=\Delta u+\mu(x)f(u)$ with Fisher-KPP reactions~$f$). Further comments on the smoothness of $\W_0$ are given in~\cite{GHR2}.
	
Finally, Theorem~\ref{th2} and Corollary~\ref{cor1} are also completely new when $u_0$ is not compactly supported. For instance, if $u_0=\1_U$, $U$ is of the type~\eqref{Ualpha} with $\alpha<0$ and $\partial\W_0$ is $C^1$, where $\W_0$ is the \ais\ of invading solutions with compactly supported $u_0$, then $u$ has an \ais\ $\W=U+\W_0$~\cite{GHR2}, so $\W$ satisfies the interior ball condition at every boundary point if $\W_0$ does so (the exterior ball condition is automatically satisfied along $\partial\W$ by convexity of $U$ and~$\W_0$). In particular, Theorem~\ref{th2} applies to such $u$ and $\W$ in the case when $\W_0$ is of class~$C^2$. But if $\W_0$ is not regular, $\W$ does not necessarily satisfy the interior ball condition at every point $z\in \partial\W$. In this case, curved front profiles may also appear in the $\O$-limit set and the inclusion in Theorem~\ref{th2} and Corollary~\ref{cor1} is strict. Indeed, for example, if $N=2$ and coincidentally
\begin{equation}\label{c-CF}
\hat{c}:=\sqrt{1\!+\!\alpha^2}\,c^*(e_r)=\sqrt{1\!+\!\alpha^2}\,c^*(e_l)\le\frac{c^*(e)}{e_2},  \hbox{ for any $e\!=\!(e_1,e_2)$ with  }e_2>\frac{1}{\sqrt{1\!+\!\alpha^2}},
\end{equation}
where
$$e_r:= \Big(\frac{-\alpha}{\sqrt{1+\alpha^2}},\frac{1}{\sqrt{1+\alpha^2}}\Big),\quad e_l:=\Big( \frac{\alpha}{\sqrt{1+\alpha^2}},\frac{1}{\sqrt{1+\alpha^2}} \Big),$$
(an example of this type is given in~\cite[Corollary~6.2]{GR2}) then $\partial\W$ has a corner at $(0,\hat{c})$ by the formula of the asymptotic invasion shape for the type \eqref{Ualpha} with $\alpha<0$ in \cite{GHR2}, and does not satisfy the interior ball condition at this boundary point. In this case, there can exist a $\Lambda$-shaped front $W: \R\times\R^2\rightarrow (0,1)$ which has level sets boundedly (in the sense of the Hausdorff distance) oscillating around $\partial U+ (0,\hat{c})t$, see~\cite[Theorems~1.3-1.4]{GW}. Therefore, it is expected that $\Omega_{(0,1)}(u)$ contains all profiles $W(s,y+\cdot)$, for $(s,y)\in\R\times\R^N$. We also conjecture that $\Lambda$-shaped profiles appear as well in the $\Omega$-limit set of an invading solution with compactly supported initial condition, so that the inclusion in Theorem~\ref{th1} and Corollary~\ref{cor2} might be strict in that case.


\section{Proofs of the main results}\label{sec3}


\subsection{Proof of Theorem~\ref{th2}}\label{sec31}

This section is devoted to the proof of the main Theorem~\ref{th2}. We preliminarily introduce  some notations and prove an auxiliary lemma which will be used in several further proofs. Namely, for $e\in\Sph$ and $\lambda\in\R$, we consider the principal eigenvalue $k_e(\lambda)$ of the eigenvalue  problem
\be\label{evp}
\dv(A(x)\nabla\varphi)\!+\!(q(x)\!-\!2\lambda A(x)e)\cdot\nabla\varphi\!+\!\big(\lambda^2eA(x)e\!-\!\lambda\dv(A(x)e)\!-\!\lambda q(x)\cdot e\big)\varphi\!=\!k_e(\lambda)\varphi
\ee
in $\R^N$ acting on periodic functions. Notice that the left-hand side is equal to $e^{\lambda x\cdot e}\big[\dv\big(A(x)\nabla(e^{-\lambda x\cdot e}\varphi)\big)+q(x)\cdot\nabla(e^{-\lambda x\cdot e}\varphi))\big]$. From Krein-Rutman and elliptic regularity theories, this principal eigenvalue $k_e(\lambda)$ is uniquely characterized by the existence (and uniqueness up to multiplication) of a positive principal  eigenfunction $\varphi$ of class $C^{2,\alpha}(\R^N)$. 

\begin{lemma}\label{lemkelambda}
For every $e\in\Sph$, there holds $k_e(\lambda)=o(\lambda)$ as $\lambda\to0$.
\end{lemma}

\begin{proof}
Fix $e\in\Sph$. Since the positive constant function $\varphi=1$ solves~\eqref{evp} with $\lambda=0$ and $0$ right-hand side, it follows by uniqueness that $k_e(0)=0$ and that any associated eigenfunction of~\eqref{evp} with $\lambda=0$ is constant. Furthermore, the map $\lambda\mapsto k_e(\lambda)$ is analytic~\cite{K} and, from standard elliptic estimates, if $\varphi$ is normalized so that, say, $\max_{\R^N}\!\varphi=1$, then $\varphi\to1$ as $\lambda\to0$ in $C^2(\R^N)$. Therefore, after integrating~\eqref{evp} over~$(0,1)^N$, using the periodicity of $A$, $q$, $\varphi$ together with $\dv q=0$, one gets that
$$\baa{rcl}
\displaystyle k_e(\lambda)\int_{(0,1)^N}\!\!\varphi & \!\!=\!\! & \displaystyle-2\lambda\int_{(0,1)^N}\!\!A(x)e\cdot\nabla\varphi+\!\!\int_{(0,1)^N}\!\!(\lambda^2eA(x)e\!-\!\lambda\dv(A(x)e)\!-\!\lambda q(x)\cdot e)\varphi\vspace{3pt}\\
& \!\!=\!\! & \displaystyle-\lambda\int_{(0,1)^N}A(x)e\cdot\nabla\varphi+\int_{(0,1)^N}(\lambda^2eA(x)e\!-\!\lambda q(x)\cdot e)\varphi,\eaa$$
after integrating by parts. By dividing by $\lambda\neq0$ and passing to the limit $\lambda\to0$, one derives $k_e(\lambda)=o(\lambda)$ as $\lambda\to0$, since $q$ is assumed to have zero average over $(0,1)^N$.
\end{proof}

The proof of Theorem~\ref{th2} is then achieved through the following three Lemmata~\ref{lem:01}-\ref{lem:exterior}. In their proofs, we will make use of Propositions~\ref{pro:U} and~\ref{pro:Wchar}, that will be proved in Section \ref{sec4}.

Let us recall our notation $\hat x:=x/|x|$, for any $x\in\R^N\setminus\{0\}$.

\begin{lemma}\label{lem:01}
Assume that the invasion property holds. Let $u$ be a solution to~\eqref{general} admitting an \ais\ $\W\neq\emptyset$. There holds that $1\in\O_e(u)$ for every $e\in\Sph$, and moreover $0\in\O_{\hat{z}}(u)$ for any $z\notin\W$.
\end{lemma}
	
\begin{proof}
The inclusion~\eqref{inner-cone} of Proposition~$\ref{pro:Wchar}$ yields $B_{\gamma}\subset\W$ (this follows by taking any $z\in\partial\W$ in~\eqref{inner-cone} if $\W\neq\R^N$, and this is trivial if $\W=\R^N$), hence  in particular, for every $e\in\Sph$, it holds that $u(t,e+\. )\to1$ as $t\to+\infty$, locally uniformly in $\R^N$, that is, $1\in\O_e(u)$. Suppose now that there exists $z\notin\W$. Since $\W$ contains a neighborhood of the origin, it follows that $z\neq0$, and moreover there exists $\bar x\in\partial\W$ such that $\hat{\bar x}=\hat{z}$. From~\eqref{outer-cone} one then derives $B_\gamma(2\bar x)\subset \R^N\setminus\W$, whence $2\bar x\in\R^N\setminus\ol\W$ and therefore~\eqref{ass-cpt} implies that $u(t,2t\bar x+\. )\to0$ as $t\to+\infty$, locally uniformly in $\R^N$. This shows that~$0\in\O_{\hat{\bar x}}(u)=\O_{\hat z}(u)$.
\end{proof}

The following two lemmata are the key-points in the proof of Theorem~\ref{th2}.

\begin{lemma}\label{lem:interior}
Assume that Hypotheses~$\ref{hyp:comb-bi}$-$\ref{hyp:c*>0}$ hold. Let $u$ be a solution to~\eqref{general} admitting an \ais\ $\W$. Assume that $\W$ satisfies an interior ball condition at a point $\bar x\in\partial\W$, i.e., there exist $r>0$ and $\nu\in\Sph$ such that ${B_r(\bar x-r\nu)\subset\W}$. Then one has $\bar x\.\nu\geq c^*(\nu)$. Moreover, if $\bar x\.\nu= c^*(\nu)$ then 
$$\displaystyle\O_{\hat{\bar x}}(u)\supset\big\{\phi_\nu(t,\cdot+y):(t,y)\in\R\times\R^N\big\}.$$	
\end{lemma}
	
\begin{lemma}\label{lem:exterior}
Assume that Hypotheses~$\ref{hyp:comb-bi}$-$\ref{hyp:c*>0}$ hold. Let $u$ be a solution to~\eqref{general} admitting an \ais\ $\W$. Assume that $\W$ satisfies an exterior ball condition at a point $\bar x\in\partial\W$, i.e., there exist $r>0$ and $\nu\in\Sph$ such that ${B_r(\bar x+r\nu)\subset\R^N\setminus\W}$. Then one has $\bar x\.\nu\leq c^*(\nu)$. Moreover, if $\bar x\.\nu= c^*(\nu)$ then 
$$\O_{\hat{\bar x}}(u)\supset\big\{\phi_\nu(t,\cdot+y):(t,y)\in\R\times\R^N\big\}.$$	
\end{lemma}

Postponing the proofs of Lemmata~\ref{lem:interior}-\ref{lem:exterior} in Sections~\ref{sec32}-\ref{sec33}, we first complete the proof of Theorem~\ref{th2}.

\begin{proof}[Proof of Theorem~$\ref{th2}$]
First of all, the existence of $z\in\partial\W$ implies that $\W$ is non-empty, and thus it contains the origin, owing to Propositions~\ref{pro:c>0invasion} and~\ref{pro:Wchar}. In particular, $z\neq0$ (recall that $\W$ is open) hence $\hat z$ is well defined.

Next, on the one hand, Lemmata~\ref{lem:interior} and~\ref{lem:exterior} directly imply the statement $(i)$ of the theorem, and consequently also the inclusion
$$\O_{\hat{z}}(u)\supset\big\{\phi_\nu(t,\cdot+y):(t,y)\in\R\times\R^N\big\}.$$
On the other hand, the inclusion $\{0,1\}\subset \O_{\hat{z}}(u)$ is provided by Lemma~\ref{lem:01}.
\end{proof}


\subsection{Proof of Lemma~\ref{lem:interior}}\label{sec32}

Let $\bar x\in\partial\W$, $r>0$ and $\nu\in\Sph$ be as in the statement of the lemma. The proof amounts to showing that if $\bar x\.\nu\leq c^*(\nu)$, then necessarily $\bar x\.\nu=c^*(\nu)$ and moreover $\phi_\nu(t,\cdot+y)\in\O_{\hat{\bar x}}(u)$, for any $(t,y)\in\R\times\R^N$.
	
So, let us assume that 	
$$c:=\bar x\.\nu\leq c^*(\nu).$$
The strategy consists in constructing two sequences $\seq{t}$ in $\R_+$ and $\seq{x}$ in $\Z^N\setminus\{0\}$ satisfying $t_n\to+\infty$ and $\hat{x_n}\to\hat{\bar x}$, and such that $u(t+t_n,x+x_n)\to \t u(t,x)$ as $n\to+\infty$ locally uniformly in $(t,x)\in\R\times\R^N$, with $\t u(t,x)$ touching from above (a suitable translation of) the function $(t,x)\mapsto U(x, x\cdot \nu - ct)$, where $U=U_\nu$ (for simplicity, we here drop the subscript $\nu$) is the profile of the front~$\phi_\nu$, that is,
$$U(x,x\cdot\nu-c^*(\nu)t)=\phi_\nu(t,x).$$
Since the function $(t,x)\mapsto U (x, x\cdot \nu - ct)$ is a subsolution to \eqref{general} (because $c\leq c^*(\nu)$ and $\partial_zU(x,z)<0$, owing to Proposition~\ref{pro:U}), this would imply that $\t u(t,x)\equiv U (x, x\cdot \nu - ct)$ and eventually that $c=c^*(\nu)$.

In order to achieve our goal, we transform the profile $U(x,z)$ making it radial in~$z$, and then we perturb it in such a way that it touches the solution $u$ from below in some vicinity of points of the type $(t,t\bar x)$. 
	
\subsubsection*{Construction of the ``almost radial'' subsolution}

Since ${B_r(\bar x-r\nu)\subset\W}$ holds, for $0<r'<r$ there holds $\ol{B_{r'}(\bar x-r'\nu)}\setminus\{\bar x\}\subset\W$. Therefore, up to reducing $r$, we can assume without loss of generality that such a stronger property holds, i.e.
\Fi{strong-sphere}
\ol{B_r(y)}\setminus\{\bar x\}\subset\W,\quad\text{ with }\; y:=\bar x-r\nu.
\Ff
In addition, since Proposition~\ref{pro:Wchar} entails that $c:=\bar x\.\nu>0$, we can further assume without loss of generality that $r\leq c/2$.
	
We then consider a $C^1$ function $\rho:[0,+\infty)\to\R$ satisfying
\be\label{defrho}
\rho'>0\hbox{ in $[0,+\infty)$},\quad \rho(0)=\frac r2,\quad \rho(+\infty)=r.
\ee
For a given parameter $\e>0$, we call
\be\label{defrhoeps}
\rho_\e(t):=\int_0^t\rho(\e s)ds+\e t,\ \ t\in[0,+\infty).
\ee
The crucial property is that, for $0<\e\ll1$, the derivative of $\rho_\e$ in $[0,+\infty)$ goes from~$\rho_\e'<r$ to $\rho_\e'>r$ very slowly. For $\e>0$, we define
$$v^\e(t,x):=U\big(x, |x-ty| - \rho_\e(t)\big),\ \ (t,x)\in[0,+\infty)\times\R^N.$$
Next, we lower the function $v^\e$ in a suitable way.

We start by considering the periodic principal eigenvalues $k_\nu(\lambda)$ given by \eqref{evp} with~${e=\nu}$ and $\lambda>0$. We know from Lemma~\ref{lemkelambda} that, for $\lambda>0$ sufficiently small, it holds 
\Fi{klambda<}
k_\nu(\lambda)\leq \frac{r}2\lambda.
\Ff
We then take $\lambda>0$ small enough such that \eqref{klambda<} holds, and in addition $\lambda<\lambda_0$, where $\lambda_0>0$ is the quantity provided by Proposition~\ref{pro:U}, for which $U=U_\nu$ satisfies the estimate~\eqref{Ueexp}. We then consider a nonnegative $C^2$ function $\chi:\R\to\R$ satisfying
\be\label{defchi}\chi(z)=\begin{cases}
0 & \text{for }z\leq-1\\
e^{-\lambda z} & \text{for }z\geq1\,.\\
\end{cases}
\ee
Hence, the function $\chi$ decays to $0$ at $+\infty$ more slowly than $U(x,\.)$. Notice that $\partial_zU$ and its first-order derivatives have the same decay~\eqref{Ueexp} as $U$, as a consequence of the parabolic estimates applied to $\phi_\nu$, and the fact that $c^*(\nu)>0$. In particular, there exists a constant $K>0$ (independent of $\e>0$) such that
\Fi{U<chi}
\forall\,x\in\R^N, \ \forall\,z\geq1,\quad |\partial_zU(x,z)+|\partial^2_{zz}U(x,z)|+|\nabla\partial_zU(x,z)| \leq K\chi(z).
\Ff
Next, for $\e>0$, we call
$$
w^\e(t,x):=\chi\big(|x-ty| - \rho_\e(t)\big)\,\vp(x)+1,\ \ (t,x)\in[0,+\infty)\times\R^N,
$$
where $\vp$ is the periodic principal eigenfunction of the elliptic operator~\eqref{evp} associated with $k_\nu(\lambda)$, normalized by $(\max_{\R^N}\vp)(\max_\R\chi)=1$. Therefore, for every $\e>0$, there holds
\be\label{weps12}
1\le w^\e\le2\ \hbox{ in $[0,+\infty)\times\R^N$}.
\ee

Finally, for $\e>0$ and for another parameter $\mu>0$, we define
$$v^\e_\mu(t,x):=v^\e(t,x)-\mu w^\e(t,x),\ \ (t,x)\in[0,+\infty)\times\R^N.$$
This function $v^\e_\mu$ satisfies $v^\e_\mu\leq1-\mu$ in $[0,+\infty)\times\R^N$. Furthermore, since $U(x,z)\to0$ as $z\to+\infty$ uniformly in $x\in\R^N$, there exists $R_\mu>0$ (depending on~$\mu>0$ but not on~$\e>0$) such that
\Fi{ve<}
\forall\,\mu>0,\ \forall\,\e>0,\ \forall\,t\ge0,\ \forall\,x\in \R^N,\ \  |x-ty|-\rho_\e(t)\geq R_\mu\ \Longrightarrow\ v^\e_\mu(t,x)<0.
\Ff
Finally, if $0<\mu<1/4$ then $\mu w^\e<1/2$ in $[0,+\infty)\times\R^N$ and, since $U(x,z)\to1$ as $z\to-\infty$ uniformly in $x\in\R^N$, we can therefore find another quantity $L\in\R$ (independent of $\mu\in(0,1/4)$ and $\e>0$) such that
\Fi{ve>}
\forall\,0<\mu<\frac14,\ \forall\,\e>0,\ \forall\,t\ge0,\ \forall\,x\in \R^N,\ \  |x-ty|-\rho_\e(t)\leq L\ \Longrightarrow\ v^\e_\mu(t,x)>\frac14\,.
\Ff
{\it{For the rest of the proof, we restrict to $0<\mu<1/4$, so that both~\eqref{ve<} and~\eqref{ve>} hold.}}

\subsubsection*{The contact points with $u$}
 
On the one hand, by \eqref{strong-sphere} and \eqref{ass-cpt}, for each $\mu\in(0,1/4)$, there exists $\tau_\mu\geq0$ (independent of $\e>0$) such that
\Fi{u>m}\baa{l}
\displaystyle\forall\,0<\mu<\frac14,\ \forall\,\e>0,\ \forall\,t\geq \tau_\mu,\ \forall\,x\in \R^N,\vspace{3pt}\\
\qquad\qquad\qquad\qquad\qquad\displaystyle|x-ty|\leq\frac34 rt\ \Longrightarrow\ u(t,x)>1-\mu\geq \sup_{[0,+\infty)\times\R^N}\,v^\e_\mu.\eaa
\Ff
On the other hand, for each $\e>0$, consider the function $\rho_\e$ defined in~\eqref{defrhoeps} and call, for $t\ge0$,
$$\gamma(\e,t):=\frac34 rt-\rho_\e(t).$$
This function $\gamma$ is continuous with respect to $(\e,t)\in(0,+\infty)\times[0,+\infty)$ and even in~$[0,+\infty)\times[0,+\infty)$ by setting $\gamma(0,t)=rt/4$ for $t\ge0$. The function $\gamma$ is also decreasing with respect to $\e\ge0$ for any $t>0$, and it is strictly concave with respect to $t\ge0$ and satisfies $\gamma(\e,+\infty)=-\infty$ for any $\e>0$. Finally, it holds $\gamma(0,t)=rt/4\ge R_\mu$ for all $t\in[4R_\mu/r,+\infty)$. From all these properties one deduces that, for every $\mu\in(0,1/4)$, there exists $\e'_\mu>0$ small enough so that, for any $\e\in(0,\e'_\mu)$, there holds
$$\big\{ t>0\ : \ \gamma(\e,t)\geq R_\mu\big\}=[\ul\tau^\e_\mu,\ol\tau^\e_\mu],$$
for some $0<\ul\tau^\e_\mu<\ol\tau^\e_\mu$ satisfying
\be\label{limits1}
\lim_{\e\to0^+}\ul\tau^\e_\mu=\frac{4R_\mu}r,\qquad \lim_{\e\to0^+}\ol\tau^\e_\mu=+\infty.
\ee

As a consequence, owing to~\eqref{ve<}, one derives
$$\baa{l}
\displaystyle\forall\,0<\mu<\frac14,\ \forall\,0<\e<\e'_\mu,\ \forall\,t\in[\ul\tau^\e_\mu,\ol\tau^\e_\mu],\ \forall\,x\in \R^N,\vspace{3pt}\\
\qquad\qquad\qquad\qquad\qquad\qquad\displaystyle|x-ty|\ge\frac34 rt\ \Longrightarrow\ v^\e_\mu(t,x)<0\leq u(t,x).\eaa$$
Then, by taking $\e_\mu''\in(0,\e'_\mu)$ small enough so that
$$\forall\,0<\mu<\frac14,\ \forall\,0<\e<\e''_\mu,\ \ \ \ \ol\tau^\e_\mu-1>\max(\ul\tau^\e_\mu,\tau_\mu),$$
with $\tau_\mu$ given by~\eqref{u>m}, one eventually gets that
\be\label{ineqstricte}
\forall\,0<\mu<\frac14,\ \forall\,0<\e<\e''_\mu,\ \ \ \ u>v^\e_\mu\ \text{ in }[\ol\tau^\e_\mu-1,\ol\tau^\e_\mu]\times\R^N.
\ee
	
We now show that, for given $0<\mu<1/4$ and $\e>0$, one has $u<v^\e_\mu$ at some point, for~$t$ large enough. Specifically, on the one hand, since $\rho_\e(t)/t\to r+\e$ as $t\to+\infty$, one derives from \eqref{ve>} that, for $t$ sufficiently large,
$$\inf_{x\in B_{r+\e/2}(y)}v^\e_\mu(t,tx)>\frac14\,.$$
On the other hand, since $\bar x\in\partial B_r(y)\cap\partial\W$ and $\W$ is the interior of its closure, there exists $x\in B_{r+\e/2}(y)\setminus\ol\W$, hence property~\eqref{ass-cpt} yields $u(t,tx)\to0$ as $t\to+\infty$. We conclude that the quantity $\inf_{\R^N}(u-v^\e_\mu)(t,\cdot)$ is negative for $t$ large enough.
	
Thus, for $0<\mu<1/4$ and $\e\in(0,\e_\mu'')$, we can define
$$t^\e_\mu:=\inf\Big\{t>\ol\tau^\e_\mu\ : \ \inf_{\R^N}\,(u-v^\e_\mu)(t,\cdot)<	0\Big\}\ \in[\ol\tau^\e_\mu,+\infty).$$
Let $\seq{t}$ in $[\ol\tau^\e_\mu,+\infty)$ and $\seq{x}$ in $\R^N$ be such that $t_n\searrow  t^\e_\mu$ as $n\to+\infty$ and $(u-v^\e_\mu)(t_n,x_n)<0$. Property~\eqref{ve<} and the non-negativity of $u$ imply that 
$$|x_n-t_ny|<\rho_\e(t_n)+R_\mu$$
for every $n\in\N$, hence $\seq{x}$ converges (up to subsequences) to some point $x^\e_\mu\in\R^N$ satisfying
$$|x^\e_\mu-t^\e_\mu y|\leq\rho_\e(t^\e_\mu)+R_\mu$$
and $(u-v^\e_\mu)(t^\e_\mu,x^\e_\mu)\le0$, whence $t^\e_\mu>\ol\tau^\e_\mu>0$ and $(u-v^\e_\mu)(t^\e_\mu,x^\e_\mu)=0$ by~\eqref{ineqstricte} and the continuity of $u-v^\e_\mu$. Together with~\eqref{ineqstricte} again, we deduce that
\Fi{contactpoint}
\forall\,0<\mu<\frac14,\ \forall\,0<\e<\e_\mu'',\ \ \ \ \min_{[t^\e_\mu-1,t^\e_\mu] \times\R^N}(u-v^\e_\mu)=(u-v^\e_\mu)(t^\e_\mu,x^\e_\mu)=0.
\Ff

Next, for any $0<\mu<1/4$, using that $\rho_\e(t)\leq(r+\e)t$ for all $t\ge0$ and $\e>0$, and that $t^\e_\mu>\ol\tau^\e_\mu\to+\infty$ as $\e\to0^+$ by~\eqref{limits1}, we infer that $\limsup_{\e\to0^+}|x^\e_\mu/t^\e_\mu- y|\leq r$. Thus, for any sequence $\seq{\e}$ in $(0,\e''_\mu)$ converging to $0^+$, the points $x^{\e_n}_\mu/t^{\e_n}_\mu$ converge as $n\to+\infty$, up to subsequence, to some $\t x\in \ol{B_r(y)}$. But then the point  $\t x$ must coincide with $\bar x$, because otherwise~\eqref{strong-sphere} and~\eqref{ass-cpt} would yield $u(t^{\e_n}_\mu,x^{\e_n}_\mu)\to 1$ as $n\to+\infty$, while we know that $u(t^{\e_n}_\mu,x^{\e_n}_\mu)= v^\e_\mu(t^{\e_n}_\mu,x^{\e_n}_\mu)\leq 1-\mu$ for all $n\in\N$. Summing up, we have that
\Fi{contactpoint0}
\forall\,0<\mu<\frac14,\qquad \lim_{\e\to0^+}t^\e_\mu=+\infty,\qquad\lim_{\e\to0^+}\frac{x^\e_\mu}{t^\e_\mu}= \bar x.
\Ff
	
\subsubsection*{The strict subsolution property}

In the light of the property of the contact points derived above, we analyze the equation satisfied by the function $v^\e_\mu$ at time~$t$ in the vicinity of the point $t\bar x$. More precisely, we will focus on $x\in t B_\omega(\bar x)$ with given~$\omega\in(0,r/2)$, so in particular $x\neq ty$.

We start by considering the function $(t,x)\mapsto v^\e(t,x)=U\big(x, |x-ty| - \rho_\e(t)\big)$, with $\e>0$. Explicit computation~yields, for all $t>0$ and $x\in tB_\omega(\bar x)$,
\[\begin{split}
\partial_t v^\e-\dv(A\nabla v^\e)-q\.\nabla v^\e=&-\big(\rho_\e'(t)+\hat{x-ty}\.y\big)\partial_zU\\
&-\dv(A\nabla U)-\hat{x-ty}\,A\,\hat{x-ty}\,\partial^2_{zz}U\\
&-\dv\big(A\hat{x-ty}\big)\,\partial_z U-2\hat{x-ty}\,A\nabla \partial_z U\\
&-q\.\nabla U-q\.\hat{x-ty}\,\partial_z U,
\end{split}\]
where, in the above terms, $U$ is treated as $U=U(x,z)$ and eventually evaluated at $z=|x-ty| - \rho_\e(t)$. Exploiting the fact that $U=U_\nu$ solves
$$-c^*(\nu)\partial_zU-\dv(A\nabla U)-2\nu A\nabla\partial_z U-\dv(A\nu)\partial_z U-\nu A\nu\partial^2_{zz}U-q\.\nabla U-q\.\nu\partial_z U=f(x,U),$$
one then infers that, for all $t>0$ and $x\in tB_\omega(\bar x)$,
\Fi{ve}
\begin{split}
\partial_t v^\e-\dv(A\nabla v^\e)-q\.\nabla v^\e = f(x,v^\e)+\mc{R}(t,x),
\end{split}
\Ff
where
$$\baa{rcl}
\mc{R}(t,x) & \!\!\!:=\!\!\! & \left[c^*(\nu)-\rho_\e'(t)-\hat{x-ty}\!\.\!y+\dv\left(A\big(\nu-\hat{x-ty}\big)\right)+q\.\big(\nu-\hat{x-ty}\big)\right]\partial_zU\vspace{3pt}\\
& \!\!\!\!\!\! & +2\big(\nu-\hat{x-ty}\big)A\nabla\partial_z U+\big(\nu A\nu -\hat{x-ty}\,A\,\hat{x-ty}\big)\partial^2_{zz}U.\eaa$$

We now estimate the reminder $\mc{R}(t,x)$ when $t>0$ and $x\in t B_\omega(\bar x)$ (which will be the region where the contact points $(t^\e_\mu,x^\e_\mu)$ of $u$ and $v^\e_\mu$ will eventually lie, by~\eqref{contactpoint}-\eqref{contactpoint0}). First, recalling that $\omega<r/2$, we have that
\Fi{tonu}
\sup_{x\in t B_\omega(\bar x)} \big|\hat{x-ty}-\nu\big|=\sup_{x\in B_\omega(\bar x)} \big|\hat{x-y}-\nu\big|=\sup_{x\in B_\omega} \big|\hat{x+r\nu}-\nu\big|\leq \frac{2\omega}{r-\omega}\leq \frac{4\omega}{r}.
\Ff
On the one hand, since
$$\dv\left(A\big(\nu-\hat{x-ty}\big)\right)=(\dv A)\.\big(\nu-\hat{x-ty}\big)-\frac{1}{|x-ty|}\big(\Tr A -\hat{x-ty}\,A\,\hat{x-ty}\big),$$
we deduce from \eqref{tonu} that there exists $C>0$, depending only on $\|A\|_{W^{1,\infty}}$, $r$, and $N$, such that
\[\begin{split}
\forall\,0<\omega<\frac{r}2,\ \forall\,t>0,\ \ \ \ \sup_{x\in t B_\omega(\bar x)}	\Big|\dv\left(A\big(\nu-\hat{x-ty}\big)\right)\Big|&\leq C\omega+\frac{C}{t}.
\end{split}\]	
On the other hand, we find again by \eqref{tonu} that
\Fi{reminder0}
\forall\,0\!<\!\omega\!<\!\frac{r}2,\ \forall\,t\!>\!0,\ \forall\,x\!\in\! t B_\omega(\bar x),\ \ \left|\hat{x\!-\!ty}\.y\!-\!c\!+\!r\right|\leq|\nu\.y\!-\!c\!+\!r|\!+\!\frac{4\omega|y|}{r}\!=\!\frac{4\omega|y|}{r},
\Ff
from which, using $c\leq c^*(\nu)$ and $\rho\leq r$, we derive
$$\baa{l}
\displaystyle\forall\,\e>0,\ \forall\,0<\omega<\frac{r}2,\ \forall\,t>0,\ \forall\,x\in t B_\omega(\bar x),\vspace{3pt}\\
\qquad\qquad\qquad\displaystyle\rho_\e'(t)+\hat{x-ty}\.y \leq\rho(\e t)+\e+c-r+\frac{4\omega|y|}{r}\leq	c^*(\nu)+\e+\frac{4\omega|y|}{r}.\eaa$$
Using these estimates, together with $\partial_zU<0$ and $\nu A\nu-\zeta A\zeta=(\nu-\zeta)A(\nu+\zeta)$ for any~$\zeta\in\R^N$, we find another constant $C'>0$, depending only on $\|A\|_{W^{1,\infty}}$, $\|q\|_{\infty}$, $r$, $|y|$, and~$N$, such~that
\Fi{reminder2}\baa{l}
\displaystyle\forall\,\e>0,\ \forall\,0<\omega<\frac{r}2,\ \forall\,t>0,\ \forall\,x\in t B_\omega(\bar x),\vspace{3pt}\\
\qquad\qquad\qquad\qquad\displaystyle\mc{R}(t,x)\leq\left(\e+C'\omega+\frac{C}{t}\right)\times\big(|\partial_zU|+|\partial^2_{zz}U|+|\nabla\partial_zU|\big),\eaa
\Ff
where again the function $U(x,z)$ and its derivatives are evaluated at $(x,|x-ty| - \rho_\e(t))$.

We now turn to the function $(t,x)\mapsto w^\e(t,x)=\chi\big(|x-ty| - \rho_\e(t)\big)\vp(x)+1$. We start with the values $(t,x)$ such that $t>0$ and $|x-ty| - \rho_\e(t)>1$. For such values, we have $|x-ty|>\rho_\e(t)+1>1>0$, and we compute
\[\begin{split}
\partial_t w^\e-\dv(A\nabla w^\e)-q\.\nabla w^\e=&-\big(\rho_\e'(t)+\hat{x-ty}\.y\big)\chi'\vp\\
&-\dv(A\nabla \vp)\chi-\hat{x-ty}\,A\,\hat{x-ty}\chi''\vp\\
&-\dv\big(A\hat{x-ty}\big)\chi'\vp-2\chi'\hat{x-ty}\,A\nabla\vp \\
&-(q\.\nabla \vp)\chi-q\.\hat{x-ty}\chi'\vp,
\end{split}\]
where $\chi$ and its derivatives are evaluated at $|x-ty| - \rho_\e(t)$. Using that $\chi(z)=e^{-\lambda z}$ for $z\geq1$, and that $\vp$ solves \eqref{evp} with $e=\nu$, one derives, for $t>0$ and $|x-ty| - \rho_\e(t)>1$,
\[\partial_t w^\e-\dv(A\nabla w^\e)-q\.\nabla w^\e=\left(\lambda\big(\rho_\e'(t)+\hat{x-ty}\.y\big)-k_\nu(\lambda)\right)\chi\vp+\t{\mc{R}}(t,x)\chi,\]
with 
\[\begin{split}
\t{\mc{R}}(t,x):= &-\left[\dv\left(A\big(\nu-\hat{x-ty}\big)\right)+q\.\big(\nu-\hat{x-ty}\big)\right]\lambda\vp\\
&-2\lambda\big(\nu-\hat{x-ty}\big)A\nabla\vp+\big(\nu A\nu -\hat{x-ty}\,A\,\hat{x-ty}\big)\lambda^2\vp.
\end{split}\]	
On the one hand, recalling that we have \eqref{reminder0}, and also $\rho_\e'\geq0$ and $r\leq c/2$, we get, always for $t>0$ and $|x-ty| - \rho_\e(t)>1$, 
\[\begin{split}
\partial_t w^\e-\dv(A\nabla w^\e)-q\.\nabla w^\e &\geq\Big(r\lambda-\frac{4\omega|y|\lambda}r-k_\nu(\lambda)\Big)\chi\vp + \t{\mc{R}}(t,x)\chi\\
& \geq\Big(\frac{r\lambda}{2}-\frac{4\omega|y|\lambda}r\Big)\chi\vp + \t{\mc{R}}(t,x)\chi,
\end{split}\]
where, for the second inequality, we have used \eqref{klambda<}. On the other hand, the term $\t{\mc{R}}$ can be estimated exactly at the same way as $\mc{R}$ before, namely, one gets
\Fi{tR}
\forall\,\e>0,\ \forall\,0<\omega<\frac{r}{2},\ \forall\,t>0,\ \forall\,x\in t B_\omega(\bar x),\quad|\t{\mc{R}}(t,x)|\leq\t C\omega+\frac{\t C}{t},
\Ff
for some constant $\t C>0$ depending only on $\|A\|_{W^{1,\infty}}$, $\|q\|_{\infty}$, $r$, $|y|$, $N$, $\|\vp\|_{W^{1,\infty}}$, and $\lambda$. Summing up, we find the following estimate, up to increasing the constant $\t C$:
$$\baa{l}
\displaystyle\forall\,\e>0,\ \forall\,0<\omega<\frac{r}{2},\ \forall\,t>0,\ \forall\,x\in t B_\omega(\bar x),\vspace{3pt}\\
\qquad\displaystyle|x-ty|-\rho_\e(t)>1\ \Longrightarrow\ \partial_t w^\e-\dv(A\nabla w^\e)-q\.\nabla w^\e\geq\Big(\frac{r\lambda\vp}2 - \t C\omega-\frac{\t C}{t}\Big)\chi.\eaa$$

We gather together the above estimate for $w^\e$ with the estimate \eqref{reminder2} on the reminder~$\mc{R}$ of the expression~\eqref{ve} and with the estimate~\eqref{U<chi}. We then get that
$$\baa{l}
\displaystyle\forall\,0<\mu<\frac14,\ \forall\,\e>0,\ \forall\,0<\omega<\frac{r}{2},\ \forall\,t>0,\ \forall\,x\in t B_\omega(\bar x),\ \ |x-ty|-\rho_\e(t)>1\ \Longrightarrow\vspace{3pt}\\
\quad\displaystyle\partial_t v^\e_\mu\!-\!\dv(A\nabla v^\e_\mu)\!-\!q\.\nabla v^\e_\mu\leq f(x,v^\e)\!+\!\Big(K\e\!+\!KC'\omega\!+\!\frac{KC}{t}\!-\!\mu\Big(\frac{r\lambda\vp}2\!-\!\t C\omega\!-\!\frac{\t C}{t}\Big)\Big)\chi.\eaa$$
Since $r>0$, $\lambda>0$ and $\min_{\R^N}\varphi>0$, it follows that, for any given $0<\mu<1/4$, there exist $\e_\mu^+>0$ small enough, $\omega_\mu^+\in(0,r/2)$ small enough, and $T_\mu^+>0$ large enough, such that
$$\baa{l}
\displaystyle\forall\,0<\e<\e_\mu^+,\ \forall\,t\geq T_\mu^+,\ \forall\,x\in t B_{\omega_\mu^+}(\bar x),\vspace{3pt}\\
\displaystyle\qquad\qquad\qquad|x-ty|-\rho_\e(t)>1\ \Longrightarrow\ \partial_t v^\e_\mu-\dv(A\nabla v^\e_\mu)-q\.\nabla v^\e_\mu < \ f(x,v^\e).\eaa$$

Let us consider $\delta\in(0,1/2)$ from Hypothesis~\ref{hyp:comb-bi}, remember~\eqref{Ue}, and call
$$m:=\min_{x\in\R^N}U(x,1)\,>0,\qquad \delta':=\frac{\min(m,\delta)}2\,>0.$$
Notice that $\delta'/2<1/4$. For any $0<\mu<\delta'/2$ and $\e>0$, if $v^\e_\mu(t,x)\in[0,\delta')$ for some $t>0$ and $x\in\R^N$, then $0<v^\e(t,x)\leq v^\e_\mu(t,x)+2\mu<2\delta'=\min(m,\delta)$, which implies at once that $|x-ty|-\rho_\e(t)>1$ (by the choice of $m$ and $\partial_zU<0$) and also that $f(x,v^\e(t,x))\leq f(x,v^\e_\mu(t,x))$ owing to Hypothesis~\ref{hyp:comb-bi}. Therefore, we derive
\Fi{strict-sub+}
\begin{array}{l}
\displaystyle\forall\,0<\mu<\frac{\delta'}2,\ \forall\,0<\e<\e^+_\mu,\ \forall\,t\geq T_\mu^+,\ \forall\,x\in t B_{\omega_\mu^+}(\bar x),\\[5pt]
\qquad\qquad\qquad v^\e_\mu(t,x)\in[0,\delta')\ \implies\  \partial_t v^\e_\mu-\dv(A\nabla v^\e_\mu)-q\.\nabla v^\e_\mu <f(x,v^\e_\mu).
\end{array}
\Ff
{\it{For the rest of the proof, we restrict to $0<\mu<\delta'/2$.}}

We now turn to the region where $t>0$ and $|x-ty| - \rho_\e(t)<-1$. There, one has $v^\e_\mu(t,x)=v^\e(t,x)-\mu$, whence, by~\eqref{ve} and \eqref{reminder2},
\Fi{strict-sub-1}
\begin{array}{l}
\displaystyle\forall\,0<\mu<\frac{\delta'}2,\ \forall\,\e>0,\ \forall\,0<\omega<\frac{r}{2},\ \forall\,t>0,\ \forall\,x\in t B_{\omega}(\bar x),\vspace{3pt}\\
\qquad|x-ty| - \rho_\e(t)<-1\ \Longrightarrow\\[5pt]
\qquad\qquad\displaystyle\partial_t v^\e_\mu\!-\!\dv(A\nabla v^\e_\mu)\!-\!q\.\nabla v^\e_\mu \leq f(x,v^\e)\!+\!\Big(\e\!+\!C'\omega\!+\!\frac{C}{t}\Big)\|\partial_zU\|_{W^{1,\infty}}.
\end{array}
\Ff    
Reminding~\eqref{Ue}, we then set
$$M:=\max_{x\in\R^N}U(x,-1)\ <1,\qquad \delta'':=\min(1-M,\delta)>0.$$
For any $0<\mu<\delta'/2$ (which is less than $\delta$) and $\e>0$, if $v^\e_\mu(t,x)\in(1-\delta'',1]$ for some $t>0$ and $x\in\R^N$, then $1>v^\e(t,x)> v^\e_\mu(t,x)>\max(M,1-\delta)$, which implies at once that $|x-ty|-\rho_\e(t)<-1$ (because $v^\e(t,x)>M$ and $\partial_zU<0$), whence $v^\e_\mu(t,x)=v^\e(t,x)-\mu$, and moreover
$$f(x,v^\e_\mu(t,x))-f(x,v^\e(t,x))\big)\geq\min_{z\in\R^N,\ s\in[1-\delta,1-\mu]}\big(f(z,s)-f(z,s+\mu)\big).$$
Notice that the right-hand side of the previous inequality is a positive constant (depending on $\mu$) thanks to Hypothesis~\ref{hyp:comb-bi}. Using this in the estimate \eqref{strict-sub-1}, for any $0<\mu<\delta'/2$, we can then find $\e_\mu^->0$ small enough, $\omega_\mu^-\in(0,r/2)$ small enough, and $T_\mu^->0$ large enough, such that
\Fi{strict-sub-}
\begin{array}{l}
\displaystyle\forall\,0<\mu<\frac{\delta'}2,\ \forall\,0<\e<\e_\mu^-,\ \forall\,t\geq T_\mu^-,\ \forall\,x\in t B_{\omega_\mu^-}(\bar x),\\[5pt]
\qquad\qquad v^\e_\mu(t,x)\in(1-\delta'',1]\ \implies\ \partial_t v^\e_\mu-\dv(A\nabla v^\e_\mu)-q\.\nabla v^\e_\mu <f(x,v^\e_\mu).
\end{array}
\Ff    

Finally, if we call $\e_\mu:=\min(\e_\mu^+,\e_\mu^-)>0$, $\omega_\mu:=\min(\omega_\mu^+,\omega_\mu^-)\in(0,r/2)$, and $T_\mu:=\max(T_\mu^+,T_\mu^-)>0$, we get, from~\eqref{strict-sub+} and~\eqref{strict-sub-},
\Fi{strict-sub0}
\begin{array}{l}
\displaystyle\forall\,0<\mu<\frac{\delta'}{2},\ \forall\,0<\e<\e_\mu,\ \forall\,t\geq T_\mu,\ \forall\,x\in t B_{\omega_\mu}(\bar x),\\[5pt]
\qquad v^\e_\mu(t,x)\in[0,\delta')\!\cup\!(1\!-\!\delta'',1]\ \implies\ \partial_t v^\e_\mu\!-\!\dv(A\nabla v^\e_\mu)\!-\!q\.\nabla v^\e_\mu\!<\!f(x,v^\e_\mu).
\end{array}
\Ff
    
\subsubsection*{Conclusion}

Take any $0<\mu<\delta'/2\ (<1/4)$. Let $\e_\mu''>0$ be the quantity in~\eqref{contactpoint} and $\e_\mu>0$, $\omega_\mu\in(0,r/2)$, $T_\mu>0$ be the ones in~\eqref{strict-sub0}. By~\eqref{contactpoint0}, we can take $\e\in(0,\e_\mu'')$ small enough in such a way that the time-space contact point $(t^\e_\mu,x^\e_\mu)$  in~\eqref{contactpoint} satisfies $t^\e_\mu\geq T_\mu$ and $x^\e_\mu\in t^\e_\mu B_{\omega_\mu}(\bar x)$. If in addition we require $\e<\e_\mu$, then we necessarily have that
$$v^\e_\mu(t^\e_\mu,x^\e_\mu)\in[\delta',1-\delta''],$$ 
because otherwise, owing to~\eqref{contactpoint} and \eqref{strict-sub0}, the solution $u$ would touch from above the function $v^\e_\mu$ in  a point where the latter is a strict subsolution, which is impossible. Together with~\eqref{contactpoint0}, this shows that we can choose $\e=\e^\mu\in(0,\min(\e''_\mu,\e_\mu,\mu))$ small enough so that $v^{\e^\mu}_\mu(t^{\e^\mu}_\mu,x^{\e^\mu}_\mu)\in[\delta',1-\delta'']$, $t^{\e^\mu}_\mu\ge1/\mu$, and $|x^{\e^\mu}_\mu/t^{\e^\mu}_\mu-\bar{x}|\le\mu$. Hence,
$$\lim_{\mu\to0^+}\e^\mu=0,\qquad\lim_{\mu\to0^+}t^{\e^\mu}_\mu=+\infty,\qquad\lim_{\mu\to0^+}\frac{x^{\e^\mu}_\mu}{t^{\e^\mu}_\mu}= \bar x.$$

Next, again for any $0<\mu<\delta'/2$, we take $\zeta_\mu\in\Z^N$ such that $x_\mu:=x^{\e^\mu}_\mu-\zeta_\mu\in[0,1]^N$, and we call 
$$t_\mu:=t^{\e^\mu}_\mu,\ \ u_\mu(t,x):=u(t+t_\mu,x+\zeta_\mu),\ \ v_\mu(t,x):=v^{\e^\mu}_\mu(t+t_\mu,x+\zeta_\mu),$$
for $(t,x)\in[-t_\mu,+\infty)\times\R^N$. With these notations, we have
\Fi{contactpoint0*}
\lim_{\mu\to0^+}\e^\mu=0,\qquad\lim_{\mu\to0^+}t_\mu=+\infty,\qquad\lim_{\mu\to0^+}\frac{\zeta_\mu}{t_\mu}=\bar x,
\Ff	
as well as, by \eqref{contactpoint},
\Fi{u>v}
\forall\,(t,x)\in[-1,0]\times\R^N,\quad u_\mu(t,x)\geq v_\mu(t,x)
\Ff
and
\Fi{u=v}
u_\mu(0,x_\mu)=v_\mu(0,x_\mu)\in[\delta',1-\delta''].
\Ff
Furthermore, reminding the definition of the functions $v^\e_\mu$ and using the periodicity of~$U$ in the first variable, we derive
$$v_\mu(t,x)=U\big(x, |x+\zeta_\mu-(t+t_\mu)y| - \rho_{\e^\mu}(t+t_\mu)\big)-\mu O_\mu(t,x)$$
for all $(t,x)\in[-t_\mu,+\infty)\times\R^N$, with the function $O_\mu$ satisfying $1\leq O_\mu\leq 2$ in~$[-t_\mu,+\infty)\times\R^N$.
	
We claim that the level sets of $(t,x)\mapsto|x+\zeta_\mu-(t+t_\mu)y| - \rho_{\e^\mu}(t+t_\mu)$ converge locally uniformly to some shifting hyperplanes as $\mu\to0^+$. Indeed, on the one hand, since $t_\mu\to+\infty$ as $\mu\to0^+$, we find
\[\begin{split}
|x+\zeta_\mu-(t+t_\mu)y|-|\zeta_\mu-t_\mu y| &=\frac{|x-ty|^2+2 (x-ty)\.(\zeta_\mu-t_\mu y)}{|x+\zeta_\mu-(t+t_\mu)y|+|\zeta_\mu-t_\mu y|}\\
&=\frac{|x-ty|^2/t_\mu+2 (x-ty)\.(\zeta_\mu/t_\mu-y)}{|(x-ty)/t_\mu + \zeta_\mu/t_\mu-y|+|\zeta_\mu/t_\mu - y|}\\
&\displaystyle\mathop{\longrightarrow}_{\mu\to0^+}\ (x-ty)\.\nu=x\.\nu-(\bar x\.\nu-r)t,
\end{split}\]
locally uniformly with respect to $(t,x)\in\R\times\R^N$. On the other hand, since by~\eqref{u=v},
$$v_\mu(0,x_\mu)=U\big(x_\mu, |x^{\e^\mu}_\mu-t_\mu y| - \rho_{\e^\mu}(t_\mu)\big)-\mu O_\mu(t,x)\in[\delta',1-\delta''],$$
and $U(\.,-\infty)\equiv1$, $U(\.,+\infty)\equiv0$, we deduce that $|x^{\e^\mu}_\mu-t_\mu y| - \rho_{\e^\mu}(t_\mu)$ remains bounded as $\mu\to0^+$, and thus the same is true for $|\zeta_\mu-t_\mu y| - \rho_{\e^\mu}(t_\mu)$. One then gets from~\eqref{contactpoint0*}
$$\lim_{\mu\to0^+}\frac{\rho_{\e^\mu}(t_\mu)}{t_\mu}=\lim_{\mu\to0^+}\Big|\frac{\zeta_\mu}{t_\mu}-y\Big|=|\bar x-y|=r.$$	
Recalling the definition of $\rho_{\e^\mu}$, this condition rewrites as
$$\lim_{\mu\to0^+}\Big(\frac{1}{\e^\mu t_\mu}\int_0^{\e^\mu t_\mu}\rho(s)ds+\e^\mu\Big)=r,$$
but then, since $\rho$ is strictly increasing with $\rho(+\infty)=r$ (and $\e^\mu\to0$ as $\mu\to0^+$) we necessarily have that $\e^\mu t_\mu\to+\infty$ as $\mu\to0^+$. This allows us to conclude that
$$\rho_{\e^\mu}(t+t_\mu)-\rho_{\e^\mu}(t_\mu)=\frac{1}{\e^\mu}\int_{\e^\mu t_\mu}^{\e^\mu t_\mu+\e^\mu t}\rho(s)ds+\e^\mu t\to rt\as \mu\to0^+,$$
locally uniformly with respect to $t\in\R$. Summing up, we have that
\Fi{moving-plane}
|x+\zeta_\mu-(t+t_\mu)y| - \rho_{\e^\mu}(t+t_\mu)=x\.\nu-\bar x\.\nu t+|\zeta_\mu-t_\mu y|-\rho_{\e^\mu}(t_\mu)+o(1)\as\mu\to0^+,
\Ff
locally uniformly with respect to $(t,x)\in\R\times\R^N$. Therefore, considering a sequence~$\seq{\mu}$ converging to $0^+$ along which $|\zeta_{\mu_n}-t_{\mu_n} y|-\rho_{\e^{\mu_n}}(t_{\mu_n})$ (which is bounded) converges towards some limit, that we call $Z\in\R$, we eventually get
$$|x+\zeta_{\mu_n}-(t+t_{\mu_n})y| - \rho_{\e^{\mu_n}}(t+t_{\mu_n})\to x\.\nu-\bar x\.\nu t+Z\as n\to+\infty,$$
locally uniformly with respect to $(t,x)\in\R\times\R^N$. Recalling that $\bar x\.\nu=c$, we have thereby shown that
\be\label{localconv}
v_{\mu_n}(t,x)\to U(x, x\.\nu-ct+Z)\as n\to+\infty,
\ee
locally uniformly with respect to $(t,x)\in\R\times\R^N$. Finally, by parabolic estimates, the functions $u_{\mu_n}$ converge locally uniformly as $n\to+\infty$ (up to subsequences) to an entire solution $\t u$ of~\eqref{general}, and by \eqref{u>v}-\eqref{u=v} it holds that
\be\label{eqZ}
\min_{t\in[-1,0],\ x\in\R^N}\big(\t u(t,x)- U(x, x\.\nu-ct+Z)\big)=\t u(0,\t x)- U(\t x, \t x\.\nu+Z)=0,
\ee
where $\t x$ is the limit (of a convergent subsequence of) $(x_{\mu_n})_{n\in\N}$ in $[0,1]^N$. Since by Proposition~\ref{pro:U}, $U(x, x\.\nu-ct+Z)$ is a subsolution to~\eqref{general} because $c\leq c^*(\nu)$ and~$\partial_zU<0$ (and it is a strict subsolution if $c< c^*(\nu)$), the strong maximum principle yields
\be\label{eqZ2}
\t u(t,x)\equiv U(x, x\.\nu-ct+Z)\ \hbox{ for all $(t,x)\in[-1,+\infty)\times\R^N$},
\ee
and moreover $c=c^*(\nu)$. This shows that $\phi_\nu(-Z/c^*(\nu)+\theta,\cdot)\in\O_{\hat{\bar x}}(u)$ for all ${\theta\ge-1}$. Furthermore, the definition of $\O_{\hat{\bar x}}(u)$ and the local uniform convergence in~\eqref{localconv} immediately yield that $\phi_\nu(s,\cdot+\xi)\in\O_{\hat{\bar x}}(u)$ for each $(s,\xi)\in[-Z/c^*(\nu)-1,+\infty)\times\R^N$. But since $\phi_\nu(s,x)=\phi_\nu(s+k\cdot\nu/c^*(\nu),x+k)$ for all $k\in\Z^N$ and $(s,x)\in\R\times\R^N$ by definition of $\phi_\nu$, one finally gets that $\phi_\nu(s,\cdot+\xi)\in\O_{\hat{\bar x}}(u)$ for each 
$(s,\xi)\in\R\times\R^N$. This concludes the proof of Lemma~\ref{lem:interior}.\hfill$\Box$


\subsection{Proof of Lemma~\ref{lem:exterior}}\label{sec33}
	
The proof follows the same scheme as the one of Lemma~\ref{lem:interior}. Now $r>0$ and $\nu\in\Sph$ are given by the exterior ball condition at~$\bar x$. We assume that
$$c:=\bar x\.\nu\geq c^*(\nu),$$
and we need to show that $c=c^*(\nu)$ and moreover that~$\phi_\nu(t,\cdot+y)\in\O_{\hat{\bar x}}(u)$, for any $(t,y)\in\R\times\R^N$. 
	
\subsubsection*{Construction of the ``almost radial'' supersolution}

Up to reducing $r>0$, we can assume without loss of generality that the following stronger exterior ball condition~holds:
\Fi{strong-sphere-ext}
\ol{B_r(y)}\setminus\{\bar x\}\subset\R^N\setminus\ol\W,\quad\text{ with }\; y:=\bar x+r\nu.
\Ff
In addition, since $c=\bar x\.\nu\ge c^*(\nu)>0$, we can further assume without loss of generality that $r\leq c/2$. We then consider the same $C^1$ function $\rho:[0,+\infty)\to\R$ satisfying~\eqref{defrho} as in the proof of Lemma~\ref{lem:interior}, and $\rho_\e$ as in~\eqref{defrhoeps} for $\e>0$. For $\e>0$, we then define
$$v^\e(t,x):=U\big(x, \rho_\e(t)-|x-ty|\big),\ \ (t,x)\in[0,+\infty)\times\R^N,$$
and
$$w^\e(t,x):=\chi\big(\rho_\e(t)-|x-ty|\big)\,\vp(x)+1,\ \ (t,x)\in[0,+\infty)\times\R^N,$$
where $\chi:\R\to\R$ is a $C^2$ non-negative function satisfying~\eqref{defchi}, but with $0<\lambda<\lambda_0$ satisfying this time 
\Fi{lambda<<}
k_\nu(\lambda)\leq \frac{c^*(\nu)}2\,\lambda,
\Ff
and $\vp$ is the periodic eigenfunction of~\eqref{evp} associated with $k_\nu(\lambda)$, normalized by $(\max_{\R^N}\varphi)\,(\max_{\R}\chi)=1$, so that~\eqref{weps12} still holds. We recall that $\lambda<\lambda_0$ yields~\eqref{U<chi}, for some constant $K>0$ independent of $\e>0$.
	
Then, for $\e>0$ and $\mu>0$, we increase the function $v^\e$ by setting
$$v^\e_\mu(t,x):=v^\e(t,x)+\mu w^\e(t,x),\ \ (t,x)\in[0,+\infty)\times\R^N.$$
This function satisfies, on the one hand,
$$\inf_{[0,+\infty)\times\R^N} v^\e_\mu=\mu>0,$$
and furthermore, since $U(x,z)\to1$ as $z\to-\infty$ uniformly in $x\in\R^N$, there exists $R_\mu>0$, depending on $\mu>0$ but not on $\e>0$, such that
\Fi{ve>1}
\forall\,\mu>0,\ \forall\,\e>0,\ \forall\,t\ge0,\ \forall\,x\in \R^N,\ \ \rho_\e(t)-|x-ty|\leq -R_\mu\ \Longrightarrow\ 
v^\e_\mu(t,x)>1.
\Ff
On the other hand, if $0<\mu<1/4$ then $\mu w^\e<1/2$ in $[0,+\infty)\times\R^N$ and, since $U(x,z)\to0$ as $z\to+\infty$ uniformly in $x\in\R^N$, we can therefore find another quantity $L\in\R$, independent of $\mu\in(0,1/4)$ and $\e>0$, such that
\Fi{ve<delta}
\forall\,0<\mu<\frac14,\ \forall\,\e>0,\ \forall\,t\ge0,\ \forall\,x\in \R^N,\ \ \rho_\e(t)-|x-ty|\geq L\ \Longrightarrow\ v^\e_\mu(t,x)<\frac34\,.
\Ff
{\it{For the rest of the proof, we restrict to $0<\mu<1/4$, so that both~\eqref{ve>1} and~\eqref{ve<delta} hold.}}

\subsubsection*{The contact points with $u$}
 
On the one hand, by \eqref{strong-sphere-ext} and \eqref{ass-cpt}, for each $\mu\in(0,1/4)$, there exists $\tau_\mu\geq0$, independent of $\e>0$, such that
\Fi{u<m}\baa{l}
\displaystyle\forall\,0<\mu<\frac14,\ \forall\,\e>0,\ \forall\,t\geq \tau_\mu,\ \forall\,x\in \R^N,\vspace{3pt}\\
\qquad\qquad\qquad\qquad\qquad\displaystyle\ \ |x-ty|\leq\frac34 rt\ \Longrightarrow\ u(t,x)<\mu=\inf_{[0,+\infty)\times\R^N}\,v^\e_\mu.\eaa
\Ff
On the other hand, consider as in the proof of Lemma~\ref{lem:interior} the following function:
$$\gamma(\e,t):=\frac34 rt-\rho_\e(t),\ \ \e>0,\ t\in[0,+\infty),$$
extended by $\gamma(0,t)=rt/4$. We have seen before that, for every $\mu\in(0,1/4)$, there exists $\e'_\mu>0$ small enough so that, for any $\e\in(0,\e'_\mu)$, there holds
$$\big\{ t>0:\gamma(\e,t)\geq R_\mu\big\}=[\ul\tau^\e_\mu,\ol\tau^\e_\mu],$$
for some $0<\ul\tau^\e_\mu<\ol\tau^\e_\mu$ satisfying
\be\label{limits5}
\lim_{\e\to0^+}\ul\tau^\e_\mu=4\frac{R_\mu}r,\qquad \lim_{\e\to0^+}\ol\tau^\e_\mu=+\infty.
\ee
As a consequence, owing to~\eqref{ve>1}, one derives
$$\baa{l}
\displaystyle\forall\,0<\mu<\frac14,\ \forall\,0<\e<\e'_\mu,\ \forall\,t\in[\ul\tau^\e_\mu,\ol\tau^\e_\mu],\ \forall\,x\in \R^N,\vspace{3pt}\\
\displaystyle\qquad\qquad\qquad\qquad\qquad\qquad|x-ty|\ge\frac34 rt\ \Longrightarrow\ v^\e_\mu(t,x)>1\geq u(t,x).\eaa$$
Then, taking $\e_\mu''\in(0,\e'_\mu)$ small enough so that $\ol\tau^\e_\mu-1>\max\{\ul\tau^\e_\mu,\tau_\mu\}$ for all $\e\in(0,\e_\mu'')$, with $\tau_\mu$ given by~\eqref{u<m}, eventually yields
\be\label{ineq5}
\forall\,0<\mu<\frac14,\ \forall\,0<\e<\e_\mu'',\quad u<v^\e_\mu\ \text{ in }[\ol\tau^\e_\mu-1,\ol\tau^\e_\mu]\times\R^N.
\ee
	
We now show that, for given $0<\mu<1/4$ and $\e>0$, the inequality $v^\e_\mu< u$ holds at some point, for~$t$ sufficiently large. Specifically, on the one hand, since $\rho_\e(t)/t\to r+\e$ as $t\to+\infty$, one derives from \eqref{ve<delta} that, for $t$ sufficiently large,
$$\sup_{x\in B_{r+\e/2}(y)}v^\e_\mu(t,tx)<\frac34\,.$$
On the other hand, since $\bar x\in\partial B_r(y)\cap\partial\W$, there exists $x\in B_{r+\e/2}(y)\cap\W$, hence property~\eqref{ass-cpt} implies that $u(t,tx)\to1$ as $t\to+\infty$. We conclude that the quantity $\inf_{\R^N}(v^\e_\mu-u)(t,\cdot)$ is negative for $t$ large enough. Thus, for $0<\mu<1/4$ and $0<\e<\e_\mu''$, we can define
$$t^\e_\mu:=\inf\Big\{t>\ol\tau^\e_\mu\ : \ \inf_{\R^N}(v^\e_\mu-u)(t,\cdot)<0\Big\}\ \in[\ol\tau^\e_\mu,+\infty).$$
Using properties~\eqref{ve>1} and~\eqref{ineq5}, one readily shows as in the proof of~\eqref{contactpoint} that $t^\e_\mu>\ol\tau^\e_\mu$ and
\Fi{contactpointSUP}
\forall\,0<\mu<\frac14,\ \forall\,0<\e<\e_\mu'',\quad\min_{[t^\e_\mu-1,t^\e_\mu] \times\R^N}(v^\e_\mu-u)=(v^\e_\mu-u)(t^\e_\mu,x^\e_\mu)=0,
\Ff
with $x^\e_\mu$ satisfying
$$|x^\e_\mu-t^\e_\mu y|<\rho_\e(t^\e_\mu)+R_\mu.$$
Then, since $\rho_\e(t)\leq(r+\e)t$ for all $t\ge0$ and $t^\e_\mu\geq\ol\tau^\e_\mu\to+\infty$ as $\e\to0^+$ by~\eqref{limits5}, we eventually infer that $\limsup_{\e\to0^+}|x^\e_\mu/t^\e_\mu- y|\leq r$. But then, we necessarily have that $x^\e_\mu\to\bar x$ as $\e\to0^+$, because otherwise the inclusion \eqref{strong-sphere-ext} and the second limit in~\eqref{ass-cpt} would lead to a contradiction with $u(t^\e_\mu,x^\e_\mu)=v^\e_\mu(t^\e_\mu,x^\e_\mu)\geq\mu>0$. Therefore, we have
\Fi{contactpoint0SUP}
\forall\,0<\mu<\frac14,\ \ \ \ \lim_{\e\to0^+}t^\e_\mu=+\infty,\quad\lim_{\e\to0^+}\frac{x^\e_\mu}{t^\e_\mu}= \bar x.
\Ff

\subsubsection*{The strict supersolution property}

We analyze the equation satisfied by the function~$v^\e_\mu$ at time $t>0$ in the vicinity of the point $t\bar x$. More precisely, we will focus on the ball~$t B_\omega(\bar x)$ with given~$\omega\in(0,r/2)$, so in particular $x\neq ty$ for every $x\in tB_\omega(\bar x)$. In order to exploit the computation already made in the proof of Lemma~\ref{lem:interior}, observe that $v^\e$ and $w^\e$ are defined exactly as in that proof, but with $U(x,|x-ty|-\rho_\e(t))$ and $\chi(|x-ty|-\rho_\e(t))$ now replaced by $U(x,\rho_\e(t)-|x-ty|))$ and $\chi(\rho_\e(t)-|x-ty|)$. One then derives
\Fi{veSUP}
\partial_t v^\e-\dv(A\nabla v^\e)-q\.\nabla v^\e = f(x,v^\e)+\mc{R}(t,x),
\Ff
where
$$\baa{rcl}
\mc{R}(t,x) & \!\!\!:=\!\!\! & \displaystyle\left[c^*(\nu)+\rho_\e'(t)+\hat{x-ty}\!\.\!y+\dv\left(A\big(\nu+\hat{x-ty}\big)\right)+q\.\big(\nu+\hat{x-ty}\big)\right]\partial_zU\vspace{3pt}\\
& \!\!\!\!\!\! & \displaystyle+2\Big(\nu+\hat{x-ty}\Big)A\nabla\partial_z U+\left(\nu A\nu -\hat{x-ty}A\hat{x-ty}\right)\partial^2_{zz}U\eaa$$	
and where $U$ is treated as $U=U(x,z)$ and eventually evaluated at $z=\rho_\e(t)-|x-ty|$.	We now deal with the reminder $\mc{R}(t,x)$ using the estimates obtained in the proof Lemma~\ref{lem:interior}. To do so, we write $y=\bar x-r\t\nu$ with $\t\nu=-\nu$, whence the estimate \eqref{tonu} holds with $\nu$ replaced by $\t\nu$. From that, and from inequalities $c\geq c^*(\nu)$ and $\rho\leq r$, one~gets
\Fi{reminder1SUP}\baa{l}
\displaystyle\forall\,0<\omega<\frac{r}{2},\ \forall\,t>0,\ \forall\,x\in t B_\omega(\bar x),\vspace{3pt}\\
\displaystyle\qquad\qquad\rho_\e'(t)+\hat{x-ty}\.y \leq\rho(\e t)+\e-c-r+\frac{4\omega|y|}{r}\leq-c^*(\nu)+\e+\frac{4\omega|y|}{r}.\eaa
\Ff	
For the other terms in the above expression of $\mc{R}$, observe that, after replacing $\nu$ with~$-\t\nu$, they are the same as in the expression of $\mc{R}$ in the proof of Lemma~\ref{lem:interior}, up to a sign. This leads to the following estimate (recall that $\partial_zU<0$):
\Fi{reminder2>}\baa{l}
\displaystyle\forall\,\epsilon>0,\ \forall\,0<\omega<\frac{r}{2},\ \forall\,t>0,\ \forall x\in t B_\omega(\bar x),\vspace{3pt}\\
\displaystyle\qquad\qquad\qquad\mc{R}(t,x)\geq-\left(\e+C'\omega+\frac{C}{t}\right)\times\big(|\partial_zU|+|\partial^2_{zz}U|+|\nabla\partial_zU|\big),\eaa
\Ff
where again $U$ and its derivatives are evaluated at $(x,\rho_\e(t)-|x-ty|)$, with a constant $C>0$ depending only on $\|A\|_{W^{1,\infty}}$, $r$, $N$, and another constant $C'>0$ depending only on the same terms as well as~$\|q\|_{\infty}$ and $|y|$.

Let us turn to the function $(t,x)\mapsto w^\e(t,x)=\chi(\rho_\e(t)-|x-ty|)\,\vp(x)+1$. Here again we exploit the computation of the proof of Lemma~\ref{lem:interior}, taking into account the replacement of $\chi(|x-ty|-\rho_\e(t))$ with $\chi(\rho_\e(t)-|x-ty|)$. In particular, in the region $\rho_\e(t)-|x-ty|>1$, one uses the previous computation for the case $|x-ty|-\rho_\e(t)>1$, with $\chi(|x-ty|-\rho_\e(t))=e^{-\lambda(|x-ty|-\rho_\e(t))}$ now replaced by $\chi(\rho_\e(t)-|x-ty|)=e^{-\lambda(\rho_\e(t)-|x-ty|)}$. Then one just needs to replace $\lambda$ with $-\lambda$. Observe that such replacement in the equation satisfied by the eigenfunction $\vp$ is equivalent to the replacement of $\nu$ by $\t\nu=-\nu$. As a consequence, one derives, for $t>0$ and $\rho_\e(t)-|x-ty|>1$ with $x\neq ty$,
$$\partial_t w^\e-\dv(A\nabla w^\e)-q\.\nabla w^\e=\left(-\lambda\big(\rho_\e'(t)+\hat{x-ty}\.y\big)-k_\nu(\lambda)\right)\chi\vp+\t{\mc{R}}(t,x)\chi,$$
where $\chi$ is evaluated at $\rho_\e(t)-|x-ty|$, and
\[\begin{split}
\t{\mc{R}}(t,x):= &-\left[\dv\left(A\big(\nu+\hat{x-ty}\big)\right)+q\.\big(\nu+\hat{x-ty}\big)\right]\lambda\vp\\
&-2\lambda\big(\nu+\hat{x-ty}\big)A\nabla\vp+\big(\nu A\nu -\hat{x-ty}\,A\,\hat{x-ty}\big)\lambda^2\vp.
\end{split}\]	
Like $\mc{R}$, the function $\t{\mc{R}}$ satisfies the same estimate as in the previous proof, i.e.~\eqref{tR}. Then, using also \eqref{reminder1SUP} and \eqref{lambda<<} one gets
$$\baa{l}
\displaystyle\forall\,\e>0,\ \forall\,0<\omega<\frac{r}{2},\ \forall\,t>0,\ \forall\,x\in t B_\omega(\bar x),\ \ \rho_\e(t)-|x-ty|>1\ \Longrightarrow\vspace{3pt}\\
\displaystyle\qquad\qquad\qquad\partial_t w^\e-\dv(A\nabla w^\e)-q\.\nabla w^\e\geq\Big(\lambda \big(\frac{c^*(\nu)}2-\e\big)\vp- \t C'\omega-\frac{\t C'}{t}\Big)\chi,\eaa$$
where $\chi$ is evaluated at $\rho_\e(t)-|x-ty|$, and $\t C'>0$ is a positive constant depending only on $\|A\|_{W^{1,\infty}}$, $\|q\|_{\infty}$, $r$, $|y|$, $N$, $\|\vp\|_\infty$ and $\lambda$.

We gather together the above estimate with the estimate \eqref{reminder2>} on the reminder~$\mc{R}$ of the expression~\eqref{veSUP}. Recalling also the estimate~\eqref{U<chi} for $U$, we derive
$$\baa{l}
\displaystyle\forall\,0<\mu<\frac14,\ \forall\,\e>0,\ \forall\,0<\omega<\frac{r}{2},\ \forall\,t>0,\ \forall\,x\in t B_\omega(\bar x),\quad\rho_\e(t)-|x-ty|>1\ \Longrightarrow\vspace{3pt}\\
\qquad\qquad\baa{rcl}\partial_t v^\e_\mu-\dv(A\nabla v^\e_\mu)-q\.\nabla v^\e_\mu & \geq & \displaystyle f(x,v^\e)-K\Big(\e+C'\omega+\frac{C}{t}\Big)\chi\vspace{3pt}\\
& & \displaystyle+\mu\Big(\lambda\big(\frac{c^*(\nu)}2-\e\big)\vp - \t C'\omega-\frac{\t C'}{t}\Big)\chi.\eaa\eaa$$
Since $\lambda>0$, $c^*(\nu)>0$ and $\min_{\R^N}\vp>0$, it follows that, for any given $0<\mu<1/4$, there exist $\e_\mu^+>0$ small enough, $\omega_\mu^+\in(0,r/2)$ small enough, and $T_\mu^+>0$ large enough, such that
$$\baa{l}\forall\,0<\e<\e_\mu^+,\ \forall\,t\geq T_\mu^+,\ \forall\,x\in t B_{\omega_\mu^+}(\bar x),\vspace{3pt}\\
\qquad\qquad\qquad\rho_\e(t)-|x-ty|>1\ \Longrightarrow\ \partial_t v^\e_\mu-\dv(A\nabla v^\e_\mu)-q\.\nabla v^\e_\mu > \ f(x,v^\e).\eaa$$

Let us call now
$$m:=\min_{x\in\R^N}U(x,1)>0,\qquad \delta':=\min(m,\delta)>0.$$
For any $0<\mu<1/4$ and $\e>0$, if $v^\e_\mu(t,x)<\delta'$ for some $t>0$ and $x\in\R^N$, then $0<v^\e(t,x)< v^\e_\mu(t,x)<\delta'=\min(m,\delta)$, which implies at once that $\rho_\e(t)-|x-ty|>1$ (by the choice of $m$ and $\partial_zU<0$) and also that $f(x,v^\e_\mu(t,x))\leq f(x,v^\e(t,x))$ owing to Hypothesis~\ref{hyp:comb-bi}. Therefore, we derive
\Fi{strict-sup+}\begin{array}{l}
\displaystyle\forall\,0<\mu<\frac14,\ \forall\,0<\e<\e_\mu^+,\ \forall\,t\geq T_\mu^+,\ \forall\,x\in t B_{\omega_\mu^+}(\bar x),\\[5pt]
\quad\qquad\qquad\qquad v^\e_\mu(t,x)<\delta'\ \implies\ \partial_t v^\e_\mu-\dv(A\nabla v^\e_\mu)-q\.\nabla v^\e_\mu >f(x,v^\e_\mu).
\end{array}
\Ff

Let us turn now to the region where $t>0$ and $\rho_\e(t)-|x-ty|<-1$. There, one has $v^\e_\mu(t,x)=v^\e(t,x)+\mu$, whence, by~\eqref{veSUP} and \eqref{reminder2>},
\Fi{strict-sup-1}
\begin{array}{l}
\displaystyle\!\!\!\!\!\forall\,0\!<\!\mu\!<\!\frac14,\,\forall\,\e\!>\!0,\,\forall\,0\!<\!\omega\!<\!\frac{r}{2},\,\forall\,t\!>\!0,\,\forall\,x\!\in\! t B_{\omega}(\bar x),\ \rho_\e(t)\!-\!|x\!-\!ty|\!<\!-1\,\Longrightarrow\\[5pt]
\qquad\displaystyle\partial_t v^\e_\mu-\dv(A\nabla v^\e_\mu)-q\.\nabla v^\e_\mu \geq f(x,v^\e)-\Big(\e+C'\omega+\frac{C}{t}\Big)\|\partial_zU\|_{W^{1,\infty}}.
\end{array}
\Ff    
We set
$$M:=\max_{x\in\R^N}U(x,-1)<1,\qquad \delta'':=\frac12\min(1-M,\delta)>0.$$
Note that $0<\delta''/2<1/4$. For any $0<\mu<\delta''/2$ and $\e>0$, if $v^\e_\mu(t,x)\in(1-\delta'',1]$ for some $t>0$ and $x\in\R^N$, one has $1>v^\e(t,x)\geq v^\e_\mu(t,x)-2\mu>\max(M,1-\delta)$, which in turn implies at once that $\rho_\e(t)-|x-ty|<-1$ (because $ v^\e(t,x)>M$ and $\partial_zU<0$), whence $v^\e_\mu(t,x)=v^\e(t,x)+\mu$ and moreover
$$\big(f(x,v^\e(t,x))-f(x,v^\e_\mu(t,x))\big)\geq\min_{z\in\R^N,\ s\in[1-\delta,1-\mu]}\big(f(z,s)-f(z,s+\mu)\big).$$
Notice that the right-hand side of the previous inequality is a positive constant (depen\-ding on $\mu$), by Hypothesis~\ref{hyp:comb-bi}. Using this in the estimate \eqref{strict-sup-1}, for any $0<\mu<\delta''/2$, we can then find $\e_\mu^->0$ small enough, $\omega_\mu^-\in(0,r/2)$ small enough, and $T_\mu^->0$ large enough, such that
\Fi{strict-sup-}\begin{array}{l}
\displaystyle\forall\,0<\mu<\frac{\delta''}2,\ \forall\,0<\e<\e_\mu^-,\ \forall\,t\geq T_\mu^-,\ \forall\,x\in t B_{\omega_\mu^-}(\bar x),\vspace{3pt}\\
\qquad\qquad v^\e_\mu(t,x)\in(1-\delta'',1]\ \implies\ \partial_t v^\e_\mu-\dv(A\nabla v^\e_\mu)-q\.\nabla v^\e_\mu >f(x,v^\e_\mu).
\end{array}
\Ff    
   		
Finally, calling $\e_\mu:=\min(\e_\mu^+,\e_\mu^-)>0$, $\omega_\mu:=\min(\omega_\mu^+,\omega_\mu^-)\in(0,r/2)$, and $T_\mu:=\max(T_\mu^+,T_\mu^-)>0$, we get from \eqref{strict-sup+} and~\eqref{strict-sup-} that
\Fi{strict-sub}\begin{array}{l}
\displaystyle\forall\,0<\mu<\frac{\delta''}2,\ \forall\,0<\e<\e_\mu,\ \forall\,t\geq T_\mu,\ \forall\,x\in t B_{\omega_\mu}(\bar x),\\[5pt]
\quad v^\e_\mu(t,x)\in(0,\delta')\cup(1\!-\!\delta'',1]\ \implies\  \partial_t v^\e_\mu\!-\!\dv(A\nabla v^\e_\mu)\!-\!q\.\nabla v^\e_\mu\!>\!f(x,v^\e_\mu).
\end{array}
\Ff

\subsubsection*{Conclusion}

Take $0<\mu<\delta''/2\ (<1/4)$. Let $\e_\mu''>0$ be the quantity in~\eqref{contactpointSUP} and $\e_\mu>0$, $\omega_\mu\in(0,r/2)$ and $T_\mu>0$ be the ones in~\eqref{strict-sub}. By~\eqref{contactpoint0SUP} we can take $\e\in(0,\e_\mu'')$ small enough in such a way that the contact point $(t^\e_\mu,x^\e_\mu)$  in~\eqref{contactpointSUP} satisfies $t^\e_\mu\geq T_\mu$ and $x^\e_\mu\in t^\e_\mu B_{\omega_\mu}(\bar x)$. If in addition we require $\e<\e_\mu$, then we necessarily have that $u(t^\e_\mu,x^\e_\mu)=v^\e_\mu(t^\e_\mu,x^\e_\mu)\in[\delta',1-\delta'']$, because otherwise, owing to~\eqref{contactpointSUP} and \eqref{strict-sub}, the solution $u$ would touch from below the function $v^\e_\mu$ in  a point where the latter is a strict supersolution, which is impossible. Together with~\eqref{contactpoint0SUP}, this shows that we can choose $\e=\e^\mu\in(0,\min(\e_\mu,\e''_\mu,\mu))$ small enough so that $v^{\e^\mu}_\mu(t^{\e^\mu}_\mu,x^{\e^\mu}_\mu)\in[\delta',1-\delta'']$, $t_\mu^{\e^\mu}\ge1/\mu$, and $|x_\mu^{\e^\mu}/t_\mu^{\e^\mu}-\bar{x}|\le\mu$. Hence,
$$\lim_{\mu\to0^+}\e^\mu=0,\qquad\lim_{\mu\to0^+}t^{\e^\mu}_\mu=+\infty,\qquad\lim_{\mu\to0^+}\frac{x^{\e^\mu}_\mu}{t^{\e^\mu}_\mu}= \bar x.$$

Next, for each $0<\mu<\delta''/2$ we take $\zeta_\mu\in\Z^N$ such that $x_\mu:=x^{\e^\mu}_\mu-\zeta_\mu\in[0,1]^N$. We call 
$$t_\mu:=t^{\e^\mu}_\mu,\qquad u_\mu(t,x):=u(t+t_\mu,x+\zeta_\mu),\qquad v_\mu(t,x):=v^{\e^\mu}_\mu(t+t_\mu,x+\zeta_\mu),$$ 
for $(t,x)\in[-t_\mu,+\infty)\times\R^N$. With these notations, we have
\Fi{contactpoint0*SUP}
\lim_{\mu\to0^+}\e^\mu=0,\qquad\lim_{\mu\to0^+}t_\mu=+\infty,\qquad\lim_{\mu\to0^+}\frac{\zeta_\mu}{t_\mu}=\bar x.
\Ff	
as well as, by \eqref{contactpointSUP},
\Fi{u<v}
\forall\,(t,x)\in[-1,0]\times\R^N,\quad u_\mu(t,x)\leq v_\mu(t,x)
\Ff
and
\Fi{u=vSUP}
u_\mu(0,x_\mu)=v_\mu(0,x_\mu)\in[\delta',1-\delta''].
\Ff
One further has that
$$v_\mu(t,x)=U\big(x, \rho_{\e^\mu}(t+t_\mu)-|x+\zeta_\mu-(t+t_\mu)y| \big)+\mu O_\mu(t,x)$$
for all $(t,x)\in[-t_\mu,+\infty)\times\R^N$, with $1\leq O_\mu(t,x)\leq 2$.

We have shown in the proof of Lemma~\ref{lem:interior} that the level sets of the functions $(t,x)\mapsto|x+\zeta_\mu-(t+t_\mu)y| - \rho_{\e^\mu}(t+t_\mu)$ converge locally uniformly to some shifting hyperplanes as $\mu\to0^+$. We point out that this only relied on~\eqref{contactpoint0*}, which is exactly~\eqref{contactpoint0*SUP} above, as well as on the boundedness of $|\zeta_\mu-t_\mu y| - \rho_{\e^\mu}(t_\mu)$ following from \eqref{u=v}, which is exactly~\eqref{u=vSUP} above, and the fact that $U(\.,z)\to0,1$ as $z\to\mp\infty$, both properties being valid in the present case. This means that the limit~\eqref{moving-plane} holds true, locally uniformly with respect to $(t,x)\in\R\times\R^N$, with $\nu$ replaced by $-\nu$, because now $y=\bar x+r\nu$ whereas in Lemma~\ref{lem:interior} we had $y=\bar x-r\nu$. Namely, there holds
$$\rho_{\e^\mu}(t+t_\mu) - |x+\zeta_\mu-(t+t_\mu)y|=x\.\nu-\bar x\.\nu t-|\zeta_\mu-t_\mu y|+\rho_{\e^\mu}(t_\mu)+o(1)\ \hbox{ as }\mu\to0^+,$$
locally uniformly with respect to $(t,x)\in\R\times\R^N$. Therefore, considering a sequence $\seq{\mu}$ converging to $0^+$ along which $|\zeta_{\mu_n}-t_{\mu_n} y|-\rho_{\e^{\mu_n}}(t_{\mu_n})$ (which is bounded) converges towards some limit $Z\in\R$, and recalling that $\bar x\.\nu=c$, we eventually get
$$v_{\mu_n}(t,x)\to U(x, x\.\nu-ct-Z)\as n\to+\infty,$$
locally uniformly with respect to $(t,x)\in\R\times\R^N$. Finally, by parabolic estimates, the functions $u_{\mu_n}$ converge as $n\to+\infty$ (up to a subsequence), locally uniformly, to an entire solution $\t u$ of~\eqref{general}, and by \eqref{u<v}-\eqref{u=vSUP} it holds that
$$\min_{t\in[-1,0],\ x\in\R^N}\big(U(x, x\.\nu-ct-Z)-\t u(t,x)\big)=U(\t x, \t x\.\nu-Z)-\t u(0,\t x)=0.$$
where $\t x$ is the limit (of a convergent subsequence of) $(x_{\mu_n})_{n\in\N}$ in $[0,1]^N$. Since by Proposition~\ref{pro:U}, $U(x, x\.\nu-ct-Z)$ is a supersolution to~\eqref{general}, because $c\geq c^*(\nu)$ and $\partial_zU<0$ (and it is a strict supersolution if $c>c^*(\nu)$), the strong maximum principle yields
$$\t u(t,x)\equiv U(x, x\.\nu-ct-Z)\ \hbox{ for all $(t,x)\in[-1,+\infty)\times\R^N$},$$
and moreover $c=c^*(\nu)$. This means that $\phi_\nu(Z/c^*(\nu)+\theta,\.)\in\O_{\hat{\bar x}}(u)$ for all $\theta\ge-1$. We then conclude as in the proof of Lemma~\ref{lem:interior} that $\phi_\nu(s,\cdot+\xi)\in\O_{\hat{\bar x}}(u)$ for all $(s,\xi)\in\R\times\R^N$. The proof of Lemma~\ref{lem:exterior} is thereby complete.\hfill$\Box$


\subsection{Proof of Theorem~\ref{th1}, and Corollaries~\ref{cor1} and~\ref{cor2}}\label{sec34}

We first prove Corollary~\ref{cor1}, and next Corollary~\ref{cor2}, which contains Theorem \ref{th1}.

\begin{proof}[Proof of Corollary~$\ref{cor1}$]
Consider a point ${z}\in\partial\W$ such that ${\mc{V}({z})\neq\emptyset}$. Take $\nu\in \mc{V}({z})$. By definition, there exists a sequence of points $\seq{z}$ in $\partial\W$ at which $\W$ fulfills the interior and exterior ball conditions, such that $\seq{z}$ converges to $z$ and the exterior unit normals $\nu_n$ at $z_n$ converge to $\nu$. By Theorem~\ref{th2}, the following properties hold for each $n\in\N$:
$$z_n\.\nu_n=c^*(\nu_n)\quad\text{and}\quad\O_{\hat{z_n}}(u)\supset\big\{\phi_{\nu_n}(t,\cdot+y):(t,y)\in\R\times\R^N\big\}\cup\{0,1\}.$$
Recall that Propositions~\ref{pro:c>0invasion} and \ref{pro:Wchar} yield $0\in\W$ (which is open) hence the sequence $(\hat{z_n})_{n\in\N}$ is well defined and converges to $\hat z$. Thanks to the continuity of the map $e\mapsto c^*(e)$, granted by Proposition~\ref{pro:U}, the first property above yields $z\.\nu=c^*(\nu)$. By the second one, for any given $(t,y)\in\R\times\R^N$, we can pick two sequences $(t_n)_{n\in\N}$ in $\R^+$ and $(x_n)_{n\in\N}$ in $\R^N\setminus\{0\}$, such that $\seq{t}$ diverges to $+\infty$ and there holds
$$\forall\,n\in\N,\qquad|\hat{x_n}-\hat{z_n}|<\frac1n\quad\text{and}\quad\sup_{x\in B_n}|u(t_n,x_n+x)-\phi_{\nu_n}(t,x+y)|<\frac1n.$$
Here the profiles $\phi_{\nu_n}$ are normalized by, say, $\phi_{\nu_n}(0,0)=1/2$. It then follows from the continuity of the fronts $\phi_e$ with respect to $e$, provided again by Proposition~\ref{pro:U}, that $u(t_n,x_n+\.)\to\phi_{\nu}(t,\.+y)$ in $L^\infty_{loc}(\R^N)$ as $n\to+\infty$. This means that $\phi_\nu(t,\cdot+y)\in \O_{\hat{z}}(u)$. Finally, the inclusion $\{0,1\}\subset \O_{\hat{z}}(u)$ directly follows from Lemma~\ref{lem:01}.
\end{proof}

\begin{proof}[Proof of Corollary~$\ref{cor2}$]
We start with the inclusion~\eqref{minimizer} for the set $\O_e(u)$, for given $e\in\Sph$. Let $\xi\in\Sph$ 
be a minimizer in the expression of $w(e)$ in \eqref{formuleFG}, i.e., $\xi\.e>0$~and
$$w(e)=\frac{c^*(\xi)}{\xi\cdot e}.$$
Calling $\bar x:=w(e)e\in\partial\W_0$, this equality rewrites as $\bar x\.\xi=c^*(\xi)$. We also have by \eqref{formuleFG} 
$$\W_0\subset\{x\in\R^N:x\.\xi<c^*(\xi)\}.$$
It follows that, for given $r>0$, $B_r(\bar x+r\xi)\cap\W_0=\emptyset$, that is, $\W_0$ satisfies the exterior ball condition at ${\bar x}$. We then apply Lemma~\ref{lem:exterior} and derive
$$\displaystyle\O_{e}(u)=\displaystyle\O_{\hat{\bar x}}(u)\supset\big\{\phi_\xi(t,\cdot+y):(t,y)\in\R\times\R^N\big\}.$$	
Moreover the inclusion $\{0,1\}\subset \O_{\hat{\bar x}}(u)$ follows from Lemma~\ref{lem:01}. 

We now turn to the inclusion~\eqref{differentiable} for the set $\O(u)$ in the case
where $\partial\W_0$ is differentiable. We already know that $\{0,1\}\subset \O(u)$. Take an arbitrary $e\in\Sph$. 
We know by \eqref{formuleFG} that $\W_0$ is convex, bounded and non-empty, hence $\ol{\W_0}$ is convex, compact and non-empty. There exists then $L\in\R$ such~that 
$$\W_0\subset\{x\in\R^N:x\.e<L\}\quad\text{and}\quad\ol{\W_0}\cap\{x\in\R^N:x\.e=L\}\neq\emptyset.$$
Since $\partial\W_0$ is differentiable, then necessarily $\{x\in\R^N\ :\ x\.e=L\}$ is the tangent plane to $\W_0$ at any point~$z\in\ol{\W_0}\cap\{x\in\R^N\ :\ x\.e=L\}$, with outward unit normal $e$ (notice that such a $z$ necessarily belongs to $\partial\W_0$). As observed in Footnote~\ref{f2} following the Freidlin-G\"artner formula \eqref{formuleFG}, one then has that the expression for $w(\hat{z})$ is uniquely minimized by $\xi=e$. One eventually concludes from the first statement of the corollary,~that
$$\big\{\phi_e(t,\cdot+y):(t,y)\in\R\times\R^N\big\}\subset \displaystyle\O_{\hat{z}}(u)\subset\O(u).$$
The proof of Corollary~\ref{cor2} is thereby complete.
\end{proof}


\section{Proof of the other results}\label{sec4}

In Section~\ref{sec41}, we state and prove some auxiliary lemmata which will be used in the subsequent sections. Sections~\ref{sec42}-\ref{sec46} are devoted to the proofs of Propositions~\ref{pro:U},~\ref{pro:c>0invasion},~\ref{proinvasion},~\ref{pro:Wchar} and~\ref{pro:level} on the properties of pulsating traveling fronts, the invasion property and the asymptotic invasion shapes.


\subsection{Some auxiliary lemmata}\label{sec41}

Propositions~\ref{pro:U} and~\ref{pro:c>0invasion}, that are proved in Sections~\ref{sec42} and~\ref{sec43}, are based on the following auxiliary lemmata.

\begin{lemma}\label{lem>0}
Assume that Hypotheses~$\ref{hyp:c*>0}$ and~$\ref{hypbis}$ hold. Then $\max_{x\in\R^N}\!f(x,s)>0$ for all $s\in[1-\delta,1)$, where $\delta>0$ is given in Hypothesis~$\ref{hypbis}$.
\end{lemma}

\begin{proof}
Assume by way of contradiction that the conclusion does not hold. By Hypo\-thesis~\ref{hypbis} and $f(\cdot,1)=0$ in $\R^N$, one then infers the existence of $\sigma\in(0,\delta]$ such that
\be\label{defsigma}
f=0\hbox{ in $[1-\sigma,1]\times\R^N$}.
\ee
Consider then any direction $e\in\Sph$, and a front $\phi_e(t,x)=U_e(x,x\cdot e-c^*(e)t)$ given by~\eqref{ptf1}-\eqref{ptf2}, from Hypothesis~\ref{hyp:c*>0}. Since $U_e(\cdot,z)\to 1$ uniformly in $\R^N$ as $z\to-\infty$, there exists $M<0$ such that
\be\label{defM}
\phi_e(t,x)\ge1-\sigma\hbox{ for all $(t,x)\in\R\times\R^N$ such that $x\cdot e-c^*(e)t\le M$}.
\ee
On the other hand, by~\eqref{Ue}, there is $\varsigma\in(0,\sigma]$ such that $U_e(\cdot,M)\le1-\varsigma$ in $\R^N$, that~is,
\be\label{phiesigma}
\phi_e(t,x)\le1-\varsigma\hbox{ for all $(t,x)\in\R\times\R^N$ such that $x\cdot e-c^*(e)t=M$}.
\ee

Under the notations~\eqref{evp}, it then follows from Lemma~\ref{lemkelambda}, together with the continuity of $\lambda\mapsto k_e(\lambda)$ and the positivity of $c^*(e)$ by Hypothesis~\ref{hyp:c*>0}, that there exists $\lambda>0$ such that
\be\label{choicelambda}
k_e(\lambda)\le\frac{c^*(e)\,\lambda}{3}.
\ee
Let $\varphi$ be the positive periodic $C^{2,\alpha}(\R^N)$ eigenfunction solving~\eqref{evp}, normalized such that $\max_{\R^N}\varphi=1$. Consider now the function $\overline{u}$ defined in $\R\times\R^N$ by
$$\overline{u}(t,x):=1-\varsigma+\varsigma\,\frac{\varphi(x)}{\min_{\R^N}\varphi}\,e^{-\lambda(x\cdot e-c^*(e)t/3)},$$
and the set
$$E:=\Big\{(t,x)\in\R\times\R^N:\frac{c^*(e)t}{3}\le x\cdot e\le c^*(e)t+M\Big\}.$$
Notice that $t\ge-3M/(2c^*(e))$ for every $(t,x)\in E$. For every $(t,x)\in\partial E$, either $x\cdot e=c^*(e)t+M$ and then $\overline{u}(t,x)\ge1-\varsigma\ge\phi_e(t,x)$ by~\eqref{phiesigma} and the definition of~$\overline{u}$, or $x\cdot e=c^*(e)t/3$ and then $\overline{u}(t,x)\ge1>\phi_e(t,x)$. Furthermore, $\phi_e(t,x)\ge1-\sigma$ for all $(t,x)\in E$ by~\eqref{defM}, whence
$$\partial_t\phi_e=\dv(A(x)\nabla\phi_e)+q(x)\.\nabla\phi_e\ \hbox{ in $E$}$$
by~\eqref{general} and~\eqref{defsigma}. On the other hand, 
$$\partial_t\overline{u}\ge\dv(A(x)\nabla\overline{u})+q(x)\.\nabla\overline{u}\ \hbox{ in $E$}$$
by~\eqref{evp} and~\eqref{choicelambda}. The parabolic maximum principle then yields $\overline{u}\ge\phi_e$ in~$E$. In~particular, for $t\ge-3M/c^*(e)$ and $x_t:=(2c^*(e)t/3)e$, there holds $(t,x_t)\in E$ and
$$U_e\Big(x_t,-\frac{c^*(e)t}{3}\Big)=\phi_e(t,x_t)\le\overline{u}(t,x_t)=1-\varsigma+\varsigma\,\frac{\varphi(x_t)}{\min_{\R^N}\varphi}\,e^{-\lambda c^*(e)t/3}.$$
By~\eqref{ptf2}, the passage to the limit as $t\to+\infty$ leads to $1\le1-\varsigma$, a contradiction. The proof of Lemma~\ref{lem>0} is thereby complete.
\end{proof}

The following Liouville-type result is a consequence of Lemma~\ref{lem>0}.

\begin{lemma}\label{lem>0bis}
Assume that Hypotheses~$\ref{hyp:c*>0}$ and~$\ref{hypbis}$ hold. Then, for any time-global solution $u:\R\times\R^N\to[0,1]$ of~\eqref{general} with $1-\delta\le u\le1$ in $\R\times\R^N$, one has $u\equiv 1$ in $\R\times\R^N$.
\end{lemma}

\begin{proof}
From the maximum principle, it suffices that the solution $u$ of~\eqref{general} with initial condition $u_0:=1-\delta$ satisfies $u(t,x)\to1$ as $t\to+\infty$ for all $x\in\R^N$ (indeed, for a solution $v$ satisfying $1-\delta\le v\le 1$ in $\R\times\R^N$, since $1\ge v(-n,\cdot)\ge1-\delta=u_0$ in $\R^N$ for all $n\in\N$, we would then have that $1\ge v(t,x)\ge u(t+n,x)$ for all $(t,x)\in\R^N$ and $n\ge-t$, whence $v(t,x)=1$ for all $t\in\R$ for all $x\in\R^N$ by passing to the limit $n\to+\infty$). So, consider the solution $u$ of~\eqref{general} with
$$u_0:=1-\delta$$
in $\R^N$. Since $u(t,x)\le1$ for all $(t,x)\in\R_+\times\R^N$ and $f(x,s)\geq0$ for all $(x,s)\in\R^N\times[1-\delta,1]$ by Hypothesis~\ref{hypbis} and $f(\cdot,1)\equiv0$ in $\R^N$, using the comparison principle one sees that $u(t,x)$ is non-decreasing in $t$, hence by parabolic estimates $u(t,\cdot)$ converges in $C^2_{loc}(\R^N)$ to a stationary solution $\t u$ of \eqref{general} satisfying $1-\delta\leq \t u\leq 1$ in $\R^N$. Moreover, $u$ is periodic in $x$, hence $\t u$ is periodic too, while $f(x,\t u(x))\ge0$ for all $x\in\R^N$ since $1-\delta\le\t u\le1$ in $\R^N$. Considering a minimum point of $\t u$ and applying this time the elliptic strong maximum principle, one infers that $\t u$ is constant, whence $f(x,\t u)=0$ for all $x\in\R^N$. Lemma~\ref{lem>0} then yields $\t u=1$, that is, $u(t,x)\to1$ as $t\to+\infty$ for all $x\in\R^N$, and the proof of Lemma~\ref{lem>0bis} is complete.
\end{proof}

The next two lemmata, which are particular cases of comparison principles established in~\cite{BH1} in more general geometries, provide some comparisons between sub- and super-solutions defined in half-spaces in the $(t,x)$ variables.

\begin{lemma}\label{lem1}{\rm{\cite[Lemma~3.2]{BH1}}}
Assume that Hypothesis~$\ref{hypbis}$ holds. Let $e\in\Sph$, $c\in\R\setminus\{0\}$, $h\in\R$, and denote $E:=\{(t,x)\in\R\times\R^N:x\cdot e-ct\ge h\}$. Let $\underline{u}:E\to[0,1]$ and $\overline{u}:E\to[0,1]$ be two $C^1$ functions which are $C^2$ in $x$ and are such that $\partial_t\underline{u}$, $\partial_t\overline{u}$, $\partial_{x_i}\underline{u}$, $\partial_{x_i}\overline{u}$, $\partial_{x_ix_j}\underline{u}$ and $\partial_{x_ix_j}\overline{u}$ are all bounded and H\"older-continuous in $E$, for every $1\le i,j\le N$. Assume that
$$\left\{\baa{l}
\partial_t\underline{u}\le\dv(A(x)\nabla\underline{u})+q(x)\.\nabla\underline{u} + f(x,\underline{u})\ \hbox{ in }E,\vspace{3pt}\\
\partial_t\overline{u}\ge\dv(A(x)\nabla\overline{u})+q(x)\.\nabla\overline{u} + f(x,\overline{u})\ \hbox{ in }E,\vspace{3pt}\\
\underline{u}\le\delta\ \hbox{ in }E,\vspace{3pt}\\
\displaystyle\mathop{\lim}_{H\to+\infty}\Big(\sup_{(t,x)\in E,\,x\cdot e-ct\ge H}\big(\underline{u}(t,x)-\overline{u}(t,x)\big)\Big)\le 0, \vspace{3pt}\\
\underline{u}\le\overline{u}\ \hbox{ on }\partial E.\eaa\right.$$
Then $\underline{u}\le\overline{u}$ in $E$.
\end{lemma}

\begin{lemma}\label{lem2}{\rm{\cite[Lemma~3.4]{BH1}}}
Assume that Hypothesis~$\ref{hypbis}$ holds. Let $e\in\Sph$, $c\in\R\setminus\{0\}$, $h\in\R$, and denote $F:=\{(t,x)\in\R\times\R^N:x\cdot e-ct\le h\}$. Let $\underline{u}:F\to[0,1]$ and $\overline{u}:F\to[0,1]$ be two $C^1$ functions which are $C^2$ in $x$ and are such that $\partial_t\underline{u}$, $\partial_t\overline{u}$, $\partial_{x_i}\underline{u}$, $\partial_{x_i}\overline{u}$, $\partial_{x_ix_j}\underline{u}$ and $\partial_{x_ix_j}\overline{u}$ are all bounded and H\"older-continuous in $F$, for every $1\le i,j\le N$. Assume that
$$\left\{\baa{l}
\partial_t\underline{u}\le\dv(A(x)\nabla\underline{u})+q(x)\.\nabla\underline{u} + f(x,\underline{u})\ \hbox{ in }F,\vspace{3pt}\\
\partial_t\overline{u}\ge\dv(A(x)\nabla\overline{u})+q(x)\.\nabla\overline{u} + f(x,\overline{u})\ \hbox{ in }F,\vspace{3pt}\\
\overline{u}\ge1-\delta\ \hbox{ in }F,\vspace{3pt}\\
\displaystyle\mathop{\lim}_{H\to-\infty}\Big(\sup_{(t,x)\in F,\,x\cdot e-ct\le H}\big(\underline{u}(t,x)-\overline{u}(t,x)\big)\Big)\le 0, \vspace{3pt}\\
\underline{u}\le\overline{u}\ \hbox{ on }\partial F.\eaa\right.$$
Then $\underline{u}\le\overline{u}$ in $F$.
\end{lemma}

From the previous two lemmata, another comparison principle follows, between sub- and super-solutions of~\eqref{general} defined this time in $\R\times\R^N$.

\begin{lemma}\label{lem3}
Assume that Hypothesis~$\ref{hypbis}$ holds. Let $e\in\Sph$, $c>0$, and $\underline{u}:\R\times\R^N\to(0,1)$ and $\overline{u}:\R\times\R^N\to(0,1)$ be two $C^1$ functions which are $C^2$ in $x$ and are such that $\partial_t\underline{u}$, $\partial_t\overline{u}$, $\partial_{x_i}\underline{u}$, $\partial_{x_i}\overline{u}$, $\partial_{x_ix_j}\underline{u}$ and $\partial_{x_ix_j}\overline{u}$ are all bounded and H\"older-continuous in $\R\times\R^N$, for every $1\le i,j\le N$. Assume that
$$\left\{\baa{l}
\partial_t\underline{u}\le\dv(A(x)\nabla\underline{u})+q(x)\.\nabla\underline{u} + f(x,\underline{u})\ \hbox{ in }\R\times\R^N,\vspace{3pt}\\
\partial_t\overline{u}\ge\dv(A(x)\nabla\overline{u})+q(x)\.\nabla\overline{u} + f(x,\overline{u})\ \hbox{ in }\R\times\R^N,\vspace{3pt}\\
\underline{u}(t,x)\to1\hbox{ and }\overline{u}(t,x)\to1\ \hbox{ as }x\cdot e-ct\to-\infty,\vspace{3pt}\\
\underline{u}(t,x)\to0\ \hbox{ as }x\cdot e-ct\to+\infty.\eaa\right.$$
Then there is the smallest $T\in\R$ such that
$$\underline{u}(t,x)\le\overline{u}(t+T,x)\ \hbox{ for all $(t,x)\in\R\times\R^N$},$$
and there is a sequence $(t_k,x_k)_{k\in\N}$ in $\R\times\R^N$ such that $\sup_{k\in\N}|x_k\cdot e-ct_k|<+\infty$ and $\underline{u}(t_k,x_k)-\overline{u}(t_k+T,x_k)\to0$ as $k\to+\infty$.
\end{lemma}

\begin{proof}
Let $B>0$ such that
$$\underline{u}\le\delta\ \hbox{ in $E:=\{(t,x)\in\R\times\R^N:x\cdot e-ct\ge B\}$}$$
and
$$\min(\underline{u},\overline{u})\ge1-\delta/2\ \hbox{ in $F:=\{(t,x)\in\R\times\R^N:x\cdot e-ct\le-B\}$}.$$
For any $\tau\ge2B/c$, calling $\overline{u}^\tau:=\overline{u}(\cdot+\tau,\cdot)$, one has $\underline{u}\le\delta\le1-\delta\le1-\delta/2\le\overline{u}^\tau$ on $\partial E$, and then $\underline{u}\le\overline{u}^\tau$ in $E$ by Lemma~\ref{lem1} applied to $\underline{u}$ and $\overline{u}^\tau$ in $E$ (notice that the partial diffierential inequation satisfied by $\overline{u}$ in $\R\times\R^N$ holds as well for $\overline{u}^s$ for each $s\in\R$, since the coefficients $A$, $q$ and $f$ are independent of $t$). On the other hand, still for any $\tau\ge 2B/c$, one has $\overline{u}^\tau\ge1-\delta/2\ge1-\delta$ in $(\R\times\R^N)\setminus E$, and then $\underline{u}\le\overline{u}^\tau$ in $(\R\times\R^N)\setminus E$ by applying this time Lemma~\ref{lem2} in this set. Therefore, $\underline{u}\le\overline{u}^\tau$ in $\R\times\R^N$ for any $\tau\ge2B/c$. 

Call now
$$T:=\inf\big\{\tau\in\R:\underline{u}\le\overline{u}^\tau\hbox{ in $\R\times\R^N$}\big\}.$$
From the previous paragraph and the fact that $\overline{u}<1$ in $\R\times\R^N$ and $\underline{u}(t,x)\to1$ as $x\cdot e-ct\to-\infty$, one gets that $T$ is a real number, and that $\underline{u}\le\overline{u}^T$ in $\R\times\R^N$ (that is, the first part of the desired conclusion). In particular, $\overline{u}^T\ge\underline{u}\ge1-\delta/2$ in $F$. Define
\be\label{defeta}
\eta:=\inf_{(t,x)\in\R\times\R^N,\,|x\cdot e-ct|\le B}\big(\overline{u}^T(t,x)-\underline{u}(t,x)\big)\ge0,
\ee
and assume by way of contradiction that $\eta>0$. Since $\partial_t\overline{u}$ is bounded in $\R\times\R^N$, there is $\sigma<T$ such that $\overline{u}^\sigma\ge1-\delta$ in $F$ and $\overline{u}^\sigma(t,x)-\underline{u}(t,x)\ge0$ for all $(t,x)\in\R\times\R^N$ such that $|x\cdot e-ct|\le B$. By Lemma~\ref{lem1} applied to $\underline{u}$ and $\overline{u}^\sigma$ in~$E$, one gets that $\underline{u}\le\overline{u}^\sigma$ in~$E$. Furthermore, by Lemma~\ref{lem2} applied to $\underline{u}$ and $\overline{u}^\sigma$ in~$F$, one also gets that $\underline{u}\le\overline{u}^\sigma$ in~$F$. Finally, $\underline{u}\le\overline{u}^\sigma$ in $\R\times\R^N$, contradicting the definition of $T$.

As a conclusion, $\eta$ in~\eqref{defeta} is equal to $0$, and the conclusion of Lemma~\ref{lem3} follows.
\end{proof}


\subsection{Proof of Proposition~\ref{pro:U}}\label{sec42}

The proof of Proposition~\ref{pro:U} makes use of some results of~\cite{BH3}, and of Lemmata~\ref{lem1} and~\ref{lem3}.

\begin{proof}[Proof of Proposition~$\ref{pro:U}$]
We assume Hypotheses~\ref{hyp:c*>0} and~\ref{hypbis} throughout the proof. We recall that Hypothesis~\ref{hypbis} is weaker than Hypothesis~\ref{hyp:comb-bi}.
\vskip 0.2cm
\noindent{\it Step 1: monotonicity and uniqueness}. First of all, for any direction $e\in\Sph$, the pulsating traveling front $\phi_e(t,x)=U_e(x,x\cdot e-c^*(e)t)$ given in~\eqref{ptf1}-\eqref{ptf2} with $c^*(e)>0$ is an invasion of $0$ by $1$ in the sense of~\cite[Definition~1.4]{BH3}, since $c^*(e)>0$ by Hypothesis~\ref{hyp:c*>0}. It then follows from~\cite[Theorem~1.11]{BH3} that $\phi_e$ is increasing in time~$t$. Furthermore, since the coefficients of~\eqref{general} are independent of~$t$, the nonnegative function $\partial_t\phi_e$ is, from parabolic regularity theory, a classical solution of
\be\label{partialtphie}
\partial_t(\partial_t\phi_e)=\dv(A(x)\nabla\partial_t\phi_e)+q(x)\.\nabla\partial_t\phi_e +\partial_sf(x,\phi_e)\,\partial_t\phi_e
\ee
in $\R\times\R^N$, whence $\partial_t\phi_e>0$ in $\R\times\R^N$ from the strong parabolic maximum principle. Furthermore, the speed $c^*(e)$ is necessarily unique by~\cite[Theorem~1.12]{BH3} and $\phi_e$ is unique up to shift in~$t$ by~\cite[Theorems~1.12 and~1.14]{BH3}. In other words, $\partial_zU_e(x,z)<0$ in $\R^N\times\R$, and $U_e$ is unique up to shift in $z$.

\vskip 0.2cm
\noindent{\it Step 2: exponential decay estimates~\eqref{Ueexp} for $U_e$}. Given a direction $e\in\Sph$, it follows from the notations~\eqref{evp} and Lemma~\ref{lemkelambda}, together with the continuity of ${\lambda\mapsto k_e(\lambda)}$ and the positivity of $c^*(e)$ by Hypothesis~\ref{hyp:c*>0}, that there exists $\lambda_0>0$ such that
$$k_e(\lambda_0)\le c^*(e)\,\lambda_0.$$
Let $\varphi_0$ be the positive periodic $C^{2,\alpha}(\R^N)$ eigenfunction associated to the eigenvalue problem~\eqref{evp} with parameter $\lambda:=\lambda_0$, normalized with $\max_{\R^N}\!\varphi_0=1$. Let $\delta>0$ be given by Hypothesis~\ref{hypbis} and let $\delta_0:=\delta\times\min_{\R^N}\!\varphi_0\in(0,\delta]$. By~\eqref{ptf2}, there is $h_0\in\R$ such that $0<U_e(x,z)\le\delta_0$ for all $(x,z)\in\R^N\times[h_0,+\infty)$, that is, $0<\phi_e\le\delta_0$ in $E:=\{(t,x)\!\in\!\R\times\R^N:x\cdot e-c^*(e)t\ge h_0\}$. Consider now the function $\overline{u}$ defined in~$E$~by
$$\overline{u}(t,x):=\frac{\delta_0\varphi_0(x)}{\min_{\R^N}\!\varphi_0}\,e^{-\lambda_0(x\cdot e-c^*(e)t-h_0)}.$$
Since $k_e(\lambda_0)\le c^*(e)\,\lambda_0$, the function $\overline{u}$ satisfies $\partial_t\overline{u}\ge\dv(A(x)\nabla\overline{u})+q(x)\.\nabla\overline{u}$ in~$E$. Furthermore, $0<\overline{u}\le\delta_0/(\min_{\R^N}\!\varphi_0)=\delta$ in $E$, whence $f(x,\overline{u})\le0$ in $E$ by Hypothesis~\ref{hypbis} and $f(\cdot,0)\equiv0$ in $\R^N$, and finally
$$\partial_t\overline{u}\ge\dv(A(x)\nabla\overline{u})+q(x)\.\nabla\overline{u}+f(x,\overline{u})\ \hbox{ in }E.$$
On the other hand, $0<\phi_e\le\delta_0\le\delta$ in $E$, while $\phi_e(t,x)\to0$ and $\overline{u}(t,x)\to0$ as $x\cdot e-c^*(e)t\to+\infty$, and $\phi_e(t,x)\le\delta_0\le\overline{u}(t,x)$ for all $(t,x)\in\R\times\R^N$ such that $x\cdot e-c^*(e)t=h_0$, that is, $(t,x)\in\partial E$. It follows from Lemma~\ref{lem1} that $\phi_e\le\overline{u}$ in $E$, that is,
$$U_e(x,z)\le\frac{\delta_0\varphi_0(x)}{\min_{\R^N}\!\varphi_0}\,e^{-\lambda_0(z-h_0)}\le\frac{\delta_0}{\min_{\R^N}\!\varphi_0}\,e^{-\lambda_0(z-h_0)}=\delta\,e^{-\lambda_0(z-h_0)}\le e^{-\lambda_0(z-h_0)}.$$
Together with the inequality $U_e<1$ in $\R\times\R^N$,~\eqref{Ueexp} follows with $C:=e^{\lambda_0h_0}$.

\vskip 0.2cm
\noindent{\it Step 3: continuity of the map $e\mapsto c^*(e)$ from $\Sph$ to $\R_+=(0,+\infty)$}. Remember first that, for each $e\in\Sph$, the pulsating front $\phi_e:\R\times\R^N\to(0,1)$ given in~\eqref{ptf1} is of class~$C^1$ in~$(t,x)$,~$C^2$ in~$x$, and $\partial_t\phi_e$, $\partial_{x_i}\phi_e$ and $\partial_{x_ix_j}\phi_e$ are all bounded and H\"older-continuous in $\R\times\R^N$, for every $1\le i,j\le N$. Furthermore, from parabolic regularity theory, the function $\partial_t\phi_e$ is a classical solution of~\eqref{partialtphie} in $\R\times\R^N$, and $\partial_{tt}\phi_e$ and $\partial_{tx_i}\phi_e$ are also bounded and H\"older-continuous in $\R\times\R^N$, for every $1\le i\le N$ (as are also the second-order partial derivatives of $\partial_t\phi_e$ with respect to $x$).

Consider a sequence $(e_n)_{n\in\N}$ in $\Sph$ converging to a direction~$e$ in~$\Sph$. For each $n\in\N$, call
$$u_n:\R\times\R^N\to(0,1),\quad (t,x)\mapsto u_n(t,x):=\phi_{e_n}(t,x),$$
which solves~\eqref{general} in $\R\times\R^N$, and define a function $v_n$ in $\R\times\R^N$ by
$$v_n(t,x):=U_e(x,x\cdot e_n-c^*(e_n)t)=\phi_e\Big(\frac{c^*(e_n)t+x\cdot(e-e_n)}{c^*(e)},x\Big).$$
Notice that both functions $u_n$ and $v_n$ are of class $C^2$ in $\R\times\R^N$, with bounded and H\"older-continuous first- and second-order partial derivatives in the $(t,x)$ variables. From~\eqref{ptf2}, there holds $u_n(t,x)\to1$ and $v_n(t,x)\to1$, resp. $0$, as $x\cdot e_n-c^*(e_n)t\to-\infty$, resp. $+\infty$. By definition, the function $u_n$ solves~\eqref{general} in $\R\times\R^N$, while the function~$v_n$ satisfies
$$\baa{l}
\partial_tv_n(t,x)-\dv(A(x)\nabla v_n)(t,x)-q(x)\cdot\nabla v_n(t,x)-f(x,v_n(t,x))\vspace{3pt}\\
\displaystyle\qquad\qquad\qquad\qquad\qquad=\frac{c^*(e_n)-c^*(e)}{c^*(e)}\,\partial_t\phi_e\Big(\frac{c^*(e_n)t+x\cdot(e-e_n)}{c^*(e)},x\Big)+R_n(t,x)\eaa$$
for all $(t,x)\in\R\times\R^N$, where
$$\baa{rcl}
R_n(t,x) & := & -\displaystyle\frac{\dv(A(x)(e-e_n))+q(x)\cdot(e-e_n)}{c^*(e)}\,\partial_t\phi_e\Big(\frac{c^*(e_n)t+x\cdot(e-e_n)}{c^*(e)},x\Big)\vspace{3pt}\\
& & \displaystyle-\frac{2}{c^*(e)}\,(e-e_n)A(x)\nabla\partial_t\phi_e\Big(\frac{c^*(e_n)t+x\cdot(e-e_n)}{c^*(e)},x\Big)\vspace{3pt}\\
& & \displaystyle-\frac{(e-e_n)A(x)(e-e_n)}{(c^*(e))^2}\partial_{tt}\phi_e\Big(\frac{c^*(e_n)t+x\cdot(e-e_n)}{c^*(e)},x\Big).\eaa$$
From parabolic estimates applied to the positive solution $\partial_t\phi_e$ of~\eqref{partialtphie}, there is a constant $C_1>0$ such that
$$\forall\,(t,x)\in\R\times\R^N,\quad|\nabla\partial_t\phi_e(t,x)|+|\partial_{tt}\phi_e(t,x)|\le C_1\times\max_{[t-1,t]\times\overline{B_1(x)}}\partial_t\phi_e.$$
Pick any vector $\xi\in\Z^N$ such that $\xi\cdot e>0$. From the parabolic Harnack inequality applied to $\partial_t\phi_e$, there is a constant $C_2>0$ such that
$$\forall\,(t,x)\in\R\times\R^N,\quad\max_{[t-1,t]\times\overline{B_1(x)}}\partial_t\phi_e\le C_2\times\partial_t\phi_e\Big(t+\frac{\xi\cdot e}{c^*(e)},x+\xi\Big).$$
But, for all $(t,x)\in\R\times\R^N$, $\phi_e(t,x)=\phi_e(t+(\xi\cdot e)/c^*(e),x+\xi)$ by~\eqref{ptf1}-\eqref{ptf2}, whence $\partial_t\phi_e(t,x)=\partial_t\phi_e(t+(\xi\cdot e)/c^*(e),x+\xi)$. Finally, one gets that
$$\forall\,(t,x)\in\R\times\R^N,\quad|\nabla\partial_t\phi_e(t,x)|+|\partial_{tt}\phi_e(t,x)|\le C_1C_2\times\partial_t\phi_e(t,x).$$

Assume now by way of contradiction that $c^*(e_n)\not\to c^*(e)$ as $n\to+\infty$. Up to extraction of a subsequence, two cases may occur: either there is $\epsilon>0$ such that $c^*(e_n)\ge c^*(e)+\epsilon$ for all $n\in\N$, or there is $\epsilon>0$ such that $c^*(e_n)\le c^*(e)-\epsilon$ for all $n\in\N$. Consider first the former case. From the previous paragraph, since $e_n\to e$ as $n\to+\infty$, there is $N_1\in\N$ large enough such that
$$|R_n(t,x)|\le\frac{\epsilon}{2c^*(e)}\,\partial_t\phi_e\Big(\frac{c^*(e_n)t+x\cdot(e-e_n)}{c^*(e)},x\Big)$$
for all $n\ge N_1$ and $(t,x)\in\R\times\R^N$, whence
\be\label{strictsuper}\baa{r}
\partial_tv_n(t,x)-\dv(A(x)\nabla v_n)(t,x)-q(x)\cdot\nabla v_n(t,x)-f(x,v_n(t,x))\vspace{3pt}\\
\displaystyle\ge\frac{\epsilon}{2c^*(e)}\,\partial_t\phi_e\Big(\frac{c^*(e_n)t+x\cdot(e-e_n)}{c^*(e)},x\Big)>0.\eaa
\ee
Take any $n\ge N_1$. By Lemma~\ref{lem3} applied with direction $e_n\in\Sph$, $c:=c^*(e_n)$, $\underline{u}:=u_n=\phi_{e_n}$ and $\overline{u}:=v_n$, there are $T\in\R$ and a sequence $(t_k,x_k)_{k\in\N}$ in $\R\times\R^N$ such that
$$U_e(x,x\cdot e_n-c^*(e_n)(t+T))=v_n(t+T,x)\ge\phi_{e_n}(t,x)=U_{e_n}(x,x\cdot e_n-c^*(e_n)t)$$
for all $(t,x)\in\R\times\R^N$ and
\be\label{limitsvn}
\sup_{k\in\N}|x_k\cdot e_n-c^*(e_n)t_k|<+\infty,\ \ \ v_n(t_k+T,x_k)-\phi_{e_n}(t_k,x_k)\to 0\hbox{ as }k\to+\infty.
\ee
For each $k\in\N$, let $y_k\in\Z^N$ and $z_k\in[0,1)^N$ such that $x_k=y_k+z_k$, and call
$$\sigma_k:=x_k\cdot e_n-c^*(e_n)t_k\ \hbox{ and }\ \tau_k:=\frac{z_k\cdot e_n-\sigma_k}{c^*(e_n)}.$$
Since the sequences $(\sigma_k)_{k\in\N}$ and $(z_k)_{k\in\N}$ are bounded, so is the sequence $(\tau_k)_{k\in\N}$ and one can assume, without loss of generality, that $\tau_k\to\tau\in\R$ and $z_k\to z\in\R^N$ as $k\to+\infty$. Now, by~\eqref{ptf1}-\eqref{ptf2} applied with directions $e_n$ and $e$, there holds that
$$\phi_{e_n}(t_k,x_k)=\phi_{e_n}(\tau_k,z_k)\ \hbox{ and }\ v_n(t_k+T,x_k)=v_n(\tau_k+T,z_k)$$
for all $k\in\N$, whence $v_n(\tau+T,z)=\phi_{e_n}(\tau,z)$ by passing to the limit $k\to+\infty$ in the second part of~\eqref{limitsvn}. The strong parabolic maximum principle then entails that $v_n(t+T,x)=\phi_{e_n}(t,x)$ for all $t\le\tau$ and $x\in\R^N$, a contradiction with the strict inequality in~\eqref{strictsuper}. Therefore, the case $c^*(e_n)\ge c^*(e)+\epsilon$ (for all $n\in\N$ or just for a subsequence) is ruled out.

Similarly, if $c^*(e_n)\le c^*(e)-\epsilon$ for all $n\in\N$ (or just for a subsequence), one is also led to a contradiction, by using this time Lemma~\ref{lem3} with direction $e_n\in\Sph$, $c:=c^*(e_n)$, $\underline{u}:=v_n$ and $\overline{u}:=u_n=\phi_{e_n}$, for $n$ large enough.

As a conclusion, the map $e\mapsto c^*(e)$ is continuous from $\Sph$ to $(0,+\infty)$.

\vskip 0.2cm
\noindent{\it Step 4: continuity of the profiles $\phi_e$ with respect to $e$}. Consider any $e\in\Sph$, any sequence $(e_n)_{n\in\N}$ in $\Sph$ converging to $e$, and any $\mu\in(0,1)$. Up to shifts in $t$, normalize the fronts $\phi_{e_n}(t,x)=U_{e_n}(x,x\cdot e_n-c^*(e_n)t)$ and $\phi_e(t,x)=U_e(x,x\cdot e-c^*(e)t)$ so that
$$\phi_{e_n}(0,0)=\phi_e(0,0)=\mu.$$
Notice that this is, uniquely, possible since the maps $t\mapsto\phi_{e_n}(t,0)=U_{e_n}(0,-c^*(e_n)t)$ and $t\mapsto\phi_e(t,0)=U_e(0,-c^*(e)t)$ are continuous increasing and $\phi_{e_n}(-\infty,0)=\phi_e(-\infty,0)=0$, $\phi_{e_n}(+\infty,0)=\phi_e(+\infty,0)=1$, by~\eqref{ptf1}-\eqref{ptf2} and the positivity of $c^*(e_n)$ and $c^*(e)$. For each $n\in\N$, let also $t_n$ be the unique real number such that
\be\label{deftn}
\min_{x\in[0,1]^N}\phi_{e_n}\Big(t_n+\frac{x\cdot e_n}{c^*(e_n)},x\Big)=\min_{x\in[0,1]^N}U_{e_n}(x,-c^*(e_n)t_n)=1-\delta,
\ee
and let $t_e$ be the unique real number such that
\be\label{defte}
\min_{x\in[0,1]^N}\phi_e\Big(t_e+\frac{x\cdot e}{c^*(e)},x\Big)=\min_{x\in[0,1]^N}U_e(x,-c^*(e)t_e)=1-\delta.
\ee

Define, for $n\in\N$ and $(t,x)\in\R\times\R^N$,
$$u_n(t,x):=\phi_{e_n}(t+t_n,x)=U_{e_n}(x,x\cdot e_n-c^*(e_n)(t+t_n)).$$
Each function $u_n:\R\times\R^N\to(0,1)$ is a classical solution of~\eqref{general} and, from standard parabolic estimates, the sequence $(u_n)_{n\in\N}$ converges, up to extraction of a subsequence, locally uniformly in $\R\times\R^N$ to a classical solution $u:\R\times\R^N\to[0,1]$ of~\eqref{general}, with bounded and H\"older-continuous partial derivatives $\partial_tu$, $\partial_{x_i}u$ and $\partial_{x_ix_j}u$ (for every $1\le i,j\le N$). Moreover, using the limit $\lim_{n\to+\infty}c^*(e_n)=c^*(e)$ from Step~3, one gets from~\eqref{deftn} that
\be\label{u1-delta}
\min_{x\in[0,1]^N}u\Big(\frac{x\cdot e}{c^*(e)},x\Big)=1-\delta.
\ee
Therefore, one infers from the maximum principle, together with $f(\cdot,0)=f(\cdot,1)=0$ in~$\R^N$, that $0<u(t,x)<1$ for all $(t,x)\in\R\times\R^N$. On the other hand, the definition of~$u_n$ and~\eqref{ptf2} applied with the direction $e_n$ imply that $u_n(t+(\xi\cdot e_n)/c^*(e_n),x+\xi)=u_n(t,x)$ for all $n\in\N$, $(t,x)\in\R\times\R^N$ and $\xi\in\Z^N$, whence
\be\label{uperiodic}
\forall\,(t,x)\in\R\times\R^N,\ \forall\,\xi\in\Z^N,\ \ u\Big(t+\frac{\xi\cdot e}{c^*(e)},x+\xi\Big)=u(t,x).
\ee

Furthermore, by~\eqref{deftn}, for every $n\in\N$ and $(t,x)\in\R\times\R^N$ such that $x\cdot e_n-c^*(e_n)t\le0$, one has
$$1>u_n(t,x)=U_{e_n}(x,x\cdot e_n-c^*(e_n)(t+t_n))\ge U_{e_n}(x,-c^*(e_n)t_n)\ge1-\delta$$
since $U_{e_n}$ is decreasing in its second variable and periodic in its first variable. Therefore,
\be\label{u1delta}
1\ge u(t,x)\ge1-\delta\hbox{ for all $(t,x)\in\R\times\R^N$ such that $x\cdot e-c^*(e)t<0$}
\ee
(and also when $x\cdot e-c^*(e)t=0$ by continuity of $u$). We now claim that
\be\label{claimu1}
u(t,x)\to1\hbox{ as $x\cdot e-c^*(e)t\to-\infty$}.
\ee
Indeed, otherwise there would exist a sequence $(\tau_k,x_k)_{k\in\N}$ such that $x_k\cdot e-c^*(e)\tau_k\to-\infty$ as $k\to+\infty$ and $\limsup_{k\to+\infty}u(\tau_k,x_k)<1$. By writing $x_k=y_k+z_k$ with $y_k\in\Z^N$ and $z_k\in[0,1)^N$, the functions $\tilde{u}_k(t,x):=u(t+\tau_k,x+y_k)$ then converge locally uniformly in $\R\times\R^N$, up to extraction of a subsequence, to a classical solution $\tilde{u}$ of~\eqref{general} in $\R\times\R^N$ such that $\tilde{u}(0,z)<1$ for some $z\in[0,1]^N$. Furthermore, $1-\delta\le\tilde{u}\le1$ in $\R\times\R^N$ from~\eqref{u1delta} and $\lim_{k\to+\infty}x_k\cdot e-c^*(e)\tau_k=-\infty$. That contradicts Lemma~\ref{lem>0bis}. Therefore, the claim~\eqref{claimu1} is proved.

Apply now Lemma~\ref{lem3} with $c:=c^*(e)$, $\underline{u}:=\phi_e$ and $\overline{u}:=u$. There exist then $T\in\R$ and a sequence $(s_m,\zeta_m)_{m\in\N}$ in $\R\times\R^N$ such that $u(\cdot+T,\cdot)\ge\phi_e$ in $\R\times\R^N$, $\sup_{m\in\N}|\zeta_m\cdot e-c^*(e)s_m|<+\infty$ and
\be\label{smzetam}
u(s_m+T,\zeta_m)-\phi_e(s_m,\zeta_m)\to0\ \hbox{ as $m\to+\infty$}.
\ee
For each $m\in\N$, let $\xi_m\in\Z^N$ and $\omega_m\in[0,1)^N$ such that $\zeta_m=\xi_m+\omega_m$, and call
$$\nu_m:=\zeta_m\cdot e-c^*(e)s_m\ \hbox{ and }\ \theta_m:=\frac{\omega_m\cdot e-\nu_m}{c^*(e)}.$$
Since the sequences $(\nu_m)_{m\in\N}$ and $(\omega_m)_{m\in\N}$ are bounded, so is the sequence $(\theta_m)_{m\in\N}$ and one can assume, without loss of generality, that $\theta_m\to\theta\in\R$ and $\omega_m\to\omega\in\R^N$ as $m\to+\infty$. Now, by~\eqref{ptf1}-\eqref{ptf2} and~\eqref{uperiodic}, there holds that
$$\phi_e(s_m,\zeta_m)=\phi_e(\theta_m,\omega_m)\ \hbox{ and }\ u(s_m+T,\zeta_m)=u(\theta_m+T,\omega_m)$$
for all $m\in\N$, whence $u(\theta+T,\omega)=\phi_e(\theta,\omega)$ by passing to the limit $m\to+\infty$ in~\eqref{smzetam}. The strong parabolic maximum principle then entails that
$$u(t+T,x)=\phi_e(t,x)$$
for all $t\le\theta$ and $x\in\R^N$, and then in $\R\times\R^N$ since both $u$ and $\phi_e$ solve~\eqref{general} in $\R\times\R^N$. Together with~\eqref{defte}-\eqref{u1-delta} and the positivity of $\partial_t\phi_e$, one infers that $T=-t_e$, that~is,
$$u\equiv\phi_e(\cdot+t_e,\cdot)\ \hbox{ in $\R\times\R^N$}.$$
Notice that, by uniqueness of the limit and compactness, the whole sequence $(u_n)_{n\in\N}$ then converges to $\phi_e(\cdot+t_e,\cdot)$ locally uniformly in $\R\times\R^N$.

Finally, remember that $\phi_{e_n}=u_n(\cdot-t_n,\cdot)$ in $\R\times\R^N$ for each $n\in\N$. If, by way of contradiction, $t_n\to-\infty$ up to a subsequence, then for each $\tau\in\R$, there holds $\mu=\phi_{e_n}(0,0)=u_n(-t_n,0)\ge u_n(\tau,0)$ for $n$ large enough (since $\partial_tu_n>0$), whence passing to the limits $n\to+\infty$
yields $\mu\ge u(\tau,0)=\phi_e(\tau+t_e,0)$ and then letting $\tau\to+\infty$ gives $\mu\ge1$. This is a contradiction. Similarly, the sequence $(t_n)_{n\in\N}$ is necessarily bounded from above too. Finally, it is bounded and, up to extraction of a subsequence, it converges to a real number $t_\infty$. Therefore,
$$\mu=\phi_{e_n}(0,0)=u_n(-t_n,0)\to u(-t_\infty,0)=\phi_e(t_e-t_\infty,0)\ \hbox{ as $n\to+\infty$},$$
whence $\phi_e(t_e-t_\infty,0)=\mu=\phi_e(0,0)$ and then $t_\infty=t_e$ by strict monotonicity of $\phi_e$ with respect to its first variable. As a conclusion, by uniqueness of the limit, the whole sequence $(t_n)_{n\in\N}$ converges to $t_e$ and the sequence $(\phi_{e_n})_{n\in\N}$ converges locally uniformly in $\R\times\R^N$ to $u(\cdot-t_e,\cdot)=\phi_e$. 
Lastly, since $U_e(x,z)=\phi_e((x\cdot e-z)/c^*(e),x)$ for all $e\in\Sph$ and $(x,z)\in\R^N\times\R$, the continuity of the map $e\mapsto c^*(e)$ from $\Sph$ to $(0,+\infty)$ and the continuity of the map $e\mapsto\phi_e$ from $\Sph$ to $L^\infty_{loc}(\R\times\R^N)$ (with the normalization $\phi_e(0,0)=\mu)$ entail the continuity of the map $e\mapsto U_e$ from $\Sph$ to $L^\infty_{loc}(\R^N\times\R)$. Furthermore, since every $U_e$ satisfies $U_e(x,z)\to1$ and $0$ as $z\to-\infty$ and $+\infty$ uniformly in $x\in\R^N$ and $U_e$ ranges in $(0,1)$ (by monotonicity in $z$), one concludes that the map $e\mapsto U_e$ is continuous from $\Sph$ to $L^\infty(\R^N\times\R)$. The proof of Proposition~\ref{pro:U} is thereby complete.
\end{proof}


\subsection{Proof of Proposition~\ref{pro:c>0invasion}}\label{sec43}

The proof of Proposition~\ref{pro:c>0invasion} relies on some results of~\cite{DR} and on Lemma~\ref{lem>0bis}.
 
\begin{proof}[Proof of Proposition~$\ref{pro:c>0invasion}$]
Suppose that Hypotheses~\ref{hyp:c*>0} and~\ref{hypbis} (weaker than Hypothesis~\ref{hyp:comb-bi}) hold. Then using the results of~\cite{DR} one infers that the invasion property for~\eqref{general} holds as well. More precisely,~\cite[Theorem~3]{DR} states that the invasion property holds under Hypotheses~\ref{hyp:c*>0} and~\ref{hypbis}, and under the additional assumption that there exists $\theta\in(0,1)$ such that $f(x,s)>0$ for all $x\in\R^N$ and $s\in[\theta,1)$. However, that additional assumption is only used to guarantee the property that, for any half-space $E:=\{x\in\R^N:x\cdot e\le0\}$ (with $e\in\Sph$) and for any solution~$w$ to~\eqref{general} defined in $\R\times E$ and satisfying $\theta\leq w\leq1$ in $\R\times E$, there holds $w(t,x)\to1$ as $x\cdot e\to-\infty$ uniformly in $t\in\R$, see~\cite[Lemma~4]{DR} and~\cite[Lemma~9]{D}. As in the proof of~\eqref{claimu1}, from standard parabolic estimates and the invariance of the equation by translation in time, it is sufficient to show that, for any entire (defined in $\R\times\R^N$) solution $v$ to~\eqref{general},
$$\big(\theta\leq v\leq1\hbox{ in $\R\times\R^N$}\big)\ \Longrightarrow\ \big(v\equiv1\hbox{ in $\R\times\R^N$}\big).$$
But this Liouville-type result holds here, under Hypotheses~\ref{hyp:c*>0} and~\ref{hypbis}, with $\theta:=1-\delta$, by Lemma~\ref{lem>0bis}. The proof of Proposition~\ref{pro:c>0invasion} is thereby complete.
\end{proof}


\subsection{Proof of Proposition~\ref{proinvasion}}\label{sec44}

In order to prove Proposition~\ref{proinvasion}, we construct suitable compactly supported sub-solutions for the associated stationary equations, with large enough maxima.

\begin{proof}[Proof of Proposition~$\ref{proinvasion}$]
Let $g:[0,1]\to\R$ be a $C^1$ function such that $g(0)=g(1)=0$ and satisfying~\eqref{theta} and $g'(1)<0$. Let us extend $g$ in $(-\infty,0)$ so that $g$ is continuously decreasing in $(-\infty,0]$. 

First of all, for all $\eta>0$ small enough, there are $a_\eta<0$ and $b_\eta\in(0,1)$ such that $g(a_\eta)=g(b_\eta)=\eta$ and $0<g(s)<\eta$ for all $s\in(b_\eta,1)$ (namely, $b_\eta=\max\{b<1:g(b)=\eta\}$). Furthermore, $a_\eta\to0$ and $b_\eta\to1$ as $\eta\to0^+$, and one can fix in the sequel $\eta>0$ small enough such that the function $g-\eta$ satisfies~\eqref{theta} with the extremes $0$ and $1$ replaced by $a_\eta$ and $b_\eta$ respectively. In particular, $g-\eta>0$ in an interval $[c_\eta,b_\eta)$, for some $c_\eta\in(0,b_\eta)$.

Now, from the variational methods of~\cite{BL} (see also~\cite[Lemma 3.3]{HN2}), there are then $R_0>0$ and a $C^2(\overline{B_{R_0}})$ function $\varphi:\overline{B_{R_0}}\to[a_\eta,b_\eta)$ solving
$$\left\{\baa{ll}
\Delta\varphi+g(\varphi)-\eta=0 & \hbox{in $\overline{B_{R_0}}$},\vspace{3pt}\\
a_\eta<\varphi<b_\eta & \hbox{in $B_{R_0}$},\vspace{3pt}\\
c_\eta<\displaystyle\max_{\overline{B_{R_0}}}\varphi=\varphi(0)<b_\eta,\vspace{3pt}\\
\varphi=a_\eta & \hbox{on $\partial B_{R_0}$}.\eaa\right.$$
Furthermore, $\varphi$ is radially symmetric and decreasing with respect to $|x|$ in $\overline{B_{R_0}}$~\cite{GNN}. Let  then~$R$ be an orthogonal matrix and $D$ be a positive diagonal matrix, representing some automorphisms $\mathcal{R}$ and $\mathcal{D}$ in the canonical basis of $\R^N$, such that $(A_{ij})_{1\le i,j\le N}=RD^2R^{-1}$. Calling $\mathcal{E}$ the open ellipsoid $\mathcal{R}\circ\mathcal{D}(B_{R_0})$, it follows that the $C^2(\overline{\mathcal{E}})$ function $x\mapsto\psi(x):=\varphi(\mathcal{D}^{-1}\circ\mathcal{R}^{-1}(x))$ solves
$$\left\{\baa{ll}
\displaystyle\sum_{1\le i,j\le N}A_{ij}\partial_{x_ix_j}\psi+g(\psi)=\eta & \hbox{in $\overline{\mathcal{E}}$},\vspace{3pt}\\
a_\eta<\psi<b_\eta & \hbox{in $\mathcal{E}$},\vspace{3pt}\\
c_\eta<\displaystyle\max_{\overline{\mathcal{E}}}\psi=\psi(0)<b_\eta,\vspace{3pt}\\
\psi=a_\eta & \hbox{on $\partial\mathcal{E}$}.\eaa\right.$$
Let us also extend $\psi$ by $a_\eta$ in $\R^N\setminus\overline{\mathcal{E}}$. In particular, $\psi$ is then continuous in $\R^N$ and the set where $\psi$ is positive is bounded and included in $\mathcal{E}$.

Therefore, there is $\epsilon>0$ such that, if the matrix and vector fields $A$ and $q$ satisfy the general assumptions of the paper, together with
$$\max_{1\le i,j\le N}\|a_{ij}-A_{ij}\|_{C^1(\R^N)}+\max_{1\le i\le N}\|q_i\|_{L^\infty(\R^N)}\le\epsilon,$$
then, for all $y\in\R^N$,
\be\label{psisub}
\dv(A(x-y)\nabla\psi)(x)+q(x-y)\cdot\nabla\psi(x)+g(\psi(x))>0\ \hbox{ for all $x\in\mathcal{E}$}.
\ee
We finally claim that, under these conditions and $f(x,s)\ge g(s)$, the invasion property holds for~\eqref{general}. To do so, choose any $\rho>0$ such that
$$\mathcal{E}\subset B_\rho,$$
and any $\theta$ such that
$$\theta\in(\psi(0),b_\eta).$$
In particular, $\theta\in(\varphi(0),b_\eta)\subset(c_\eta,b_\eta)\subset(0,1)$. Consider now any solution $u$ of~\eqref{general} with initial condition $u_0$ such that
$$\theta\,\1_{B_\rho}\le u_0\le 1\ \hbox{ in $\R^N$}.$$
Let also $v$ be the solution of~\eqref{general} with $g$ instead of $f$ and with initial condition $v_0:=\max(\psi,0)$ in $\R^N$. One then has $v_0\le u_0$ in $\R^N$ and it follows from the above properties of $\psi$ and the parabolic maximum principle that
\be\label{inequv}
0\le v(t,x)\le u(t,x)\le1\ \hbox{ for all $t>0$ and $x\in\R^N$}
\ee
and that the function $v$ is increasing with respect to $t>0$. Therefore, from standard parabolic estimates, there is a $C^2(\R^N)$ function $v_\infty:\R^N\to[0,1]$ such that $v(t,\cdot)\to v_\infty$ locally uniformly in $\R^N$ as $t\to+\infty$, and
\be\label{vinfty}
\dv(A(x)\nabla v_\infty)+q(x)\cdot\nabla v_\infty+g(v_\infty)=0\ \hbox{ in $\R^N$},
\ee
together with $v_\infty>v_0=\max(\psi,0)$ in $\R^N$. Since $v_0$ is compactly supported, there is $r_0>0$ such that $v_\infty>v_0(\cdot+\xi)$ in $\R^N$ for all $\xi\in B_{r_0}$. Let then
$$r^*:=\sup\big\{r>0:v_\infty>v_0(\cdot+\xi)\hbox{ in $B_r$ for all $\xi\in B_r$}\}.$$
One has $0<r_0\le r^*\le+\infty$. Assume by way of contradiction that $r^*<+\infty$. Then there is $y\in\R^N$ such that $v_\infty\ge v_0(\cdot+y)=\max(\psi(\cdot+y),0)$ in $\R^N$ with equality at a point $x^*$. Since $v_\infty>0$ in $\R^N$ and $\psi=a_\eta<0$ in $\R^N\setminus\mathcal{E}$, one has $x^*+y\in\mathcal{E}$, that is, $x^*\in\mathcal{E}-y$, and $v_\infty(x^*)=\psi(x^*+y)$. On the other hand, by~\eqref{psisub}, the function $x\mapsto\tilde{\psi}(x):=\psi(x+y)$ satisfies
$$\dv(A(x)\nabla\tilde{\psi})+q(x)\cdot\nabla\tilde{\psi}+g(\tilde{\psi})>0\ \hbox{ in }\mathcal{E}-y.$$
This, together with $v_\infty\ge\tilde{\psi}$ in $\mathcal{E}-y$, $v_\infty(x^*)=\tilde{\psi}(x^*)$ and~\eqref{vinfty} evaluated at $x^*\in\mathcal{E}-y$, leads to a contradiction. Therefore, $r^*=+\infty$ and $1\ge v_\infty>v_0(\cdot+\xi)$ in $\R^N$ for all $\xi\in\R^N$. In particular,
$$\max(\psi(0),0)=v_0(0)<v_\infty(x)\le1\ \hbox{ for all $x\in\R^N$},$$
whence $\zeta(t)<v_\infty(x)\le1$ for all $t\ge0$ and $x\in\R^N$, where $\zeta$ solves $\zeta'(t)=g(\zeta(t))$ for all $t\ge0$ and $\zeta(0)=v_0(0)=\max(\psi(0),0)$. Since $\psi(0)\in(c_\eta,b_\eta)$ while $g\ge\eta>0$ in $[c_\eta,b_\eta]$ and $g>0$ in $(b_\eta,1)$, one infers that $\zeta(+\infty)=1$ and that $v_\infty(x)=1$ for all $x\in\R^N$. As a conclusion, remembering~\eqref{inequv} and the definition of $v_\infty$, it follows that
$$u(t,x)\to1\ \hbox{ as $t\to+\infty$ locally uniformly in $\R^N$},$$
that is,~\eqref{invading} holds. The proof of Proposition~\ref{proinvasion} is thereby complete.
\end{proof}


\subsection{Proof of Proposition~\ref{pro:Wchar}}\label{sec45}

In order to prove Proposition~$\ref{pro:Wchar}$, we need to show the characterization~\eqref{Wint} for the \ais\ $\W$, as well as the two estimates~\eqref{inner-cone}-\eqref{outer-cone}. Notice indeed that those estimates imply that $\W$ is star-shaped with respect to the origin, and that its boundary~$\partial\W$, if non-empty, is Lipschitz-regular. We start with the proof of~\eqref{Wint}.

\begin{proof}[Proof of~\eqref{Wint}] Define $\W'$ as the right-hand side in~\eqref{Wint}. Assume that $u$ admits an  \ais~$\W=\inter\ol{\W}$. As an immediate consequence of the openness of $\W$ and the first property in~\eqref{ass-cpt} (applied with $C:=\overline{B_r(x)}$ for any $x\in\W$ and $r>0$ such that $\overline{B_r(x)}\subset\W$), we get $\W\subset\W'$. Conversely, the second property in~\eqref{ass-cpt} satisfied by $\W$ yields $\W'\subset\ol{\W}$, and being $\W'$ open, we derive $\W'\subset\inter\ol{\W}$, and the latter set coincides with $\W$ by Definition~\ref{def:W}. Finally, $\W=\W'$.
\end{proof}

The proof of~\eqref{inner-cone}-\eqref{outer-cone} relies on the following result.

\begin{lemma}\label{lem:Wlip}
Under the invasion property~\eqref{invading}, if a solution $u$ to~\eqref{general} admits an \ais\ $\W$, then
\Fi{BinW}
\forall\,\xi\in\W,\ \ \forall\,\lambda\in[0,1),\quad B_{(1-\lambda)\gamma}(\lambda\xi)\subset\W,
\Ff
where $\gamma>0$ is such that \eqref{spreading} holds for the invading solution with initial datum $\theta\1_{B_\rho}$.
\end{lemma}

\begin{proof}
We have proved above that the \ais\ $\W$ of~$u$ is given by~\eqref{Wint}. We now prove that such set satisfies \eqref{BinW} with $\gamma>0$ being such that~\eqref{spreading} holds for the invading solution $v$ of~\eqref{general} with initial datum $\theta\1_{B_\rho}$, where $\theta\in(0,1)$ and $\rho>0$ are given by the invasion property.
	
Fix $\xi\in\W$. Being $\W$ open, there exists $r>0$ such that $\ol{B_r(\xi)}\subset\W$. Thus, by Definition \ref{def:W}, there holds
$$\lim_{t\to+\infty}\,\Big(\min_{x\in \ol{B_r(\xi)}}u(t, tx)\Big)=1.$$
In particular, there exists $T>0$ such that 
$$\forall\,\tau\geq T,\ \ \forall\,x\in B_{\rho+\sqrt{N}}(\tau\xi),\qquad u(\tau,x)\geq \theta.$$
For $\tau\geq T$, let $h_\tau\in\Z^N$ be such that $h_\tau-\tau\xi\in[0,1]^N$, whence $u(\tau,\.+h_\tau)\geq \theta\1_{B_{\rho}}$. Thanks to the invariance of the equation by temporal translations and by spatial translations by~$\Z^N$, the function $u(\.+\tau+,\.+h_\tau)$ is still a solution to~\eqref{general} and thus, since the solution~$v$ of~\eqref{general} with initial datum $\theta\1_{B_\rho}$ satisfies~\eqref{spreading} for some $\gamma>0$, we infer by comparison~that 
$$\forall\,\tau\geq T,\qquad\min_{x\in\overline{B_{\gamma t}}} u(t+\tau,x+h_\tau)\geq\min_{x\in\overline{B_{\gamma t}}} v(t,x)\to1 \as t\to+\infty.$$

We rewrite the above property as follows:
\Fi{tauT}
\min_{y\in\overline{B_{\gamma(s-\tau)}(h_\tau)}} u(s,y)\to1 \as s-\tau\to+\infty\hbox{ with }\tau\ge T.
\Ff
For fixed $\tau\ge T$ and $x\in B_\gamma$, there holds
$$\frac{|sx-h_\tau|}{\gamma(s-\tau)}\to \frac{|x|}\gamma<1 \as s\to+\infty,$$
hence \eqref{tauT} yields
$$B_\gamma\subset \{x\!\in\!\R^N\,:\, {u(s, sx)\to1} \as s\to+\infty\}.$$ 
It then follows from~\eqref{Wint} that $B_\gamma\subset\W$, that is, the inclusion~\eqref{BinW} holds for $\lambda=0$.
	
Fix now $\lambda\in(0,1)$ and $x\in B_{(1-\lambda)\gamma}(\lambda\xi)$. Taking any $s\ge T/\lambda$ and $\tau:=\lambda s\ge T$, we infer
$$\frac{|sx-h_\tau|}{\gamma(s-\tau)}=\frac{|sx-h_{\lambda s}|}{\gamma(1-\lambda) s}\leq\frac{|x-\lambda\xi|}{\gamma(1-\lambda)}+\frac{|\lambda s\xi-h_{\lambda s}|}{\gamma(1-\lambda)s},$$
whence, being $h_{\lambda s}-\lambda s\xi\in[0,1]^N$ bounded,
$$\frac{|sx-h_\tau|}{\gamma(s-\tau)}\to\frac{|x-\lambda\xi|}{\gamma(1-\lambda)}<1\as s\to+\infty.$$
Observe that $s\to+\infty$ is equivalent to $s-\tau=(1-\lambda)s\to+\infty$. As a consequence, we deduce from~\eqref{tauT} that $u(s, sx)\to1$ as $s\to+\infty$. Owing to~\eqref{Wint}, we have thereby shown that $B_{(1-\lambda)\gamma}(\lambda\xi)\subset\W$. The proof of Lemma~\ref{lem:Wlip} is complete.
\end{proof}

\begin{proof}[End of the proof Proposition~$\ref{pro:Wchar}$]
It remains to show~\eqref{inner-cone}-\eqref{outer-cone}. We assume that the invasion property holds. Let ${z}\in\partial\W$. Firstly, for any $\lambda\in[0,1)$ and any point $x\in B_{(1-\lambda)\gamma}(\lambda{z})$, we can find $\xi\in\W$ such that $x\in B_{(1-\lambda)\gamma}(\lambda\xi)$, thus $x\in\W$ by Lemma~\ref{lem:Wlip}. This is the inclusion \eqref{inner-cone}.

For the other inclusion, consider again ${z}\in\partial\W$, but this time $\lambda>1$. Assume by contradiction that there exists $\xi\in B_{(\lambda-1)\gamma}(\lambda{z})\cap\W$. We see that
$$|{z}-\frac1\lambda\xi|=\frac1\lambda|\lambda{z}-\xi|<\frac{(\lambda-1)\gamma}\lambda=\Big(1-\frac1\lambda\Big)\gamma.$$
It then follows from Lemma \ref{lem:Wlip} that ${z}\in\W$: a contradiction with the fact that $\W$ is~open. This is the inclusion~\eqref{outer-cone}.
\end{proof}


\subsection{Proof of Proposition~\ref{pro:level}}\label{sec46}

Consider a solution $u$ to~\eqref{general} with a compactly supported initial condition $u_0:\R^N\to[0,1]$, and assume that $u$ has an \ais\ $\W$. Let
$$L:=\sup_{(x,s)\in\R^N\times(0,1]}\frac{f(x,s)}{s}.$$
From the standing assumptions on $f$, $L$ is a real number. From the maximum principle, there holds, for all $t>0$ and $x\in\R^N$,
$$0\le u(t,x)\le e^{Lt}v(t,x),$$
where $v$ is the solution of the equation
$$\partial_t v=\dv(A(x)\nabla v)+q(x)\.\nabla v,\quad t>0,\ x\in\R^N,$$
with initial condition $u_0$. In other words,
$$v(t,x)=\int_{\R^N}p(t,x,y)\,u_0(y)\,dy,$$
where $p$ is the fundamental solution associated to the above equation. From~\cite{Fr}, there is a constant $\gamma>0$ large enough such that the kernel $p$ satisfies the Gaussian estimates
$$\forall\,t>0,\ \forall\,(x,y)\in\R^N\times\R^N,\ \ \ 0\le p(t,x,y)\le \frac{\gamma\,e^{-|x-y|^2/(\gamma t)}}{t^{N/2}}.$$
Since $u_0:\R^N\to[0,1]$ is compactly supported, there is then a constant $\Gamma>0$ such that
\be\label{convGamma}
\sup_{|x|\ge\Gamma t}u(t,x)\to0\ \hbox{ as }t\to+\infty.
\ee

Remember that the \ais\ $\W$ is given by Definition~\ref{def:W}. Assume in this paragraph that $\W=\emptyset$. Then $\max_{|x|\le\Gamma t}u(t,x)\to0$ as $t\to+\infty$ by choosing $C:=\overline{B_\Gamma}$ in the second assertion of~\eqref{ass-cpt}. Therefore, together with~\eqref{convGamma}, $\|u(t,\cdot)\|_{L^\infty(\R^N)}\to0$ as $t\to+\infty$, and, for any $\lambda\in(0,1)$, the upper level sets $E_\lambda(t)=\{x\in\R^N:u(t,x)>\lambda\}$ are empty for all $t>0$ large enough, whence~\eqref{ElambdaW} holds trivially in this case.

Assume in the rest of the proof that $\W\neq\emptyset$, and observe that $\W\subset B_\Gamma$ by Definition~\ref{def:W} and~\eqref{convGamma}. Consider any $\lambda\in(0,1)$. Let first $r>0$ be arbitrary and denote
$$\left\{\baa{rcl}
\mathcal{I}_r & := & \big\{x\in\R^N:d(x,\R^N\setminus\W)\ge r\big\},\vspace{3pt}\\
\mathcal{E}_r & := & \big\{x\in\R^N:d(x,\W)\ge r\big\},\eaa\right.$$
where $d(x,A):=\inf_{y\in A}|x-y|$ for any $A\subset\R^N$. Notice that the ``exterior" set $\mathcal{E}_r$ is not empty since~$\W$ is bounded, and that the ``interior" set $\mathcal{I}_r$ is not empty provided that $r>0$ is small enough, since~$\W$ is open and not empty. On the one hand, the compact set $\mathcal{I}_r$ is included in $\W$ and, if it is not empty, there holds $\min_{x\in t\mathcal{I}_r}u(t,x)\to1$ as $t\to+\infty$ by~\eqref{ass-cpt}, whence $\mathcal{I}_r\subset t^{-1}E_\lambda(t)$ for all $t$ large enough. On the other hand, the compact set $\mathcal{E}_r\cap\overline{B_\Gamma}$ is included in $\R^N\setminus\overline{\W}$ and, if it is not empty, there holds $\max_{x\in t(\mathcal{E}_r\cap\overline{B_\Gamma})}u(t,x)\to0$ as $t\to+\infty$ by~\eqref{ass-cpt}. Together with~\eqref{convGamma}, this implies that $\sup_{x\in t\mathcal{E}_r}u(t,x)\to0$ as $t\to+\infty$, whence $t^{-1}E_\lambda(t)\subset\R^N\setminus\mathcal{E}_r$ for all $t$ large enough. To sum up,
\be\label{ErFr}
\mathcal{I}_r\subset t^{-1}E_\lambda(t)\subset\R^N\setminus\mathcal{E}_r
\ee
for all $t$ large enough. 

Consider finally any $\epsilon>0$. By taking $r\in(0,\epsilon]$ small enough such that $\mathcal{I}_r\neq\emptyset$, it follows from the first inclusion of~\eqref{ErFr} that $t^{-1}E_\lambda(t)\neq\emptyset$ for all $t$ large enough, and from the second one that
\be\label{hausdorff1}
\sup_{x\in t^{-1}E_\lambda(t)}d(x,\W)\le r\le\epsilon
\ee
for all $t$ large enough. For any $x\in\overline{\W}$, there exists $y\in\W$ such that $|y-x|<\epsilon$, and thus, since $y\in\mathcal{I}_r$ for some~$r>0$ (depending on $y$) small enough, we deduce that 
\[\overline{\W}\subset\bigcup_{r>0}(\mathcal{I}_r+B_\epsilon).\]
Being $\overline{\W}$ compact, we can extract a finite covering $\mathcal{I}_{r_1}+B_\epsilon,\dots,\mathcal{I}_{r_p}+B_\epsilon$ of $\overline{\W}$. Since by~\eqref{ErFr} each one of the $\mathcal{I}_{r_j}$ is contained in $t^{-1}E_\lambda(t)$ for $t$ larger than some $T_j>0$, we infer that $\overline{\W}\subset (t^{-1}E_\lambda(t)+B_\epsilon)$ for $t$ larger than $T:=\max\{T_1,\dots,T_p\}$, whence
\be\label{hausdorff2}
\sup_{x\in\W}d(x,t^{-1}E_\lambda(t))\leq\epsilon
\ee
for $t$ larger than $T$. By putting together~\eqref{hausdorff1} and~\eqref{hausdorff2}, it follows from the arbitrariness of $\epsilon>0$ that the sets $t^{-1}E_\lambda(t)$ converge to $\W$ as $t\to+\infty$ in the sense of the Hausdorff distance. The proof of Proposition~\ref{pro:level} is thereby complete.\hfill$\Box$




\begin{thebibliography}{1}
{\footnotesize{		
		\bibitem{AG}
		M. Alfaro and T. Giletti.
		\newblock Varying the direction of propagation in reaction-diffusion equations in periodic media.
		\newblock {\em Netw. Heterog. Media}, 11:369--393, 2016.

		\bibitem{AW}
		D.~G. Aronson and H.~F. Weinberger.
		\newblock Multidimensional nonlinear diffusion arising in population genetics.
		\newblock {\em Adv. Math.}, 30:33--76, 1978.
		
		\bibitem{BH1}
		H. Berestycki and F. Hamel.
		\newblock Front propagation in periodic excitable media.
		\newblock {\em Comm. Pure Appl. Math.}, 55:949--1032, 2002.

		\bibitem{BH2}
		H.~Berestycki and F.~Hamel.
		\newblock Generalized travelling waves for reaction-diffusion equations.
		\newblock In {\em Perspectives in nonlinear partial differential equations},
		volume 446 of {\em Contemp. Math.}, pages 101--123, Amer. Math. Soc., Providence, RI, 2007.

		\bibitem{BH3}
		H.~Berestycki and F.~Hamel.
		\newblock Generalized transition waves and their properties.
		\newblock {\em Comm. Pure Appl. Math.}, 65:592--648, 2012.
		
		\bibitem{BHN} H. Berestycki, F. Hamel, and G. Nadin.
		\newblock Asymptotic spreading in heterogeneous diffusive media.
		\newblock {\em J.~Funct. Anal.}, 255:2146--2189, 2008.

		\bibitem{BHN1}
		H.~Berestycki, F.~Hamel, and N.~Nadirashvili.
		\newblock The speed of propagation for {KPP} type problems. {I}. {P}eriodic framework.
		\newblock {\em J. Europ. Math. Soc.}, 7:173--213, 2005.

		\bibitem{BL}
		H. Berestycki and P.-L. Lions.
		\newblock Une m\'ethode locale pour l'existence de solutions positives de 
		probl\`emes semilin\'eaires elliptiques dans $\R^N$.
		\newblock {\em J.~Anal. Math.}, 38:144--187, 1980.

		\bibitem{BN}
		H. Berestycki and G. Nadin.
		\newblock Asymptotic spreading for general heterogeneous Fisher-KPP type equations.
		\newblock {\em Memoirs Amer. Math. Soc.}, 280(1381), 2022.
	
		\bibitem{B}
		M. Bramson.
		\newblock Convergence of solutions of the Kolmogorov equation to travelling
		waves, {\em Memoirs Amer. Math. Soc.}, 44, 1983.

		\bibitem{DG} 
		W. Ding and T. Giletti.
		\newblock Admissible speeds in spatially periodic bistable reaction-diffusion equations.
		\newblock {\em Adv. Math.}, 389:107889, 2021.

		\bibitem{DHZ1}
		W. Ding, F. Hamel, and X.-Q. Zhao.
		\newblock Propagation phenomena for periodic bistable reaction-diffusion equations.
		\newblock {\em Calc. Var. Part. Diff. Equations}, 54:2517--2551, 2015.

		\bibitem{DHZ2}
		W. Ding, F. Hamel, and X.-Q. Zhao.
		\newblock Bistable pulsating fronts for reaction-diffusion equations in a periodic habitat.
		\newblock {\em Indiana Univ. Math.~J.}, 66:1189--1265, 2017.

		\bibitem{DLL} 
		W. Ding, Z. Liang, and W. Liu.
		\newblock Continuity of pulsating wave speeds for bistable reaction-diffusion equations in 
		spatially periodic media.
		\newblock {\em J.~Math. Anal. Appl.}, 519:126794, 2023.

		\bibitem{DS} 
		F. Dkhil and A. Stevens.
		\newblock Traveling wave speeds in rapidly oscillating media.
		\newblock {\em Disc. Cont. Dyn. Systems~A}, 25:89--108, 2009.

		\bibitem{DLBL} 
		J. Dowdall, V. LeBlanc, and F. Lutscher.
		\newblock Invasion pinning in a periodically fragmented habitat.
		\newblock {\em J. Math. Biol.}, 77:55--78, 2018.

		\bibitem{DM1} 
		Y. Du and H. Matano.
		\newblock Convergence and sharp thresholds for propagation in nonlinear diffusion problems.
		\newblock {\em J.~Europ. Math. Soc.}, 12:279--312, 2010.
		
		\bibitem{DM2}
		Y. Du and H. Matano.
		\newblock Radial terrace solutions and propagation profile of multistable 
		reaction-diffusion equations over $\R^N$.
		\newblock {\em https://arxiv.org/pdf/1711.00952.pdf}.
		
		\bibitem{DP}
		Y. Du and P. Pol\'a\v{c}ik.
		\newblock Locally uniform convergence to an equilibrium for nonlinear parabolic equations 
		on $\R^N$.
		\newblock {\em Indiana Univ. Math.~J.}, 64:787--824, 2015.

		\bibitem{D}
		R.~Ducasse.
		\newblock Propagation properties of reaction-diffusion equations in periodic domains.
		\newblock {\em Analysis \& PDE}, 13:2259--2288, 2020.

		\bibitem{DR}
		R.~Ducasse and L.~Rossi.
		\newblock Blocking and invasion for reaction-diffusion equations in periodic media.
		\newblock {\em Calc. Var. Part. Diff. Equations}, 57:Art. 142, 2018.
		
		\bibitem{D1}
		A. Ducrot.
		\newblock On the large time behaviour of the multi-dimensional Fisher-KPP equation 
		with compactly supported initial data.
		\newblock {\em Nonlinearity}, 28:1043--1076, 2015.

		\bibitem{D2}
		A. Ducrot.
		\newblock A multi-dimensional bistable nonlinear diffusion equation in periodic medium.
		\newblock {\em Math. Ann.}, 366:783--818, 2016.
		
		\bibitem{DGM}
		A. Ducrot, T. Giletti, and H. Matano.
		\newblock Existence and convergence to a propagating terrace in one-dimensional 
		reaction-diffusion equations.
		\newblock {\em Trans. Amer. Math. Soc.}, 366:5541--5566, 2014.
		
		\bibitem{E1}
		S. Eberle.
		\newblock Front blocking in the presence of gradient drift.
		\newblock {\em J. Diff. Equations}, 267:7154--7166, 2019.

		\bibitem{E2}
		S. Eberle.
		\newblock Front blocking versus propagation in the presence of drift term in the direction 
		of propagation.
		\newblock {\em Nonlin. Anal.}, 197:111836, 2020.

		\bibitem{FZ}
		J. Fang and X.-Q. Zhao.
		\newblock Bistable traveling waves for monotone semiflows with applications.
		\newblock {\em J.~Europ. Math. Soc.}, 17:2243--2288, 2015.

		\bibitem{FM}
		P. C. Fife and J. B. McLeod.
		\newblock The approach of solutions of non-linear diffusion equations to traveling front solutions.
		\newblock {\em Arch. Ration. Mech. Anal.}, 65:335--361, 1977.
		
		\bibitem{F} R.~A. Fisher.
		\newblock The advance of advantageous genes.
		\newblock {\em Ann. Eugenics}, 7:335--369, 1937.
		
		\bibitem{FG} 
		M. Freidlin and J. G\"artner.
		\newblock On the propagation of concentration waves in periodic and random media.
		\newblock {\em Sov. Math. Dokl.}, 20:1282--1286, 1979.
		
		\bibitem{Fr}
		A. Friedman.
		\newblock {\em Partial Differential Equations of Parabolic Type}.
		\newblock Prentice-Hall, Englewood Cliffs, New Jersey, 1964.
		
		\bibitem{G1}
		J. G\"artner.
		\newblock Location of wave fronts for the multi-dimensional KPP equation and Brownian first 
		exit densities.
		\newblock {\em Math. Nachr.}, 105:317--351, 1982.
		
		\bibitem{GNN}
		B. Gidas, W.-M. Ni and L. Nirenberg.
		\newblock Symmetry and related properties via the maximum principle.
		\newblock {\em Comm. Math. Phys.}, 68:209--243, 1979.
			
		\bibitem{GR}
		T.~Giletti and L.~Rossi.
		\newblock Pulsating solutions for multidimensional bistable and multistable equations.
		\newblock {\em Math. Ann.}, 378:1555--1611, 2020.

		\bibitem{GR2}
		T.~Giletti and L.~Rossi.
		\newblock Stability of propagating terraces in spatially periodic multistable equations in $\R^N$.
		\newblock {\em https://arxiv.org/abs/2503.07128}.
				
		\bibitem{G} 
		H. Guo.
		\newblock Propagating speeds of bistable transition fronts in spatially periodic media.
		\newblock {\em Calc. Var. Part. Diff. Equations}, 57:47, 2018.

		\bibitem{GHR2} 
		H. Guo, F. Hamel and L. Rossi.
		\newblock Reaction-diffusion equations in periodic media: spreading speeds 
		and spreading sets.
		\newblock {\em In preparation}.
		
		\bibitem{GW}
		H. Guo and H. Wang.
		\newblock Curved fronts of bistable reaction-diffusion equations in spatially periodic 
		media: $N\ge2$.
		\newblock {\em https://arxiv.org/abs/2501.03815}.
				
		\bibitem{H} 
		F. Hamel.
		\newblock Qualitative properties of monostable pulsating fronts: exponential decay 
		and monotonicity.
		\newblock {\em J. Math. Pures Appl.}, 89:355--399, 2008.

		\bibitem{HFR} 
		F. Hamel, J. Fayard and L. Roques.
		\newblock Spreading speeds in slowly oscillating environments.
		\newblock {\em Bull. Math. Biol.}, 72:1166--1191, 2010.

		\bibitem{HMR1}
		F. Hamel, R.~Monneau, and J.-M.~Roquejoffre.
		\newblock Existence and qualitative pro\-perties of multidimensional conical bistable fronts.
		\newblock {\em Disc. Cont. Dyn. Syst. A}, 13:1069--1096, 2005.

		\bibitem{HMR2}
		F. Hamel, R.~Monneau, and J.-M.~Roquejoffre.
		\newblock Asymptotic properties and classification of bistable fronts with Lipschitz level sets.
		\newblock {\em Disc. Cont. Dyn. Syst.~A}, 14:75--92, 2006.
		
		\bibitem{HN}
		F. Hamel and G. Nadin.
		\newblock Spreading properties and complex dynamics for monostable 
		reaction-diffusion equations.
		\newblock {\em Comm. Part. Diff. Equations}, 37:511--537, 2012.
		
		\bibitem{HN2}
		F. Hamel and H. Ninomiya.
		\newblock Localized and expanding entire solutions of reaction-diffusion equations.
		\newblock {\em J. Dyn. Diff. Equations}, 34:2937--2974, 2022.
		
		\bibitem{HNRR1}
		F. Hamel, J. Nolen, J.-M. Roquejoffre, and L. Ryzhik.
		\newblock A short proof of the logarithmic Bramson correction in Fisher-KPP equations.
		\newblock {\em Netw. Heterog. Media}, 8:275--289, 2013.

		\bibitem{HNRR2}
		F. Hamel, J. Nolen, J.-M. Roquejoffre, and L. Ryzhik.
		\newblock The logarithmic delay of KPP fronts in a periodic medium.
		\newblock {\em J.~Europ. Math. Soc.}, 18:465--505, 2016.

		\bibitem{HR2}
		F. Hamel and L. Rossi.
		\newblock Asymptotic one-dimensional symmetry for the Fisher-KPP equation.
		\newblock {\em J. Europ. Math. Soc.}, forthcoming.

		\bibitem{HR3}
		F. Hamel and L. Rossi.
		\newblock Spreading, flattening and logarithmic lag for reaction-diffusion equations in $\R^N$: 
		old and new results.
		\newblock {\em EMS Surv. Math. Sci.}, forthcoming.

		\bibitem{HR1}
		F. Hamel and L. Rossi.
		\newblock Spreading speeds and spreading sets of reaction-diffusion equations.
		\newblock {\em https://arxiv.org/abs/2105.08344}.
		
		\bibitem{HPS}
		S. Heinze, G. Papanicolaou, and A. Stevens.
		\newblock Variational principles for propagation speeds in inhomogeneous media.
		\newblock {\em SIAM J.~Appl. Math.}, 62:129--148, 2001.

		\bibitem{K}
		T. Kato.
		\newblock {\em Perturbation theory for linear operators}.
		\newblock Berlin, Germany, Springer-Verlag, 1966.

		\bibitem{KPP}
		A.~N. Kolmogorov, I.~G. Petrovsky, and N.~S. Piskunov.
		\newblock \'Etude de l'\'equation de la diffusion avec croissance de la quantit\'e de 
		mati\`ere et son application \`a un probl\`eme biologique.
		\newblock {\em Bull. Univ. \'Etat Moscou, S\'er. Intern. A}, 1:1--26, 1937.
		
		\bibitem{L}
		K.-S. Lau.
		\newblock On the nonlinear diffusion equation of Kolmogorov, Petrovsky, and Piscounov.
		\newblock {\em J. Diff. Equations}, 59:44--70, 1985.

		\bibitem{LZ1}
		X. Liang and X.-Q. Zhao.
		\newblock Asymptotic speeds of spread and traveling waves for monotone semiflows 
		with applications.
		\newblock {\em Comm. Pure Appl. Math.}, 60:1--40, 2007. 
		
		\bibitem{LZ2}
		X. Liang and X.-Q. Zhao.
		\newblock Spreading speeds and traveling waves for abstract monostable evolution systems.
		\newblock {\em J.~Func. Anal.}, 259:857--903, 2010. 

		\bibitem{MN}
		H. Matano and M. Nara.
		\newblock Large time behavior of disturbed planar fronts in the {Allen-Cahn} equation.
		\newblock {\em J. Diff. Equations}, 251:3522--3557, 2011.
		
		\bibitem{MNT}
		H. Matano, M. Nara, and M. Taniguchi.
		\newblock Stability of planar waves in the {Allen-Cahn} equation.
		\newblock {\em Comm. Part. Diff. Equations}, 34:976--1002, 2009.		

		\bibitem{MZ1} 
		C. B. Muratov and X. Zhong. 
		\newblock Threshold phenomena for symmetric decreasing solutions of 
		reaction-diffusion equations.
		\newblock {\em Nonlin. Diff. Equations Appl.}, 20:1519--1552, 2013.
		
		\bibitem{MZ2} 
		C. B. Muratov and X. Zhong.
		\newblock Threshold phenomena for symmetric-decreasing radial solutions of 
		reaction-diffusion equations.
		\newblock {\em Disc. Cont. Dyn. Syst. A}, 37:915--944, 2017.
			
		\bibitem{N}
		G. Nadin,
		\newblock Critical travelling waves for general heterogeneous one-dimensional reaction-diffusion 
		equations.
		\newblock {\em Ann. Inst. H.~Poincar\'e, Analyse Non Lin\'eaire}, 32:841--873, 2015.

		\bibitem{NT}
		H. Ninomiya and M. Taniguchi,
		\newblock Existence and global stability of traveling curved fronts in the Allen-Cahn equations.
		\newblock {\em J. Diff. Equations}, 213:204--233, 2005.

		\bibitem{NRR}
		J. Nolen, J.-M. Roquejoffre, and L. Ryzhik.
		\newblock Convergence to a single wave in the Fisher-KPP equation.
		\newblock {\em Chinese Ann. Math. Ser. B (special issue in honor of H. Brezis)}, 
		38:629--646, 2017.
		
		\bibitem{PX}
		G. Papanicolaou and J.X. Xin.
		\newblock Reaction-diffusion fronts in periodically layered media.
		\newblock {\em J.~Stat. Phys.}, 63:915--931, 1991.

		\bibitem{P1} P. Pol\'a{\v{c}}ik.
		\newblock Threshold solutions and sharp transitions for nonautonomous 
		parabolic equations on $\R^N$.
		\newblock {\em Arch. Ration. Mech. Anal.}, 199:69--97, 2011.
		
		\bibitem{P3}
		P.~Pol{\'a}{\v{c}}ik.
		\newblock Planar propagating terraces and the asymptotic one-dimensional
		symmetry of solutions of semilinear parabolic equations.
		\newblock {\em SIAM J. Math. Anal.}, 49:3716--3740, 2017.

		\bibitem{P4}
		P.~Pol\'{a}\v{c}ik.
		\newblock Propagating terraces and the dynamics of front-like solutions of 
		reaction-diffusion equations on {$\mathbb{R}$}.
		\newblock {\em Memoirs Amer. Math. Soc.}, 264(1278):v+87, 2020.
				
		\bibitem{RRR}
		J.-M. Roquejoffre, L. Rossi, and V. Roussier-Michon.
		\newblock Sharp large time behaviour in $N$-dimensional Fisher-KPP equations.
		\newblock {\em Disc. Cont. Dyn. Syst. A}, 39:7265--7290, 2019.

		\bibitem{RR1}
		J.-M. Roquejoffre and V. Roussier-Michon.
		\newblock Nontrivial large-time behaviour in bistable reaction-diffusion equations.
		\newblock {\em Ann. Mat. Pura Appl.}, 188:207--233, 2009.

		\bibitem{R1}
		L.~Rossi.
		\newblock The {F}reidlin-{G}{\"a}rtner formula for general reaction terms.
		\newblock {\em Adv. Math.}, 317:267--298, 2017.
	
		\bibitem{S}
		B. Shabani.
		\newblock Logarithmic Bramson correction for multi-dimensional periodic Fisher-KPP equations.
		\newblock {\em https://arxiv.org/abs/1910.08178}.

		\bibitem{U1} 
		K. Uchiyama.
		\newblock The behavior of solutions of some semilinear diffusion equation for large time.
		\newblock {\em J.~Math. Kyoto Univ.}, 18:453--508, 1978.

		\bibitem{U2} 
		K. Uchiyama.
		\newblock Asymptotic behavior of solutions of reaction-diffusion equations with varying 
		drift coefficients.
		\newblock {\em Arch. Ration. Mech. Anal.}, 90:291--311, 1985.
		
		\bibitem{VV}
		S. Vakulenko and V.~A.~Volpert.
		\newblock Generalized travelling waves for perturbed monotone reaction-diffusion systems.
		\newblock {\em Nonlinear Anal. Theory Meth. Appl.}, 46:755--776, 2001.

		\bibitem{W}
		H. F. Weinberger.
		\newblock On spreading speeds and traveling waves for growth and migration in periodic habitat.
		\newblock {\em J.~Math. Biol.}, 45:511--548, 2002.
		
		\bibitem{X1}
		X. Xin.
		\newblock Existence and uniqueness of travelling waves in a reaction-diffusion equation 
		with combustion nonlinearity.
		\newblock {\em Indiana Univ. Math. J.}, 40:985--1008, 1991.

		\bibitem{X2}
		J. Xin.
		\newblock Existence and stability of travelling waves in periodic media governed by a 
		bistable nonlinearity.
		\newblock {\em J. Dyn. Diff. Equations}, 3:541--573, 1991.

		\bibitem{X22}
		J. Xin.
		\newblock Existence of planar flame fronts in convective-diffusive periodic media.
		\newblock {\em Arch. Ration. Mech. Anal.}, 121:205--233, 1992.

		\bibitem{X3}
		J. Xin.
		\newblock Existence and nonexistence of traveling waves and reaction-diffusion front propagation 
		in periodic media.
		\newblock {\em J.~Stat. Phys.}, 73:893--926, 1993.

		\bibitem{X4}
		J. Xin.
		\newblock Analysis and modeling of front propagation in heterogeneous media.
		\newblock {\em SIAM Review}, 42:161--230, 2000.

		\bibitem{XZ}
		J. X. Xin and J. Zhu.
		\newblock Quenching and propagation of bistable reaction-diffusion fronts in multidimensional 
		periodic media.
		\newblock {\em Physica~D Nonlinear Phenomena}, 81:94--110, 1995.

		\bibitem{Y}
		E. Yanagida.
		\newblock Irregular behavior of solutions for Fisher's equation.
		\newblock {\em J.~Dyn. Diff. Equations}, 19:895--914, 2007.

		\bibitem{Zh}
		X.-Q. Zhao.
		\newblock {\em Dynamical Systems in Population Biology}.
		\newblock Springer-Verlag, New York, 2003.

		\bibitem{Z1}
		A. Zlato{\v{s}}.
		\newblock Sharp transition between extinction and propagation of reaction.
		\newblock {\em J.~Amer. Math. Soc.}, 19:251--263, 2006.

		\bibitem{Z2}
		A. Zlato{\v{s}}.
		\newblock Existence and non-existence of transition fronts for bistable and ignition reactions.
		\newblock {\em Ann. Inst. H.~Poincar\'e}, 34:1687--1705, 2017.

}}						
\end{thebibliography}
\end{document}